\documentclass[12pt]{amsart}

\usepackage{newtxtext}
\usepackage{newtxmath}

\usepackage{latexsym,amscd, graphicx, color, amsthm, bm, amsmath, cancel, enumitem,young,chessboard,mathtools}  
\usepackage{tikz}
\usepackage[margin=1in]{geometry}
\usepackage{hyperref}
\usepackage[all,cmtip]{xy}

\numberwithin{equation}{section}

\newtheorem{theorem}{Theorem}[section]
\newtheorem{proposition}[theorem]{Proposition}
\newtheorem{corollary}[theorem]{Corollary}
\newtheorem{lemma}[theorem]{Lemma}
\newtheorem{conjecture}[theorem]{Conjecture}
\newtheorem{observation}[theorem]{Observation}

\newtheorem{example}[theorem]{Example}

\newtheorem{definition}[theorem]{Definition}

\theoremstyle{definition}

\newcommand{\Hilb}{{\mathrm{Hilb}}}

\newcommand{\low}{{\mathrm{low}}}

\newcommand{\symm}{{\mathfrak{S}}}

\newcommand{\stair}{{\mathrm{st}}}

\newcommand{\llambda}{{{\bm \lambda}}}
\newcommand{\mmu}{{{\bm \mu}}}
\newcommand{\ggamma}{{\bm\gamma}}

\newcommand{\OP}{{\mathcal{OP}}}

\newcommand{\EEE}{{\mathcal{E}}}

\newcommand{\sign}{{\mathrm{sign}}}
\newcommand{\Stir}{{\mathrm{Stir}}}
\newcommand{\Hom}{{\mathrm{Hom}}}

\newcommand{\Ind}{{\mathrm{Ind}}}
\newcommand{\Res}{{\mathrm{Res}}}

\newcommand{\DDD}{{\mathfrak{D}}}
\newcommand{\FFF}{{\mathfrak{F}}}

\newcommand{\AAA}{{\mathcal{A}}}
\newcommand{\BBB}{{\mathcal{B}}}

\newcommand{\LLL}{{\mathcal{L}}}

\newcommand{\PPP}{{\mathfrak{P}}}

\newcommand{\FF}{{\mathbb{F}}}
\newcommand{\CC}{{\mathbb{C}}}

\newcommand{\RR}{{\mathbb{R}}}
\newcommand{\TT}{{\mathbb{T}}}
\newcommand{\ZZ}{{\mathbb{Z}}}

\newcommand{\XX}{{\mathbf{X}}}

\newcommand{\xx}{{\mathbf{x}}}
\newcommand{\yy}{{\mathbf{y}}}

\newcommand{\sss}{{\mathbf{s}}}

\newcommand{\Frob}{{\mathrm{Frob}}}
\newcommand{\grFrob}{{\mathrm{grFrob}}}
\newcommand{\Low}{{\mathrm{Low}}}

\newcommand{\Gale}{{\mathrm{Gale}}}
\newcommand{\antisymm}{\varepsilon_{\bm\lambda}}
\newcommand{\SH}{SI_{n,r}^\perp}

\begin{document}

\title[Superspace coinvariants for wreath products]
{Superspace coinvariants for wreath products}

\author[Sutanay Bhattacharya]{Sutanay Bhattacharya}
\author[Brendon Rhoades]{Brendon Rhoades}
\address{Department of Mathematics, University of California, San Diego}
\email{(subhattacharya, bprhoades)@ucsd.edu}

\begin{abstract}
    Let $\Omega$ be the superspace ring of regular differential forms on the affine space $\CC^n$. If $G \subseteq GL_n(\CC)$ is a complex reflection group, the {\em $G$-superspace coinvariant ring} is the quotient $SR_G := \Omega_n/SI_G$ where $SI_G \subseteq \Omega$ is the ideal generated by $G$-invariants with vanishing constant term. We study this ring when $G = \ZZ_r \wr \symm_n$ is the group of $r$-colored permutation matrices. We prove a conjecture of Sagan and Swanson on a monomial basis for $SR_G$ and give an Operator Theorem description of its inverse system. We also give a combinatorial model for the ungraded and exterior-graded structure of $SR_G$ as a $G$-module.
\end{abstract}

\maketitle

\section{Introduction}
\label{sec:Intro}

Let $V = \CC^n$ be a finite-dimensional complex vector space and let $S := \CC[V] = \CC[x_1,\dots,x_n]$ be the ring of regular functions on $V$, and let $G \subseteq GL(V) = GL_n(\CC)$ be a complex reflection group. The natural action of $G$ on $V$ induces an action of $G$ on $S$. The {\em $G$-coinvariant ring} is the quotient $R_G := S/I_G$ where $I_G := (S^G_+) \subseteq S$ is the ideal generated by $G$-invariants with vanishing constant term. The quotient $R_G$ is a graded $G$-module, and when $G$ is a complex reflection group Chevalley \cite{Chevalley} proved that $R_G$ is a graded refinement of the regular representation $\CC[G]$.

Let $\Omega := \CC[V] \otimes \wedge(V^*) = \CC[x_1,\dots,x_n] \otimes \wedge \{ \theta_1,\dots,\theta_n\}$ be the ring of regular differential forms on $\CC^n$, otherwise known as the {\em superspace ring}. The ring $\Omega$ is bigraded by placing the $x$-variables in degree $(1,0)$ and the $\theta$-variables in degree $(0,1)$. Drawing terminology from physics, we refer to the $x$-variables as {\em bosonic}  and the $\theta$-variables as {\em fermionic}. Any reflection group $G \subseteq GL(V)$ acts naturally on $\Omega$ in a bigraded fashion. 
The {\em $G$-superspace coinvariant ring} is $$SR_G := \Omega/SI_G$$ where $SI_G := (\Omega^G_+) \subseteq \Omega$ is the (two-sided) ideal generated by $G$-invariants with vanishing constant term. The ring $SR_G$ is a bigraded $G$-module. A $GL_n(\FF_q)$-analog of the Hilbert series result was recently obtained by Rhoades and Wilson \cite{RWq}.

When $G = \symm_n$ is the group of $n \times n$ permutation matrices, the study of $SR_n := SR_{\symm_n}$ was initiated in the late 2010s by the Fields Institute Combinatorics Group (see \cite{Zabrocki}). A family of conjectures about  $SR_n$ were resolved over the last few years. These included calculating the bigraded Hilbert series of $SR_n$ \cite{RW}, finding an explicit monomial basis of $SR_n$ \cite{ACKMR}, and determining the bigraded $\symm_n$-isomorphism type of $SR_n$ \cite{MRW}. 

The structure of $SR_G$ remains mysterious for most complex reflection groups $G$. Before this paper, the vector space dimension of $SR_{n,r}$ was unknown for $r > 2$. Swanson and Wallach \cite{SW} determined the bigraded module structure of $SR_G$ for $n \leq 2$.  In this paper we study the case where $G = \ZZ_r \wr \symm_n$ is the group of $r$-colored permutation matrices. Elements of $\ZZ_r \wr \symm_n$ are $n \times n$ monomial matrices in which the unique nonzero entry in any row or column is an $r$th root-of-unity. For $r = 1$ we recover $\symm_n$ and for $r = 2$ we obtain the hyperoctohedral group of signed permutations. As reflected by the notation, we can think of $\ZZ_r \wr \symm_n$ as a wreath product of the cyclic group of order $r$ with $\symm_n$. In reflection group notation, one usually writes $\ZZ_r \wr \symm_n = G(r,1,n)$. To avoid clutter, we write
\begin{equation}
    SR_{n,r} := SR_{\ZZ_r \wr \symm_n}
\end{equation}
for the superspace coinvariant ring attached to $\ZZ_r \wr \symm_n$. Our main results are as follows.

\begin{itemize}
    \item 
    Sagan and Swanson conjectured \cite{SS} an explicit monomial basis for $SR_{n,r}$; we prove their conjecture (Theorem~\ref{thm:ss-basis}). This gives the first explicit basis of $SR_{n,r}$ for $r > 2$, and the first monomial basis of $SR_{n,r}$ in the hyperoctohedral case $r =2$. 
    
    \item As a corollary (Corollary~\ref{cor:hilbert}), the bigraded Hilbert series $\Hilb(SR_{n,r};q,z)$ is governed by the Sagan--Swanson $q$-Stirling numbers associated to $\ZZ_r \wr \symm_n$.
    
    \item We give a combinatorial model for the ungraded (and exterior-graded) $\ZZ_r \wr \symm_n$-module structure of $SR_{n,r}$ using new objects called {\em $r$-ified ordered set partitions} (Theorem~\ref{thm: module iso}). This combinatorial model admits a natural refinement to the fermionic pieces of $SR_{n,r}$   (Theorem~\ref{thm:fermionic-piece-iso}). In the hyperoctohedral case $r=2$, we identify $SR_{n,2}$ with the $\det$-twisted action of the type B Weyl group on the type B Coxeter complex (Corollary~\ref{cor:B-identification}).
\end{itemize}

When $r=2$, the first author \cite{Bhattacharya} used the theory of hyperplane arrangements to describe a non-monomial vector space basis of $SR_{n,2}$.  The monomial basis of $SR_{n,r}$ from the first bullet point will be used in a crucial way to compute the module structure of $SR_{n,r}$ in the third bullet point. Along the way to proving our results, we establish an `Operator Theorem' (Theorem~\ref{thm:operator}) which characterizes the Macaulay-inverse system attached to $SR_{n,r}$ in terms of certain `higher Euler operators' on $\Omega$. 

Let $G$ be a Weyl group acting on its reflection representation $V \cong \CC^n$. We write $\OP_G$ for the Coxeter complex whose faces are formed by the reflecting hyperplanes of $G$ and $\OP_{G,k}$ for its set of $k$-dimensional faces. In type $A_{n-1}$ we may identify $\OP_{G,k}$ with the family of ordered set partitions of $[n]$ with $k$ blocks. The set $\OP_G$ carries a natural permutation action of $G$ and its subsets $\OP_{G,k}$ are stable under this action. Murai--Rhoades--Wilson conjectured \cite[Conj. 7.1]{MRW} that for all $d \geq 0$ there exists a $G$-equivariant surjection
\begin{equation}
\label{eq:surjection-conjecture}
\varphi:
    \text{(fermionic degree $n-d$ part of } SR_G) \twoheadrightarrow \CC[\OP_G] \otimes \det
\end{equation}
where $\det$ is the linear determinant character of $G$.
It is proven in \cite{MRW} that such a $\varphi$ exists (and is an isomorphism) in type A. Our results imply that $\varphi$ exists (and is an isomorphism) in type BC,  giving more evidence for \cite[Conj. 7.1]{MRW}. It was observed in \cite{MRW} that $\varphi$ exists (but just barely fails to be an isomorphism) in type F. It is conjectured that $\varphi$ is an isomorphism in type D.

The conjectural surjection~\eqref{eq:surjection-conjecture} is reminiscent of diagonal coinvariant theory. Let $G$ be an irreducible Weyl group with reflection representation $V$ and let $V^*$ be the dual of $V$.  Then $G$ acts on the space $\CC[V \oplus V^*]$ in a bigraded fashion. The {\em diagonal coinvariant ring} is the quotient
\begin{equation}
    DR_G := (\CC[V \oplus V^*])/ ( \CC[V \oplus V^*]^G_+)
\end{equation}
of this space by the ideal generated by $G$-invariants with vanishing constant term.  Let $Q$ be the root lattice associated to $G$ and let $h$ be the Coxeter number. The {\em finite torus} associated to $G$ is the quotient group $Q/(h+1)Q$. This set carries an action of $G$. Gordon used rational Cherednik algebras to prove \cite{Gordon} that there exists a $G$-equivariant  surjection
\begin{equation}
\label{eq:gordon-surjection}
    \psi: DR_G \twoheadrightarrow \CC[Q/(h+1)Q] \otimes \det
\end{equation}
for any Weyl group $G$. Thanks to Haiman's work on Hilbert schemes \cite{Haiman}, the surjection \eqref{eq:gordon-surjection} is an isomorphism in type A.  However, \eqref{eq:gordon-surjection} already fails to be an isomorphism in type BC. In this sense, the Coxeter complex $\OP_G$ gives a slightly `more faithful' combinatorial approximation to the superspace coinvariant ring $SR_G$ than the finite torus $Q/(h+1)$ to the diagonal coinvariant ring $DR_G$. Working over the ring $\wedge (V \oplus V^*)$ with two copies of $n$ fermionic variables, Kim--Rhoades \cite{KR} gave a completely uniform description of the bigraded $G$-module structure of
\begin{equation}
FDR_G := \wedge (V \oplus V^*) / ( \wedge (V \oplus V^*)^G_+)
\end{equation}
for any complex reflection group $G$. This shows a general trend of fermionic variables giving rise to more uniformly combinatorial (and more easily analyzed) quotient rings than bosonic variables. Lentfer \cite{Lentfer} proved a combinatorial supersymmetry result conjectured in type A by F. Bergeron \cite{Bergeron} which gives a duality between the setting of arbitrarily many sets of bosonic variables and that of arbitrarily many sets of fermionic variables.

The rest of our paper is organized as follows. In {\bf Section~\ref{sec:Background}} we give background on (super)commutative algebra and the representation theory of $\ZZ_r \wr \symm_n$. In {\bf Section~\ref{sec:Basis}} we prove the Sagan--Swanson conjecture to obtain a monomial basis for $SR_{n,r}$ and prove our Operator Theorem. In {\bf Section~\ref{sec:Module}} we use this basis to determine the ungraded and fermionic-graded module structure of $SR_{n,r}$. We close in {\bf Section~\ref{sec:Conclusion}} with some open problems.

\section{Background}
\label{sec:Background}

\subsection{Combinatorics} 
Given $n \geq 0$ we write $[n] := \{1,\dots,n\}$. For $0 \leq r \leq n$, the {\em Gale order} $\leq_\Gale$ is the partial order on the family of subsets of $[n]$ of size $r$ defined by
\[
\{ i_1 < \cdots < i_r \} \leq_\Gale \{j_1 < \cdots < j_r \}
\]
if $i_s \leq j_s$ for all $s = 1, \dots, r$. 

A {\em partition} of $n \geq 0$ is a weakly decreasing sequence $\lambda = (\lambda_1 \geq \cdots \geq \lambda_k)$ of positive integers with $\lambda_1 + \cdots + \lambda_k = n$. We write $\lambda \vdash n$ to mean that $\lambda$ is a partition of $n$ and $|\lambda| = n$ for the sum of the parts of $\lambda$. More generally, given $n,r \geq 0$, an {\em $r$-partition of $n$} is a sequence $\llambda = (\lambda^{(1)},\dots,\lambda^{(r)})$ of $r$ partitions such that $|\lambda^{(1)}| + \cdots + |\lambda^{(r)}| = n$. We write $\llambda \vdash_r n$ to mean that $\llambda$ is an $r$-partition of $n$ and write $|\llambda| := |\lambda^{(1)}| + \cdots + |\lambda^{(r)}| = n.$

Let $\xx = (x_1,x_2,\dots)$ be an infinite list of variables and let $\Lambda = \bigoplus_{n \geq 0} \Lambda_n$ be the graded ring of symmetric functions over the ground field $\CC(q,z)$. We refer the reader to Macdonald's book \cite{Macdonald} as an excellent  symmetric function theory reference. Given $d > 0$, we have the {\em elementary symmetric function} $e_d(\xx)$, the {\em complete homogeneous symmetric function} $h_d(\xx)$, and the {\em power sum symmetric function} $p_d(\xx)$ of degree $d$ given by
\begin{equation}
    h_d(\xx) := \sum_{1 \leq i_1 \leq \cdots \leq i_d} x_{i_1} \cdots x_{i_d}, \quad \quad
    e_d(\xx) := \sum_{1 \leq i_1 < \cdots < i_d} x_{i_1} \cdots x_{i_d}, \quad \quad
    p_d(\xx) := \sum_{i \geq 1} x_i^d.
\end{equation}
We will sometimes consider these polynomials over finite variable sets.

Bases of $\Lambda_n$ are indexed by partitions $\lambda \vdash n$. In this paper we will use the {\em elementary basis} $\{e_\lambda(\xx)\}_{\lambda \vdash n}$, the {\em complete homogeneous basis} $\{ h_\lambda(\xx) \}_{\lambda \vdash n}$, the {\em power sum basis} $\{ p_\lambda(\xx) \}_{\lambda \vdash n}$,  and the {\em Schur basis} $\{ s_\lambda(\xx) \}_{\lambda \vdash n}.$ The first three of these bases are given by
\begin{equation}
    e_\lambda(\xx) := \prod_{i=1}^k e_{\lambda_i}(\xx), \quad \quad 
    h_\lambda(\xx) := \prod_{i=1}^k h_{\lambda_i}(\xx), \quad \quad
    p_\lambda(\xx) := \prod_{i=1}^k p_{\lambda_i}(\xx)
\end{equation}
for any partition $\lambda = (\lambda_1, \dots, \lambda_k) \vdash n$. Assuming $\lambda = (\lambda_1 \geq \cdots \geq \lambda_N \geq 0)$ has $\leq N$ parts, the {\em Schur polynomial} in the finite variable set $\{x_1,\dots,x_N\}$ is given by the {\em bialternant formula}
\begin{equation}
    s_\lambda(x_1,\dots,x_N) := \frac{\det(x_j^{\lambda_i + i - 1})_{1 \leq i, j \leq N}}{\det(x_j^{i-1})_{1 \leq i, j \leq N}}.
\end{equation}
This ratio is a symmetric polynomial in $\{x_1,\dots,x_N\}$ whose $N \to \infty$ limit is a well-defined formal power series in $\xx = (x_1,x_2,\dots)$; this is the {\em Schur function}
\begin{equation}
    s_\lambda(\xx) := \lim_{N \to \infty} s_\lambda(x_1,\dots,x_N).
\end{equation} We will also use the combinatorial definition of $s_\lambda(\xx)$ given by
\begin{equation}
    s_\lambda(\xx) := \sum_T \xx^T
\end{equation}
where $T$ ranges over {\em semistandard Young tableaux} of shape $\lambda$ (see e.g, \cite{Macdonald}) and $\xx^T := x_1^{c_1(T)} x_2^{c_2(T)} \cdots $ where $c_i(T)$ is the number of $i$'s in $T$.
We write $\langle -, - \rangle$ for the {\em Hall inner product} on $\Lambda$ with respect to which the Schur basis is orthonormal.

Let $r \geq 1$ and consider $r$ infinite alphabets of variables \[\xx^{(1)} = (x_1^{(1)}, x_2^{(1)}, \dots), \quad  \dots \quad , \,  \xx^{(r)} = (x_r^{(1)}, x_r^{(2)}, \dots ).\] 
We write $\Lambda^{(r)}$ for the ring of formal power series in these variables of bounded degree which are symmetric in each of the $r$ variable sets separately. From an algebraic point of view, we have the identification \[\Lambda^{(r)} = \overbrace{\Lambda \otimes \cdots \otimes \Lambda}^r\] of $\Lambda^{(r)}$ with the $r$-fold tensor product of $\Lambda$ with itself. We have the direct sum decomposition
\begin{equation}
    \Lambda^{(r)} = \bigoplus_{n_1,\dots,n_r \geq 0} (\Lambda_{n_1} \otimes \cdots \otimes \Lambda_{n_r}).
\end{equation}
The ring $\Lambda^{(r)}$ is graded by placing the direct summand $\Lambda_{n_1} \otimes \cdots \otimes \Lambda_{n_r}$ in degree $n_1 + \cdots + n_r$. Given any partition-indexed basis $\{ f_\lambda(\xx) \}$ of $\Lambda$ where $f_\lambda$ is homogeneous of degree $|\lambda|$, we have an induced $r$-partition indexed basis $\{ f_\llambda \}$ of $\Lambda^{(r)}$ given by
\begin{equation}
    f_\llambda := f_{\lambda^{(1)}}(\xx^{(1)}) \otimes \cdots \otimes f_{\lambda^{(r)}}(\xx^{(r)}) \quad \text{ for } \llambda = (\lambda^{(1)}, \dots, \lambda^{(r)}) \text{ an $r$-partition}.
\end{equation}
Observe that $\deg(f_\llambda) = |\llambda|.$ The Hall inner product extends to $\Lambda^{(r)}$ by declaring the Schur basis $\{ s_\llambda \}$ to be orthonormal. i.e.
\begin{equation}
    \langle s_\llambda, s_\mmu \rangle := \prod_{i=1}^r \langle s_{\lambda^{(i)}}, s_{\mu^{(i)}} \rangle  = \begin{cases}
        1 & \text{if $\llambda = \mmu$,} \\
        0 & \text{otherwise.}
    \end{cases}
\end{equation}

\subsection{Commutative algebra}\label{sectionCommAlg} Let $V = \CC^n$ be an $n$-dimensional complex vector space. As in the introduction, we write $S = \CC[V] \cong \CC[x_1,\dots,x_n]$ for the ring of polynomial functions on $V$ with its standard grading $\deg(x_i) = 1$. Given any homogeneous ideal $I \subseteq S$, the quotient $S/I = \bigoplus_{i \geq 0} (S/I)_i$ is a graded $\CC$-vector space. The {\em Hilbert series} of $S/I$ is given by
\begin{equation}
    \Hilb(S/I;q) := \sum_{i \geq 0} \dim_\CC (S/I)_i \cdot q^i
\end{equation}
where $q$ is a formal variable tracking bosonic (commutative) degree. More generally, if $V = \bigoplus_{i \geq 0} V_i$ is a graded $\CC$-vector space with each piece $V_i$ finite-dimensional we write
\begin{equation}
    \Hilb(V;q) = \sum_{i \geq 0} \dim_\CC V_i \cdot q^i
\end{equation}
for its graded dimension.

A sequence $f_1, \dots, f_t \in S$ of homogeneous polynomials of positive degree is a {\em regular sequence} if for all $1 \leq i \leq t$ the map
\[
f_i \times (-): S/(f_1,\dots,f_{i-1}) \longrightarrow S/(f_1,\dots, f_{i-1})
\]
induced by multiplication by $f_i$ is injective. If $f_1, \dots, f_n$ is a regular sequence of the maximum possible length $n$ we have
\begin{equation}
    \Hilb(S/I;q) = [\deg(f_1)]_q \cdots [\deg(f_n)]_q
\end{equation}
where
\begin{equation}
    [m]_q := 1 + q + q^2 + \cdots + q^{m-1} \quad (m \geq 0).
\end{equation}
A homogeneous ideal $I \subseteq S$ is a {\em complete intersection} if it can be generated by a regular sequence. A useful criterion for proving that a sequence is regular is as follows.

\begin{lemma}
    \label{lem:locus-criterion}
    Let $f_1,\dots,f_n \in S = \CC[x_1,\dots,x_n]$ be a list of $n$ homogeneous polynomials of positive degree. The sequence $f_1,\dots,f_n$ is a regular sequence if and only if the system of polynomial equations $f_1 = \cdots = f_n = 0$ has only the trivial solution $0 \in \CC^n$.
\end{lemma}

Let $I \subseteq S$ be an ideal and let $f \in S$. The {\em colon ideal} is 
\begin{equation}
    (I:f) := \{g \in S \,:\, f \cdot g \in I \}.
\end{equation}
It is not hard to see that $(I:f)$ is an ideal in $S$ which contains $I$.
We will need a criterion which first appeared in the context of Hessenberg varieties \cite{AHMMS} for showing that colon ideals are generated by regular sequences.

\begin{lemma}
    \label{lem:regular-criterion}
    {\em (Abe--Horiguchi--Maeno--Murai--Sato \cite[Lem. 2.4]{AHMMS})} Let $I \subseteq S$ be a graded ideal such that $S/I$ is a finite-dimensional $\CC$-vector space. Let $d \geq 0$ be maximal such that $(S/I)_d \neq 0$. Assume further that $I$ is a complete intersection. Let $f \in S$ be a homogeneous element and assume that $f \notin I$.

    Suppose $f_1, f_2, \dots, f_n \in S$ is a regular sequence of homogeneous polynomials contained in $(I:f)$ such that  
    $\deg(f_1) + \deg(f_2) +  \cdots + \deg(f_n) + \deg(f) = d$.
    We have the equality of ideals
    \begin{equation}
        (I:f) = (f_1, f_2, \dots, f_n).
    \end{equation}
\end{lemma}

For $1 \leq i \leq n$, we have the operator $\partial_i: S \to S$ on $S = \CC[x_1,\dots,x_n]$ given by partial differentiation with respect to $x_i$.  Since mixed partials commute, given $f \in S$ we have a well-defined operator $\partial f: S \to S$ given by
\begin{equation}
    \partial f := f(\partial_1, \dots, \partial_n) \quad \text{where} \quad f= f(x_1,\dots,x_n).
\end{equation}
This gives rise to a module structure $\odot: S \times S \to S$ defined by
\begin{equation}
    f \odot g := \partial f(g).
\end{equation}
If $I \subseteq S$ is a graded ideal, the {\em inverse system} (or {\em harmonic space}) is given by
\begin{equation}
    I^\perp := \{ g \in S \,:\, f \odot g = 0 \text{ for all } f \in I \}.
\end{equation}
It can be shown that $I^\perp$ is a graded subspace of $S$ with the same Hilbert series as $S/I$.

Also as in the introduction, we write $\Omega = \CC[V] \otimes_\CC \wedge(V^*)$ for the rank $n$ superspace ring. Fixing the standard basis for $V = \CC^n$ induces the identification $\Omega = \CC[x_1,\dots,x_n] \otimes_\CC \wedge \{ \theta_1,\dots,\theta_n\}$. We may regard $\Omega$ as the unital associative $\CC$-algebra with generators $x_1,\dots,x_n, \theta_1,\dots ,\theta_n$ subject to the relations
\begin{equation}
    x_i x_j = x_j x_i \quad \quad x_i \theta_j = \theta_j x_i \quad \quad \theta_i \theta_j = - \theta_j \theta_i
\end{equation}
for all $1 \leq i,j \leq n$.
Since we are not in characteristic 2, we have $\theta_i^2 = 0$ for all $i$. 

The $\CC$-algebra $\Omega = \bigoplus_{i,j \geq 0} \Omega_{i,j}$ is bigraded with $\Omega_{i,j} = \CC[V]_i \otimes_\CC \wedge^j(V^*).$ If $I \subseteq \Omega$ is a bigraded left (or right) ideal, then $I$ is automatically a two-sided ideal and the quotient $\Omega/I = \bigoplus_{i,j \geq 0}(\Omega/I)_{i,j}$ is a bigraded vector space. We write
\begin{equation}
    \Hilb(\Omega/I;q,z) := \sum_{i,j \geq 0} \dim_\CC (\Omega/I)_{i,j} \cdot q^i z^j
\end{equation}
for its bigraded dimension, where $z$ is a formal variable tracking fermionic (exterior) degree. Similarly, if $V \subseteq \Omega$ is a bigraded subspace with $V = \bigoplus_{i,j \geq 0} V_{i,j}$, we write
\begin{equation}
    \Hilb(V;q,z) := \sum_{i,j \geq 0} \dim_\CC(V_{i,j}) \cdot q^i z^j
\end{equation}
for its bigraded dimension.

For $1 \leq i \leq n$, the {\em fermionic partial derivative} (or {\em contraction operator}) $\partial_i^\theta: \Omega \to \Omega$ is the $S$-module homomorphism characterized by
\begin{equation}
    \partial_i^\theta: \theta_{j_1} \cdots \theta_{j_k} \mapsto \begin{cases}
        (-1)^{s-1} \theta_{j_1} \cdots \widehat{\theta_{j_s}} \cdots \theta_{j_k} & \text{if $j_s = i$,} \\
        0 & \text{if $i \neq j_1,\dots,j_k$}
    \end{cases}
\end{equation}
for any distinct indices $1 \leq j_1,\dots,j_k \leq n$. The usual partial derivative $\partial_i$ acts on $\Omega$ via the first tensor factor $\Omega = \CC[x_1,\dots,x_n] \otimes \wedge \{ \theta_1,\dots,\theta_n\}$.  It can be shown that the $\partial_i$ and the $\partial_i^\theta$ satisfy the same relations as $\Omega$, i.e.
\[
\partial_i \partial_j = \partial_j \partial_i, \quad 
\partial_i \partial^\theta_j = \partial^\theta_j \partial_i, \quad\partial^\theta_i \partial^\theta_j = - \partial^\theta_j \partial^\theta_i
\]
for all $1 \leq i , j \leq n$. Consequently, if $f = f(x_1,\dots,x_n, \theta_1,\dots,\theta_n) \in \Omega$, we have a well-defined operator $\partial f : \Omega \to \Omega$ given by 
\begin{equation}
    \partial f := f(\partial_1, \dots, \partial_n, \partial^\theta_1, \dots, \partial^\theta_n).
\end{equation}
As in the commutative case, this gives a module structure $\odot: \Omega \times \Omega \to \Omega$ defined by 
\begin{equation}
    f \odot g := \partial f(g).
\end{equation}
If $I \subseteq \Omega$ is a bigraded ideal, the {\em inverse system} is again defined as
\begin{equation}
    I^\perp := \{ g \in \Omega \,:\, f \odot g = 0 \text{ for all } f \in I \}.
\end{equation}
One again has that $I^\perp$ is a bigraded subspace of $\Omega$ with the same bigraded Hilbert series as $\Omega/I$.

\subsection{Representation theory}\label{sectionRepTheory} As in the introduction, we write $\ZZ_r \wr \symm_n$ for the group of $n \times n$ monomial matrices whose nonzero entries are $r^{th}$ roots-of-unity. For example, if $r = 5$ and $n = 4$ we have
\[
\begin{pmatrix}
    0 & \zeta^2 & 0 & 0 \\
    0 & 0 & 0 & \zeta^4 \\
    1 & 0 & 0 & 0 \\
    0 & 0 & \zeta^4 & 0
\end{pmatrix} \in \ZZ_5 \wr \symm_4 \quad \text{where } \zeta = \exp(2 \pi i /5).
\]
There is a natural epimorphism $\ZZ_r \wr \symm_n \twoheadrightarrow \symm_n$ given by replacing the nonzero entries of a matrix $g \in \ZZ_r \wr \symm_n$ with 1s; the image of $g$ under the map is the {\em underlying permutation matrix} of $g$. 
We review the representation theory of $\ZZ_r \wr \symm_n$, starting with the case of $\symm_n$.

Irreducible representations of $\symm_n$ are in one-to-one  correspondence with partitions of $n$. Given $\lambda \vdash n$, we write $V^\lambda$ for the corresponding irreducible $\symm_n$-module. If $V$ is any finite-dimensional $\symm_n$-module, there are unique multiplicities $c_\lambda \geq 0$ so that 
\begin{equation}V \cong \bigoplus_{\lambda \vdash n} c_\lambda V^\lambda.\end{equation}  
The {\em Frobenius image} of $V$ is the symmetric function 
\begin{equation}
\Frob(V) := \sum_{\lambda \vdash n} c_\lambda \cdot s_\lambda
\end{equation} given by replacing each irreducible $V^\lambda$ with the corresponding Schur function $s_\lambda$. More generally, if $V = \bigoplus_{i \geq 0} V_i$ or $V = \bigoplus_{i,j \geq 0} V_{i,j}$ is a graded or bigraded $\symm_n$-module, we have the {\em (bi)graded Frobenius images}
\begin{equation}
\grFrob(V;q) = \sum_{i \geq0} \Frob(V_i) \cdot q^i \quad \text{or} \quad 
\grFrob(V;q,z) = \sum_{i,j \geq 0} \Frob(V_{i,j}) \cdot q^i z^j.
\end{equation}

The irreducible representations of $\ZZ_r \wr \symm_n$ (and more generally those of $G \wr \symm_n$ where $G$ is a finite group) were constructed by Specht \cite{Specht} in his thesis. His construction is as follows.

The group $\ZZ_r \wr \symm_n$ is generated by its subgroup $(\ZZ_r)^n$ of diagonal matrices together with its subgroup $\symm_n$ of permutation matrices (consisting only of 0s and 1s). Let $U$ be a $\ZZ_r$-module and $V$ be an $\symm_n$-module. We build a $(\ZZ_r \wr \symm_n)$-module $U \wr V$ with underlying vector space $U^{\otimes n} \otimes V$ by the rules
\begin{equation}
    \mathrm{diag}(g_1,\dots,g_n) \cdot (u_1 \otimes \cdots \otimes u_n \otimes v) := (g_1 \cdot u_1) \otimes \cdots \otimes (g_n \cdot u_n) \otimes v
\end{equation}
for all diagonal matrices $\mathrm{diag}(g_1,\dots,g_n) \in \ZZ_r \wr \symm_n$ with $g_i \in \ZZ_r$ and
\begin{equation}
    w \cdot (u_1 \otimes \cdots \otimes u_n \otimes v) := u_{w^{-1}(1)} \otimes \cdots \otimes u_{w^{-1}(n)} \otimes (w \cdot v)
\end{equation}
for all $w \in \symm_n$. If $V$ is an irreducible $\ZZ_r$-module and $U$ is an irreducible $\symm_n$-module, then $U \wr V$ is an irreducible $(\ZZ_r \wr \symm_n)$-module, but not all irreducible $(\ZZ_r \wr \symm_n)$-modules are obtained in this way.

Let $\alpha = (\alpha_1,\dots,\alpha_r)$ be a list of $r$ nonnegative integers with $\alpha_1 + \cdots + \alpha_r = n$. We have a block diagonal subgroup
\begin{equation}
    \ZZ_r \wr \symm_\alpha := \ZZ_r \wr \symm_{\alpha_1} \times \cdots \times \ZZ_r \wr \symm_{\alpha_r} \subseteq \ZZ_r \wr \symm_n.
\end{equation}
Given $(\ZZ_r \wr \symm_{\alpha_i})$-modules $W_i$, the tensor product $W_1 \otimes \cdots \otimes W_r$ is naturally a $(\ZZ_r \wr \symm_\alpha)$-module and the induction $\Ind_{\ZZ_r \wr \symm_\alpha}^{\ZZ_r \wr \symm_n}(W_1 \otimes \cdots \otimes W_r)$ is a $(\ZZ_r \wr \symm_n)$-module.

Let $\zeta \in \CC^\times$ be a primitive $r^{th}$ root-of-unity. We index the irreducible representations of $\ZZ_r = \langle \zeta \rangle$ in the following (slightly nonstandard) way. For $1 \leq i \leq r$ let $\rho_i: \ZZ_r \to GL_1(\CC)$  be the map $\rho_i: \zeta \mapsto \zeta^i$ and let $U_i$ be the vector space $\CC^1$ with the $\ZZ_r$-module structure afforded by $\rho_i$. Thus $U_r$ is the trivial representation and $U_1, \dots, U_{r-1}$ are the nontrivial irreducible representations of $\ZZ_r$.

Irreducible representations of $\ZZ_r \wr \symm_n$ are indexed by $r$-partitions of $n$. Given an $r$-partition $\llambda = (\lambda^{(1)}, \dots, \lambda^{(r)}) \vdash_r n$, define $V^\llambda$ to be the induced $(\ZZ_r \wr \symm_n)$-module
\begin{equation}
    V^\llambda := \Ind_{\ZZ_r \wr \symm_\alpha}^{\ZZ_r \wr \symm_n}((U_1 \wr V^{\lambda^{(1)}}) \otimes \cdots \otimes (U_r \wr V^{\lambda^{(r)}}))
\end{equation}
where $\alpha_i = |\lambda^{(i)}|$. If $V$ is a finite-dimensional $(\ZZ_r \wr \symm_n)$-module, there are unique multiplicities $c_\llambda \geq 0$ so that 
\begin{equation}
    V \cong \bigoplus_{\llambda \vdash_r n} c_\llambda \cdot V^\llambda.
\end{equation}
As in the case of $\symm_n$, we define $\Frob(V) \in \Lambda^{(r)}$ by
\begin{equation}
    \Frob(V) := \sum_{\llambda \vdash_r n} c_\llambda \cdot s_\llambda.
\end{equation}
If $V = \bigoplus_{i \geq 0} V_i$ or $V = \bigoplus_{i,j \geq 0} V_{i,j}$ are graded or bigraded $(\ZZ_r \wr \symm_n)$-modules, we define $\grFrob(V;q)$ or $\grFrob(V;q,z)$ via
\begin{equation}
    \grFrob(V;q) := \sum_{i \geq 0} \Frob(V_i) \cdot q^i \quad \text{or} \quad \grFrob(V;q,z) := \sum_{i,j \geq 0} \Frob(V_{i,j}) \cdot q^i z^j.
\end{equation}

Let $V = \CC^n$. Regarding $\ZZ_r \wr \symm_n$ as a subgroup of $GL_n(\CC) = GL(V)$, we have a natural `defining action' of $\ZZ_r \wr \symm_n$ on $V$. Making the identifications $S = \CC[V]$ and $\Omega = \CC[V] \otimes \wedge(V^*)$, we have induced actions of $\ZZ_r \wr \symm_n$ on $S$ and $\Omega$. These actions are characterized by
\begin{equation}
    \mathrm{diag}( \zeta^{j_1}, \dots, \zeta^{j_i}) \cdot x_i = (\zeta^{j_j})^{-1} x_i \quad \text{and} \quad 
    \mathrm{diag}( \zeta^{j_1}, \dots, \zeta^{j_i}) \cdot \theta_i = (\zeta^{j_j})^{-1} \theta_i
\end{equation}
for diagonal matrices $\mathrm{diag}( \zeta^{j_1}, \dots, \zeta^{j_i}) \in \ZZ_r \wr \symm_n$ and 
\begin{equation}
    w \cdot x_i = x_{w(i)} \quad \text{and} \quad w \cdot \theta_i = \theta_{w(i)}
\end{equation}
for permutation matrices $w \in \symm_n \subseteq \ZZ_r \wr \symm_n$.

Given a graded ideal $I \subseteq S$ which is stable under this action, the inverse system $I^\perp$ is a graded $(\ZZ_r \wr \symm_n)$-module and we have $\grFrob(S/I;q) = \grFrob(I^\perp;q) \in \Lambda^{(r)}$. Similarly, if $I \subseteq \Omega$ is a $(\ZZ_r \wr \symm_n)$-stable bigraded ideal, then $I^\perp \subseteq \Omega$ is a graded $(\ZZ_r \wr \symm_n)$-module and we have $\grFrob(\Omega/I;q,z) = \grFrob(I^\perp;q,z)$.

A simple example involving the cyclic group of order $r$ should help solidify our wreath product conventions.

\begin{example}
    \label{ex:rank-1}
    Let $n = 1$ and consider the bigraded action of $\ZZ_r \wr \symm_1 \cong \ZZ_r$ on $\Omega = \CC[x_1] \otimes \wedge \{ \theta_1\}$. Let $\zeta = \exp(2 \pi i /r)$ and write $\ZZ_r = \langle \zeta \rangle$. The action of $\ZZ_r$ on $\Omega$ is characterized by
    \[
    \zeta \cdot (x_1^i) = \zeta^{-i} x_1^i \quad \text{and} \quad \zeta \cdot (x_1^i \theta_1) = \zeta^{-(i+1)} x_1^i \theta_1.
    \]
    It is easily seen that the invariant subalgebra $\Omega^{\ZZ_r}$ is generated by $\{ x_1^r, x_1^{r-1}\theta_1\}$, in agreement with Solomon's Theorem~\ref{thm:solomon}. The superspace coinvariant ring 
    \[
    SR_{1,r} = \Omega/(\Omega^{\ZZ_r}_+) = \Omega/(x_1^r, x_1^{r-1} \theta_1)
    \]
    has vector space basis
    \[
    \{ 1, \theta_1, x_1, x_1 \theta_1, x_1^2, x_1^2 \theta_1, \dots, x_1^{r-2} \theta_1, x_1^{r-1} \}
    \]
    and this basis is an eigenbasis for the action of $\zeta \in \ZZ_r$. Using this basis one readily computes the bigraded Frobenius image
    \begin{multline*}
    \grFrob(SR_{1,r};q,z) = \\ s_{\varnothing, \dots, \varnothing,  \varnothing,1} + (z + q) \cdot s_{\varnothing, \dots, \varnothing, 1, \varnothing} + (qz + q^2) \cdot s_{\varnothing, \dots, 1, \varnothing, \varnothing} + \cdots + (q^{r-2}z + q^{r-1}) \cdot s_{1, \dots ,\varnothing,\varnothing,\varnothing}.
    \end{multline*}
    Our  unconventional indexing of the irreducible $\ZZ_r$-modules $U_1, U_2, \dots, U_r$ was designed to produce the relatively nice pattern in the bigraded character displayed above.
\end{example}

There are two linear characters of $\ZZ_r\wr\symm_n$ that are worth noting explicitly. The first is $\sign$, given by $$\sign(g):=\text{ determinant of the underlying permutation matrix of }g.$$ The second is $\chi$, given by $$\chi(g):=\text{ product of the non-zero entries in the matrix of }g.$$
With our notational conventions, these irreducible characters correspond to the $r$-partitions
\begin{equation}
    \sign  \leftrightarrow (\varnothing,\ldots,\varnothing,(1^n)) \quad \text{and} \quad 
    \chi \leftrightarrow ((n),\varnothing,\ldots,\varnothing).
\end{equation}

\section{Monomial Basis}
\label{sec:Basis}

\subsection{The ring $SR_{n,r}$} Regarding $\Omega = \CC[x_1,\dots,x_n] \otimes \wedge \{\theta_1,\dots,\theta_n\}$ as the algebra of regular differential forms on $\CC^n$, we have the {\em Euler operator} $d: \Omega \to \Omega$ given by
\begin{equation}
    df :=\sum_{i=1}^n (\partial_i f) \cdot \theta_i.
\end{equation}
As in the introduction, we let $\Omega^{\ZZ_r \wr \symm_n} \subseteq \Omega$ be the $(\ZZ_r \wr \symm_n)$-invariant subalgebra of $\Omega$, write $SI_{n,r} := (\Omega^{\ZZ_r \wr \symm_n}_+) \subseteq \Omega$ for the (two-sided) ideal of $\Omega$ generated by $(\ZZ_r \wr \symm_n)$-invariants with vanishing constant term, and let $SR_{n,r} := \Omega/SI_{n,r}$ be the corresponding quotient ring.

 For $i \geq 0$, write $p_i \in S$ for the power sum of degree $i$, i.e.
\begin{equation}
    p_i := x_1^i + \cdots + x_n^i.
\end{equation}
The polynomial $p_i$ is $(\ZZ_r \wr \symm_n)$-invariant if and only if $r \mid i$. The algebra $S^{\ZZ_r \wr \symm_n}$ of polynomial $(\ZZ_r \wr \symm_n)$-invariants is freely generated by $\{ p_r, p_{2r}, \dots, p_{nr} \}$. The polynomial ideal $I_{n,r} := (S^{\ZZ_r \wr \symm_n}_+)$ generated by $(\ZZ_r \wr \symm_n)$-invariants with vanishing constant term therefore satisfies
\begin{equation}
    I_{n,r} = (S^{\ZZ_r \wr \symm_n}_+) = (p_r, p_{2r}, \dots, p_{nr}).
\end{equation}
A fundamental result of Solomon yields an explicit generating set of the superspace ideal $SI_{n,r} \subseteq \Omega$. A finite subgroup $G \subseteq GL_n(\CC)$ is a {\em complex reflection group} if the algebra $S^G$ of polynomial $G$-invariants can be generated by $n$ homogeneous elements $f_1,\dots,f_n \in S^G$ of positive degree. Such $f_1,\dots,f_n$ are called a {\em fundamental system of polynomial invariants} for $G$.

\begin{theorem}
    \label{thm:solomon}
    {\em (Solomon \cite{Solomon})} Let $G \subseteq GL_n(\CC)$ be a complex reflection group and let $f_1,\dots,f_n \in S^G$ be a fundamental system of polynomial invariants. Then $\Omega^G$ is a free $S^G$-algebra of rank $2^n$ with basis $\{ d f_{i_1} \cdots d f_{i_p} \,:\, 1 \leq i_1 < \cdots < i_p \leq n \}$.
\end{theorem}

Applying Solomon's Theorem~\ref{thm:solomon} in the case $G = \ZZ_r \wr \symm_n$, the superspace ideal $SI_{n,r} \subseteq \Omega$ has generating set 
\begin{equation}
\label{eq:power-sum-generators}
    SI_{n,r} = ( p_r, p_{2r} ,\dots, p_{nr}, dp_r, dp_{2r}, \dots, dp_{nr}).
\end{equation}
The generators in \eqref{eq:power-sum-generators} will be useful in our analysis of $SR_{n,r}$.

\subsection{Staircases} 
Sagan and Swanson conjectured \cite[Conj. 5.6]{SS} a basis for the quotient ring $SR_{n,r}$. For $J \subseteq [n]$, write
\begin{equation}
\theta_J := \prod_{j \in J} \theta_j
\end{equation}
where the product of exterior variables is taken in increasing order. The {\em $(J,r)$-staircase} associated to $\ZZ_r \wr \symm_n$ is the length $n$ sequence \[\stair(J,r) = (\stair(J,r)_1,\dots,\stair(J,r)_n)\] given as follows. Let $U :=[n] - J$ be the complement of $J$ in $[n]$ and write $U = \{u_1 < \cdots < u_k \}$. We define 
\begin{equation}
\stair(J,r)_{u_i} := ir -1.
\end{equation}
Given $j \in J$, there exists a unique $0 \leq i \leq k+1$  such that $u_{i-1} < j < u_i$ where we set $u_0 := 0$ and $u_{k+1} := n+1$. We define
\begin{equation}
\stair(J,r)_j := ir - 2.
\end{equation}

\begin{example}
\label{ex:staircase}
Suppose $n = 9, r = 3,$ and $J = \{3,4,6,9\}$. Then 
$\stair(J,3) = (\stair(J,3)_1, \dots, \stair(J,3)_9)$ is given by
\[
\stair(J,3) = (2,5,\underline{7},\underline{7},8,\underline{10},11,14,\underline{16})
\]
where the underlined numbers are in positions indexed by $J$. The sequence $\stair(J,3)$ is shown in Figure~\ref{fig:staircase} with elements of $J$ bolded.
\end{example}

\begin{figure}
\begin{tikzpicture}[x=0.6cm, y=0.35cm, line join=round]
  \foreach \h [count=\i from 1] in {2,5,7,7,8,10,11,14,16} {
    \foreach \j in {1,...,\h} {
      \draw (\i-1,\j-1) rectangle (\i,\j);
    }
  }

  \foreach \i in {3,4,6,9} {
    \draw[line width=1.4pt] (\i-1,0) -- (\i,0);
  }

  \node[left] at (-0.5,-1) {$i$};
  \foreach \i in {1,...,9} {
    \ifnum\i=3
      \node at (\i-0.5,-1) {\textbf{\i}};
    \else\ifnum\i=4
      \node at (\i-0.5,-1) {\textbf{\i}};
    \else\ifnum\i=6
      \node at (\i-0.5,-1) {\textbf{\i}};
    \else\ifnum\i=9
      \node at (\i-0.5,-1) {\textbf{\i}};
    \else
      \node at (\i-0.5,-1) {\i};
    \fi\fi\fi\fi
  }

  \node[left] at (-0.5,-2.6) {$\mathrm{st}(J,r)_i$};
  \foreach \i/\v in {1/2,2/5,3/7,4/7,5/8,6/10,7/11,8/14,9/16} {
    \ifnum\i=3
      \node at (\i-0.5,-2.6) {\textbf{\v}};
    \else\ifnum\i=4
      \node at (\i-0.5,-2.6) {\textbf{\v}};
    \else\ifnum\i=6
      \node at (\i-0.5,-2.6) {\textbf{\v}};
    \else\ifnum\i=9
      \node at (\i-0.5,-2.6) {\textbf{\v}};
    \else
      \node at (\i-0.5,-2.6) {\v};
    \fi\fi\fi\fi
  }
\end{tikzpicture}
\caption{The staircase $\stair(J,r)$ for $n = 9, r = 3$, and $J = \{3,4,6,9\}$.}
\label{fig:staircase}
\end{figure}
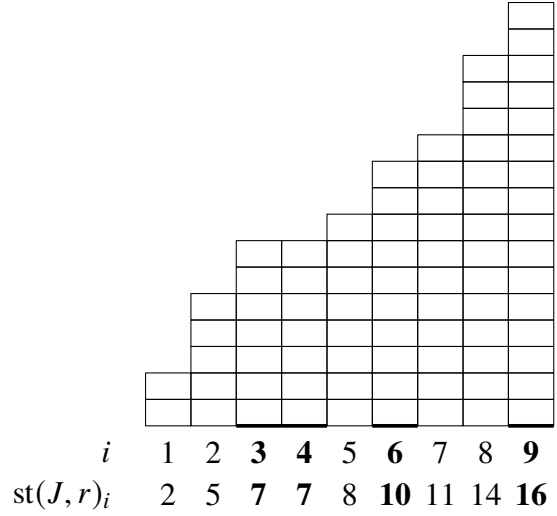

For $J \subseteq [n]$, let $\AAA_{n,r}(J,r) \subseteq \Omega$ be the following set of monomials in the $x$-variables:
\begin{equation}
    \AAA_{n,r}(J,r) := \{ x_1^{a_1} \cdots x_n^{a_n} \,:\, a_i \leq \stair(J,r)_i \}.
\end{equation}
In other words, the set $\AAA_{n,r}(J,r)$ consists of monomials in the $x$-variables whose exponent sequences are componentwise $\leq \stair(J,r)$. We also consider the family of superspace monomials
\begin{equation}
\label{eq:basis-monomials}
    \AAA_{n,r} := \bigsqcup_{J \subseteq [n]} \AAA_{n,r}(J,r) \cdot \theta_J.
\end{equation}

\begin{example}
    \label{ex:staircase-two}
    Continuing the case $(n,r,J) = (9,3,\{3,4,6,9\})$ of Example~\ref{ex:staircase}, the subset $J \subseteq [n]$ contributes
    \[
        \AAA_{n,r}(J,r) \cdot \theta_J = \{ x_1^{a_1} \cdots x_9^{a_9} \cdot \theta_{3469} \,:\, (a_1,\dots,a_9) \leq (2,5,7,7,8,10,11,14,16) \}
    \]
    to the set $\AAA_{n,r}$ where $\leq$ is componentwise inequality.
\end{example}

Observe that the union in \eqref{eq:basis-monomials} is in fact disjoint since the $\theta_J$ monomials are distinct. The Sagan--Swanson basis conjecture for $SR_{n,r}$ reads as follows.

\begin{conjecture}
    \label{conj:SS} {\em (Sagan--Swanson \cite[Conj. 5.6]{SS})}
    The set of monomials $\AAA_{n,r}$ descends to a vector space basis of $SR_{n,r}$.
\end{conjecture}

In this section we prove Conjecture~\ref{conj:SS}. We will also show that the constituent sets $\AAA_{n,r}(J,r)$ of $\AAA_{n,r}$ descend to vector space bases of certain quotient rings $S/(I_{n,r}:f_{J,r})$ (Proposition~\ref{prop:colon-ideal-quotient-basis}). Conjecture~\ref{conj:SS} has already been proven for the symmetric group. 

\begin{theorem}
    \label{thm:sn-basis}
    {\em (Angarone--Commins--Karn--Murai--R. \cite{ACKMR})}
    Conjecture~\ref{conj:SS} is true when $r=1$.
\end{theorem}

When $r = 2$, the first author \cite{Bhattacharya} gave a non-monomial basis of $SR_{n,2}$. We will use Theorem~\ref{thm:sn-basis} in the proof of Conjecture~\ref{conj:SS}.

\subsection{Spanning} There is no available elementary proof that $\AAA_{n,1}$ spans or is linearly independent in $SR_{n,1}$. The proof of Theorem~\ref{thm:sn-basis} in \cite{ACKMR} uses {\em Solomon--Terao algebras} of hyperplane arrangements \cite{AMMN}. Starting from Theorem~\ref{thm:sn-basis}, we will use more elementary reasoning to deduce the fact that $\AAA_{n,r}$ spans $SR_{n,r}$. 
We introduce a ring embedding to transfer between the understood case of $r=1$ and the general setting of $r \geq 1$.

\begin{definition}
\label{def:expansion-homomorphism}
Let $\varphi_r: \Omega \to \Omega$ be the `expansion homomorphism' defined on generators by
\begin{equation}
    \varphi_r: x_i \mapsto x_i^r \quad \text{and} \quad \varphi_r: \theta_i \mapsto x_i^{r-1} \cdot \theta_i.
\end{equation}
\end{definition}

It is easy to see that the assignments $\varphi_r(x_i)$ and $\varphi_r(\theta_i)$ given by Definition~\ref{def:expansion-homomorphism} satisfy
\begin{equation}
    \varphi_r(x_i) \varphi_r(x_j) = \varphi_r(x_j) \varphi_r(x_i) \quad 
    \varphi_r(x_i) \varphi_r(\theta_j) = \varphi_r(\theta_j) \varphi_r(x_i)  \quad \varphi_r(\theta_i) \varphi_r(\theta_j) = - \varphi_r(\theta_j) \varphi_r(\theta_i)
\end{equation}
for all $1 \leq i,j  \leq n$.
In particular, Definition~\ref{def:expansion-homomorphism} gives a well-defined $\CC$-algebra homomorphism $\varphi_r: \Omega \to \Omega$. The action of $\varphi_r$ on power sums $p_i$ and their images $dp_i$ under the Euler derivation $d$ is given by
\begin{equation}
    \varphi_r: p_i \mapsto p_{ri} \quad \text{and} \quad \varphi_r: dp_i \mapsto \frac{1}{r} \cdot dp_{ri}
\end{equation}
for all $i \geq 1$. A glance at the generating sets of $SI_{n,1}$ and $SI_{n,r}$ in \eqref{eq:power-sum-generators} yields $\varphi_r(SI_{n,1}) \subseteq SI_{n,r}$. We will use the map $\varphi_r$ to deduce spanning for general $r \geq 1$ from the case $r=1$.

It will be useful to identify superspace monomials $m \in \Omega$ which may be expressed as a product $m = \varphi_r(m') \cdot m''$ where $m'$ is a superspace monomial and $m'' \in S$ is a monomial in the $x$-variables alone. Consider a superspace monomial $m = (x_1^{a_1} \cdots x_n^{a_n}) \cdot \theta_J$ where $J \subseteq [n]$.  An index $1 \leq i \leq n$ is {\em $r$-low in $m$} if
\begin{itemize}
    \item we have $i \in J$ so that $\theta_i$ appears in $m$, and
    \item the exponent $a_i$ of $x_i$ satisfies $a_i < r-1$.
\end{itemize}
Write $\Low_r(m)$ for the set of low indices and $\low_r(m)$ for the number of low indices, i.e.
\begin{equation}
    \Low_r(m) := \{ 1 \leq i \leq n \,:\, i \text{ is $r$-low in $m$} \}, \quad \quad \low_r(m) := \# \Low_r(m).
\end{equation}

\begin{example}
    \label{ex:low-example}
    Suppose $n = 6$ and $r = 3$ and let $m$ be the superspace monomial
\[
m = x_1^0 x_2^1 x_3^4 x_4^1 x_5^2 x_6^0  \cdot \theta_2 \theta_5 \theta_6. 
\]
We have 
\[\Low_3(m) = \{2,6\} \quad \text{so that} \quad \low_3(m) = 2.\]
\end{example}

Observe that $\low_1 \equiv 0$ in the case $r=1$, so the statistic $\low_r(m)$ is uninteresting for the symmetric group. The statistic $\low_r(m)$ will facilitate the induction argument showing that $\AAA_{n,r}$ spans $SR_{n,r}$ (Lemma~\ref{lem:spanning}). The key observation is that $\low_r(m)$ measures the failure of $m$ to lie in the image of $\varphi_r$ up to multiplication by a purely bosonic monomial.

\begin{observation}
    \label{obs:low-image}
    Let $m \in \Omega_n$ be a superspace monomial. We have $\low_r(m) = 0$ if and only if 
    \[
    m = x_1^{a_1} \cdots x_n^{a_n} \cdot \varphi_r(m')
    \]
    for some $a_1, \dots, a_n \geq 0$ and some superspace monomial $m' \in \Omega$.
\end{observation}

Given $J \subseteq [n]$ and $r \geq 1$, we want to understand how the sets of superspace monomials $\varphi_r(\AAA_{n,1}(J,1))$ and $\AAA_{n,r}(J,r)$ are related. An example will show the general pattern.

\begin{example}
    \label{ex:expansion}
    As in Examples~\ref{ex:staircase} and \ref{ex:staircase-two}, we
    let $n = 9, r = 3,$ and $J = \{3,4,6,9\}$. The $1$-staircase associated to $J$ is 
\[
\stair(J,1) = (0,1,\underline{1},\underline{1},2,\underline{2},3,4,\underline{4})
\]
where the underlined numbers are in positions indexed by $J$. Applying $\varphi_3$ to monomials in the set $\AAA_{n,1}(J,1) \cdot \theta_J$ yields
\[
\varphi_3(\AAA_{n,1}(J,1) \cdot \theta_J) = 
\{ x_1^{a_1} \cdots x_9^{a_9} \,:\, (a_1,\dots,a_9) \leq (0,3,5,5,6,8,9,12,14)  \} \cdot \theta_J
\]
where the inequality $\leq$ is componentwise. On the other hand, one has 
\[
\stair(J,3) = (2,5,\underline{7},\underline{7},8,\underline{10},11,14,\underline{16})
\]
so that
\[
\AAA_{n,3}(J,3) \cdot \theta_J = \{ x_1^{b_1} \cdots x_n^{b_n} \,:\, (b_1,\dots,b_n) \leq (2,5,7,7,8,10,11,14,16) \} \cdot \theta_J.
\]
We have the vector addition
\[
(2,5,7,7,8,10,11,14,16) = (0,3,5,5,6,8,9,12,14) + (2,2,2,2,2,2,2,2,2).
\]
\end{example}

Example~\ref{ex:expansion} generalizes as follows.

\begin{observation}
    \label{obs:addition}
    Let $J \subseteq [n]$ and assume $1 \notin J$. One has 
    \[ \AAA_{n,r}(J,r) \cdot \theta_J = 
    \varphi_r(\AAA_{n,1}(J,1) \cdot \theta_J) \times \{x_1^{c_1} \cdots x_n^{c_n} \,:\, c_1, \dots, c_n \leq r-1 \}
    \]
    where the notation $\{-\} \times \{-\}$ stands for all possible products of a monomial in the first set with a monomial in the second set.
\end{observation}

If $1 \in J$ we have $\AAA_{n,1}(J,1)  =\varnothing$ and Observation~\ref{obs:addition} does not hold. We are ready to prove that $\AAA_{n,r}$ descends to a spanning set of $SR_{n,r}$.

\begin{lemma}
    \label{lem:spanning}
    The set of monomials $\AAA_{n,r}$ descends to a spanning set of $SR_{n,r}$.
\end{lemma}

\begin{proof}
    Let $J \subseteq [n]$ and fix a superspace monomial 
    \begin{equation}m = x_1^{a_1} \cdots x_n^{a_n} \cdot \theta_J \in \Omega_n\end{equation} with bosonic exponents $a_1,\dots,a_n \geq 0$ and fermionic product $\theta_J$. We show that $m$ lies in the span of $\AAA_{n,r}$ modulo $SI_{n,r}$.

    We use the low indices of $m$ to factor the superspace ring $\Omega$.  Consider the partition $[n] = T_0 \sqcup T_1$ where
    \begin{equation}
        T_0 := \Low_r(m) \quad \text{and} \quad T_1 := [n] - \Low_r(m).
    \end{equation}
    For $i = 0,1$ we write $\Omega_{T_i}$ for the superspace ring over variables indexed by $T_i$. That is, we have
    \begin{equation}
        \Omega_{T_0} := \CC[x_i \,:\, i \in T_0] \otimes \wedge \{ \theta_i \,:\, i \in T_0 \} \quad \text{and} \quad 
        \Omega_{T_1}  = \CC[x_i \,:\, i \in T_1 ] \otimes \wedge \{ \theta_i \,:\, i \in T_1\}.
    \end{equation}
     We define ideals 
    \begin{equation}
        SI_{T_0,r} \subseteq \Omega_{T_0} \quad \text{and} \quad SI_{T_1,r} \subseteq \Omega_{T_1}
    \end{equation}
    in the obvious way.  The sets $T_i$ have a total order inherited from $[n]$. This gives rise to sets of monomials
    \begin{equation}
        \AAA_{T_0,r} \subseteq \Omega_{T_0} \quad \text{and} \quad \AAA_{T_1,r} \subseteq \Omega_{T_1}.
    \end{equation}

    The ring $\Omega$ factors as $\Omega = \Omega_{T_0} \otimes \Omega_{T_1}$.
    Recalling that every low index has a corresponding $\theta$-variable, the monomial $m \in \Omega$ has decomposition $m = \pm m_0 \cdot  m_1$ where 
    \begin{equation}
        m_0 = \prod_{i \in T_0} x_i^{a_i} \times \theta_{T_0} \in \Omega_{T_0} \quad \text{and} \quad 
        m_1 = \prod_{i \in T_1} x_i^{a_i} \times \theta_{T_1 \cap J} \in \Omega_{T_1}
    \end{equation}
    and the sign $\pm$ depends on the sets $T_0 \subseteq J$.
    Since $a_i < r-1$ for all $i \in T_0$, we have
    \begin{equation}
    \label{eq:SS-containment}
        m_0 \cdot \AAA_{T_1,r} \subseteq \AAA_{n,r},
    \end{equation}
    where in the containment \eqref{eq:SS-containment} we consider superspace monomials up to sign. We have 
    \begin{equation}
        \Low_r(m_0) = \Low_r(m) =  T_0 \quad \text{and} \quad \Low_r(m_1) = \varnothing.
    \end{equation}
    For all $i \in T_1$, apply the division algorithm to obtain integers $q_i \geq 0$ and $0 \leq c_i \leq r-1$ such that
    \begin{align}
        a_i &= q_i r + c_i \quad (i \in T_1 - J), \\
        a_i-r+1 &= q_i r + c_i \quad (i \in T_1 \cap J).
    \end{align}
    Since $\low_r(m_1) = 0$, we know $a_i \geq r-1$ for all $i \in T_1 \cap J$, so we indeed have $q_i \geq 0$ for all $i$.

    Let $m_1' \in \Omega_{T_1}$ be the superspace monomial given by
    \begin{equation}
        m_1' := \prod_{i \in T_1} x_i^{q_i} \cdot \theta_{T_1 \cap J}.
    \end{equation}
    As in Observation~\ref{obs:low-image}, the monomials $m_1'$ and $m_1$ are related under the ring map $\varphi_r$ by
    \begin{equation}
        m_1 = \varphi_r(m_1') \times \prod_{i \in T_1} x_i^{c_i}.
    \end{equation}
    Theorem~\ref{thm:sn-basis} furnishes constants $\alpha_b \in \CC$ and superspace elements $f_j,g_j \in \Omega_{T_1}$ so that 
    \begin{equation}
    \label{eq:span-one}
        m_1' = \sum_{b \in \AAA_{T_1,1}} \alpha_b \cdot b + \sum_{j=1}^{\# T_1} f_j \cdot p_j(x_i \,:\, i \in T_1) + \sum_{j=1}^{\# T_1} g_j \cdot dp_j(x_i \,:\, i \in T_1).
    \end{equation}
    Applying $\varphi_r$ to both sides of Equation~\eqref{eq:span-one} and multiplying by $\prod_{i \in T_1} x_i^{c_i}$ results in
    \begin{multline}
        \label{eq:span-two}
        m_1 = \varphi_r(m_1') \cdot \prod_{i \in T_1} x_i^{a_i} 
         = 
        \sum_{b \in \AAA_{T_1,1}} \alpha_b \cdot \left[ \varphi_r(b) \cdot \prod_{i \in T_1} x_i^{c_i} \right] \\  + \sum_{j=1}^{\# T_1} \tilde{f}_j \cdot p_{rj}(x_i \,:\, i \in T_1)  + \sum_{j=1}^{\# T_1}  \tilde{g}_j \cdot dp_{rj}(x_i \,:\, i \in T_1)
    \end{multline}
    where
    \begin{equation}
        \tilde{f}_j = \prod_{i \in T_1} x_i^{c_i} \times \varphi_r(f_j) \quad \text{and} \quad \tilde{g}_j = \frac{1}{r} \prod_{i \in T_1} x_i^{c_i} \times \varphi_r(g_j).
    \end{equation}
    Next, we multiply both sides of \eqref{eq:span-two} by $m_0 = \prod_{i \in T_0} x_i^{a_i} \cdot \theta_{T_0}$. Up to an immaterial sign which does not affect the spanning statement we wish to prove, this results in
    \begin{multline}
        \label{eq:span-three}
        \pm m = m_1 \cdot m_0  = 
        \sum_{b \in \AAA_{T_1,1}} \alpha_b \cdot \left[ \varphi_r(b) \cdot \prod_{i \in T_1} x_i^{c_i} \cdot m_0 \right] 
         \\ + \sum_{j=1}^{\# T_1} \tilde{f}_j \cdot p_{rj}(x_i \,:\, i \in T_0) \cdot m_0  + \sum_{j=1}^{\# T_1}  \tilde{g}_j \cdot dp_{rj}(x_i \,:\, i \in T_1) \cdot m_0.
    \end{multline}
    Since $\theta_i^2 = 0$ for all $1 \leq i \leq n$, Equation~\eqref{eq:span-three} may be rewritten using \[p_{rj}(x_1,\dots,x_n) = p_{rj}(x_i \,:\, i \in T_0) + p_{rj}(x_i \,:\, i \in T_1)\] as 
    \begin{multline}
        \label{eq:span-four}
        \pm m = m_1 \cdot m_0  = 
        \sum_{b \in \AAA_{T_1,1}} \alpha_b \cdot \left[ \varphi_r(b) \cdot \prod_{i \in T_1} x_i^{c_i} \cdot m_0 \right] - \sum_{j=1}^{\# T_1} \tilde{f}_j \cdot p_{rj}(x_i \,:\, i \in T_1) \cdot m_0 
         \\ + \sum_{j=1}^{\# T_1} \tilde{f}_j \cdot p_{rj}(x_1, \dots, x_n) \cdot m_0  + \sum_{j=1}^{\# T_1}  \tilde{g}_j \cdot dp_{rj}(x_1,\dots x_n) \cdot m_0.
    \end{multline}

The two sums on the lower line of Equation~\eqref{eq:span-four} are contained in the ideal $SI_{n,r}$ so that 
\begin{multline}
        \label{eq:span-five}
        \pm m = m_1 \cdot m_0 \\ \equiv
        \sum_{b \in \AAA_{T_1,1}} \alpha_b \cdot \left[ \varphi_r(b) \cdot \prod_{i \in T_1} x_i^{c_i} \cdot m_0 \right] 
        - \sum_{j=1}^{\# T_1} \tilde{f}_j \cdot p_{rj}(x_i \,:\, i \in T_0) \cdot m_0  \mod SI_{n,r}.
\end{multline}
For any $b \in \AAA_{T_1,1}$, Observation~\ref{obs:addition} implies that $\varphi_r(b) \cdot \prod_{i \in T_1} x_i^{c_i} \in \AAA_{T_1,r}$. Furthermore, since the monomial $m_0 = \prod_{i \in T_0} x_i^{a_i} \times \theta_{T_0}$ satisfies $a_i < r-1$ for all $i \in T_0$, we have $\varphi_r(b) \cdot \prod_{i \in T_1} x_i^{c_i} \cdot m_0 \in \AAA_{n,r}.$ The first sum on the lower line of \eqref{eq:span-five} is therefore a $\CC$-linear combination of elements of $\AAA_{n,r}$.

We use \eqref{eq:span-five} to prove that $m$ lies in the span of $\AAA_{n,r}$ modulo $SI_{n,r}$ by induction on $\low_r(m)$.  If $\low_r(m) = 0$ we have $T_0 = \varnothing$ so that $p_{rj}(x_i \,:\, i \in T_0) = 0$ for $i = 1, 2, \dots, \# T_1$. The second term on the lower line of \eqref{eq:span-five} vanishes in this case and we deduce that  $m$ lies in the span of $\AAA_{n,r}$ modulo $SI_{n,r}$ as desired. 

Suppose $\low_r(m) > 0$. We examine the second term on the lower line of \eqref{eq:span-five} in detail. Let $1 \leq j \leq \# T_1$. Then $\tilde{f}_j = \prod_{i \in T_i} x_i^{c_i} \times \varphi_r(f_j)$ where $f_j \in \Omega_{T_1}$. In particular, we have $\tilde{f}_j \in \Omega_{T_1}$ and (by Observation~\ref{obs:low-image}) every superspace monomial $\hat{m}$ appearing in the expansion of $\tilde{f}_j$ satisfies $\low_r(\hat{m}) = 0.$ On the other hand, we have $m_0 \in \Omega_{T_0}$ and $\low_r(m_0) = \# T_0$ is as large as possible. Since $j \geq 1$, every monomial $\hat{m}'$ appearing in the expansion of $p_{rj}(x_i \,:\, i \in T_0) \cdot m_0 = \sum_{i \in T_0} x_i^{rj} \cdot m_0$ satisfies $\low_r(\hat{m}') = \low_r(m_0) - 1$. We deduce that every monomial $\hat{m}''$ appearing in the expansion of $\tilde{f}_j \cdot p_{rj}(x_i \,:\, i \in T_0) \cdot m_0$ satisfies $\low_r(\hat{m}'') = \low_r(m_0) - 1 = \low_r(m) - 1$. We inductively conclude that 
\begin{equation}
    \sum_{j = 1}^{\# T_1} \tilde{f}_j \cdot p_{rj}(x_i \,:\, i \in T_0) \cdot m_0 \in \mathrm{span}_\CC (\AAA_{n,r}) \mod SI_{n,r}
\end{equation}
so that \eqref{eq:span-five} forces 
\begin{equation}
    m \in \mathrm{span}_\CC(\AAA_{n,r}) \mod SI_{n,r}
\end{equation}
as required.
\end{proof}

\subsection{The inverse system $SI_{n,r}^\perp$}\label{subsec: inverse system}
We want to upgrade the spanning result of Lemma~\ref{lem:spanning} to prove that $\AAA_{n,r}$ descends to a vector space basis of $SR_{n,r}$. We achieve this by bounding the (bigraded) dimension of $SR_{n,r}$ from below. Thanks to the equality $\Hilb(SR_{n,r};q,z) = \Hilb(SI_{n,r}^\perp; q,z)$, it is equivalent to bound the inverse system $SI_{n,r}^\perp$ from below; this section provides the lemmata necessary to do this.

We begin with the inverse system $I_{n,r}^\perp \subseteq S$ of the polynomial coinvariant quotient $S/I_{n,r}$ associated to $\ZZ_r \wr \symm_n$. Let $\delta_{n,r} \in S$ be the polynomial
\begin{equation}
\label{eq:vandermonde}
    \delta_{n,r} := (x_1 \cdots x_n)^{r-1} \times \prod_{1 \leq i < j \leq n} (x_i^r - x_j^r).
\end{equation}
When $r = 1$, this is the classical Vandermonde determinant. The following classical result of Steinberg \cite{Steinberg} states that $\delta_{n,r}$ generates the inverse system $I_{n,r}^\perp$ under the $\odot$-action of partial differentiation.

\begin{theorem}
    \label{thm:steinberg}
    {\em (Steinberg \cite{Steinberg})}
    The inverse system $I_{n,r}^\perp \subseteq S$ is the smallest linear subspace containing $\delta_{n,r}$ which is closed under the partial derivatives $\partial_1, \dots, \partial_n$.
\end{theorem}

In order to describe the inverse system  $SI_{n,r}^\perp \subseteq \Omega$ we will need more operators. For $j \geq 1$, let $d_j: \Omega \to \Omega$ be the {\em higher Euler operator} given by
\begin{equation}
    d_j(f) := \sum_{i=1}^n \partial_i^j f \cdot \theta_i.
\end{equation}
When $j = 1$ this is the classical Euler operator. In general, the operator $d_j$ lowers bosonic degree by $j$ while raising fermionic degree by $1$. The $d_j$ operators satisfy $d_{j_1} d_{j_2} = -d_{j_2} d_{j_1}$ for all $j_1,j_2 > 0$. Given $J \subseteq n$, we write
\begin{equation}
    d_J := d_{j_1} \cdots d_{j_k} \quad \text{where $J = \{j_1 < \cdots < j_k \}$}   
\end{equation}
for the corresponding composition of $d_j$-operators.
Swanson and Wallach \cite{SW} proved that certain $d_j$ operators preserve the inverse system $SI_{n,r}^\perp$.

\begin{theorem}
    \label{thm:swanson-wallach} {\em (Swanson--Wallach \cite{SW})}
    Suppose $j \equiv 1 \mod r$. The operator $d_j: \Omega \to \Omega$ preserves the inverse system $SI_{n,r}^\perp$. 
\end{theorem}

We will see  (Theorem~\ref{thm:operator}) that the $d_j$ with $j \equiv 1 \mod r$ are the only extra ingredient needed to generate the inverse system $SI_{n,r}^\perp$. To this end, we use these $d_j$ to construct certain elements of $SI_{n,r}^\perp$ with favorable triangularity properties with respect to fermionic monomials. Our construction of these inverse system elements will be slightly indirect.

Let $\yy = (y_1, y_2, \dots )$ be a new set of variables. We define two matrices as follows. The $k$-row {\em power matrix} $P_{k,r}(\yy)$ is the $k \times n$ matrix whose $(i,j)$-entry is
\begin{equation}
P_{k,r}(\yy)_{i,j} := y_i^{r \cdot (n-j) + 1}.
\end{equation}
The $k$-row {\em factor matrix} $F_{k,r}(\xx,\yy)$ is the $k \times n$ matrix whose $(i,j)$-entry is 
\begin{equation}
    F_{k,r}(\xx,\yy)_{i,j} := y_i \cdot (y_i^r - x_{j+1}^r)(y_i^r - x_{j+2}^r) \cdots (y_i^r - x_n^r).
\end{equation}

\begin{example}
\label{ex:power-factor}
Suppose $n = 5$, $k = 3$, and $r=4$. The power matrix is given by
\[
P_{3,4}(\yy) = 
\begin{footnotesize}
\begin{pmatrix}
    y_1^{17} & y_1^{13} & y_1^9 & y_1^5 & y_1 \\
     y_2^{17} & y_2^{13} & y_2^9 & y_2^5 & y_2 \\
    y_3^{17} & y_3^{13} & y_3^9 & y_3^5 & y_3
\end{pmatrix}\end{footnotesize}
\]
and the factor matrix $F_{3,4}(\xx,\yy)$ is given by
\[  
\begin{footnotesize}
\begin{pmatrix}
    y_1 (y_1^4 - x_2^4)(y_1^4 - x_3^4)(y_1^4 - x_4^4)(y_1^4 - x_5^4) &
    y_1 (y_1^4 - x_3^4)(y_1^4 - x_4^4)(y_1^4 - x_5^4) & 
    y_1 (y_1^4 - x_4^4)(y_1^4 - x_5^4) &
    y_1 (y_1^4 - x_5^4) &y_1
    \\
     y_2 (y_2^4 - x_2^4)(y_2^4 - x_3^4)(y_2^4 - x_4^4)(y_2^4 - x_5^4) &  
     y_2 (y_2^4 - x_3^4)(y_2^4 - x_4^4)(y_2^4 - x_5^4) & 
    y_2 (y_2^4 - x_4^4)(y_2^4 - x_5^4) &
    y_2 (y_2^4 - x_5^4) &y_2\\
      y_3 (y_3^4 - x_2^4)(y_3^4 - x_3^4)(y_4^4 - x_4^4)(y_3^4 - x_5^4) & 
      y_3 (y_3^4 - x_3^4)(y_3^4 - x_4^4)(y_3^4 - x_5^4) & 
      y_3 (y_4^4 - x_4^4)(y_3^4 - x_5^4) &
      y_3 (y_3^4 - x_5^4) &y_3
\end{pmatrix}
\end{footnotesize}.
\]
\end{example}

The matrices $P_{3,4}(\yy)$ and $F_{3,4}(\xx,\yy)$ are related in the following way.

\begin{observation}
    \label{obs:matrix-mult}
    There exists a lower triangular $n \times n$ matrix $C$ with 1s on the diagonal and entries in $S = \CC[x_1,\dots,x_n]$ so that 
    \begin{equation}
        F_{k,r}(\xx,\yy) = P_{k,r}(\yy) \cdot C.
    \end{equation}
\end{observation}

Observation~\ref{obs:matrix-mult} is similar to \cite[Obs. 4.8]{MRW}. In the above case of $n = 5, k = 3,$ and $r=4$, the matrix $C$ implements the following sequence column operations on $P_{3,4}(\yy)$. Let $e_j$ be the elementary symmetric polynomial of degree $j$.
\begin{enumerate}
    \item We subtract $e_1(x_2^4, x_3^4, x_4^4, x_5^4)$ times column 2 from column 1, then add $e_2(x_2^4, x_3^4, x_4^4, x_5^4)$ times column 3 to column 1, then subtract $e_3(x_2^4, x_3^4, x_4^4, x_5^4)$ times column 4 from column 1, then add $e_4(x_2^4, x_3^4, x_4^4, x_5^4)$ times column 5 to column 1.
    \item We subtract $e_1(x_3^4, x_4^4, x_5^4)$ times column 3 from column 2, then add $e_2(x_3^4, x_4^4, x_5^4)$ times column 4 to column 2, then subtract $e_3(x_3^4, x_4^4, x_5^4)$ times column 5 from column 2.
    \item We subtract $e_1(x_4^4, x_5^4)$ times column 4 from column 3, then add $e_2(x_4^4, x_5^4)$ times column 5 to column 3.
    \item We subtract $e_1(x_5^4)$ times column 5 from column 2.
\end{enumerate}
Since $C$ adds and subtracts multiples of subsequent columns from previous columns, we see that $C$ is lower triangular with 1s on the diagonal.

For $J \subseteq [n]$ with $\# J = k$, let $E_J$ be the $(n-k) \times n$ echelon matrix with 0,1-entries whose pivot columns are those indexed by $[n] - J$. For example, if $n = 5$ and $J = \{1,3,4\}$ then
\[
E_J = \begin{pmatrix} 
0 & 1 & 0 & 0 & 0 \\
0 & 0 & 0 & 0 & 1
\end{pmatrix}.
\]
The {\em augmented factor matrix} $\widetilde{F}_{k,r}(\xx,\yy,J)$ is the $n \times n$ matrix  with block form
\begin{equation}
    \widetilde{F}_{k,r}(\xx,\yy,J) = \left(\begin{array}{c}
         F_{k,r}(\xx,\yy) \\ \hline E_J
    \end{array} \right).
\end{equation}
When $n = 5, r = 4,$ and $J = \{1,3,4\}$ the matrix $\widetilde{F}_{k,r}(\xx,\yy,J)$ is given by
\[  
\begin{footnotesize}
\left(
\begin{array}{ccccc}
    y_1 (y_1^4 - x_2^4)(y_1^4 - x_3^4)(y_1^4 - x_4^4)(y_1^4 - x_5^4) &
    y_1 (y_1^4 - x_3^4)(y_1^4 - x_4^4)(y_1^4 - x_5^4) & 
    y_1 (y_1^4 - x_4^4)(y_1^4 - x_5^4) &
    y_1 (y_1^4 - x_5^4) &y_1
    \\
     y_2 (y_2^4 - x_2^4)(y_2^4 - x_3^4)(y_2^4 - x_4^4)(y_2^4 - x_5^4) &  
     y_2 (y_2^4 - x_3^4)(y_2^4 - x_4^4)(y_2^4 - x_5^4) & 
    y_2 (y_2^4 - x_4^4)(y_2^4 - x_5^4) &
    y_2 (y_2^4 - x_5^4) &y_2\\
      y_3 (y_3^4 - x_2^4)(y_3^4 - x_3^4)(y_4^4 - x_4^4)(y_3^4 - x_5^4) & 
      y_3 (y_3^4 - x_3^4)(y_3^4 - x_4^4)(y_3^4 - x_5^4) & 
      y_3 (y_4^4 - x_4^4)(y_3^4 - x_5^4) &
      y_3 (y_3^4 - x_5^4) &y_3 \\ \hline
      0 & 1 & 0 & 0 & 0 \\
      0 & 0 &0 & 0 & 1
\end{array}
\right)
\end{footnotesize}.
\]

For $U \subseteq [n]$ with $\# U = \# J =  k$ we  write $\widetilde{F}_{k,r}(\xx,\xx_U, J)$ for the matrix obtained from $\widetilde{F}_{k,r}(\xx,\yy,J)$ by substituting the variables $\{x_u \,:\, u \in U \}$ for $y_1, y_2, \dots, y_k$. We are only concerned with the up-to-sign determinant of $\widetilde{F}_{k,r}(\xx,\xx_U, J)$, so the order of this substitution does not matter.
For $J \subseteq [n]$, let $f_{J,r} \in S$ be the polynomial
\begin{equation}
    f_{J,r} := \prod_{j \in J} x_j \times \prod_{j < i \leq n} (x_j^r - x_i^r) \in S.
\end{equation}
Similar polynomials were introduced in \cite{RW} when $r = 1$ and \cite{Bhattacharya} when $r= 2$. For example, if $n = 5, r= 4,$ and $J = \{1,3,4\}$ we have
\[
f_{J,r} = x_1(x_1^4 -x_2^4)(x_1^4 - x_3^4)(x_1^4 - x_4^4)(x_1^4 - x_5)^5 \times x_3 (x_3^4 - x_4^4)(x_3^4 - x_5^4) \times x_4 (x_4^4 - x_5^4)
\]
with one factor for each element of $J$.

\begin{lemma}
    \label{lem:augmented-determinant}
    For $U \subseteq [n]$ with $\# U = \# J = k$, let $\PPP_{J,U} \in S$ be the polynomial defined up to sign by
    \begin{equation}
        \PPP_{J,U} := \det \widetilde{F}_{k,r}(\xx,\xx_U,J).
    \end{equation}
    \begin{enumerate}
        \item If $U \not\leq_{\Gale} J$ we have $\PPP_{J,U} = 0$.
        \item We have $\PPP_{J,J} = \pm f_{J,r}$.
    \end{enumerate}
\end{lemma}

Lemma~\ref{lem:augmented-determinant} and its proof are similar to \cite[Lem. 4.9]{MRW}.

\begin{proof}
    Let $\overline{F}_{k,r}(\xx,\xx_U,J)$ be the maximal $k \times k$ submatrix of the upper block $F_{k,r}(\xx,\xx_U,J)$  of $\widetilde{F}_{k,r}(\xx,\xx_U,J)$  consisting of the columns indexed by $J$. We have
    \begin{equation}
    \label{eq:P-reformulation}
        \PPP_{J,U} = \det \widetilde{F}_{k,r}(\xx,\xx_U,J) = \pm \det\overline{F}_{k,r}(\xx,\xx_U,J).
    \end{equation}

    (1) If $U \not\leq_{\Gale} J$, after rearranging the rows of $\widetilde{F}_{k,r}(\xx,\xx_U,J)$ we obtain a matrix with block form
    \[
    \begin{pmatrix}
        {\bf *} & {\bf *} \\ {\bf 0} & {\bf *}
    \end{pmatrix}
    \]
    where the block of zeros intersects the main diagonal. By \eqref{eq:P-reformulation} we obtain $\PPP_{J,U} = 0$.

    (2) Up to a rearrangement of rows, the matrix $\overline{F}_{k,r}(\xx,\xx_J,J)$ is upper triangular with diagonal entries
    \begin{equation}
        \left\{ x_j \cdot \prod_{j < i \leq n} (x_j^r - x_i^r) \,:\, j \in J  \right\}.
    \end{equation}
    Since these are exactly the factors of $f_{J,r}$, we conclude from \eqref{eq:P-reformulation} that $\PPP_{J,J} = \pm f_{J,r}$.
\end{proof}

We use our running example $n = 5, r = 4$, $J = \{1,3,4\}$ to illustrate Lemma~\ref{lem:augmented-determinant}. Suppose $U = \{a < b < c\} \subseteq [5]$. Then $\PPP_{J,U} \in S$ is the $5 \times 5$ determinant
\[  \det
\begin{tiny}
\left(
\begin{array}{ccccc}
    x_a (x_a^4 - x_2^4)(x_a^4 - x_3^4)(x_a^4 - x_4^4)(x_a^4 - x_5^4) &
    x_a (x_a^4 - x_3^4)(x_a^4 - x_4^4)(x_a^4 - x_5^4) & 
    x_a (x_a^4 - x_4^4)(x_a^4 - x_5^4) &
    x_a (x_a^4 - x_5^4) & x_a
    \\
     x_b (x_b^4 - x_2^4)(x_b^4 - x_3^4)(x_b^4 - x_4^4)(x_b^4 - x_5^4) &  
     x_b (x_b^4 - x_3^4)(x_b^4 - x_4^4)(x_b^4 - x_5^4) & 
    x_b (x_b^4 - x_4^4)(x_b^4 - x_5^4) &
    x_b (x_b^4 - x_5^4) & x_b\\
      x_c (x_c^4 - x_2^4)(x_c^4 - x_3^4)(x_c^4 - x_4^4)(x_c^4 - x_5^4) & 
      x_c (x_c^4 - x_3^4)(x_c^4 - x_4^4)(x_c^4 - x_5^4) & 
      x_c (x_c^4 - x_4^4)(x_c^4 - x_5^4) &
      x_c (x_c^4 - x_5^4) & x_c \\ \hline
      0 & 1 & 0 & 0 & 0 \\
      0 & 0 &0 & 0 & 1
\end{array}
\right)
\end{tiny}
\]
of the matrix $\widetilde{F}_{k,r}(\xx,\xx_U,J)$ obtained from $\widetilde{F}_{k,r}(\xx,\yy,J)$ by substituting $(y_1,y_2,y_3) \to (x_a,x_b,x_c)$. Expanding along the bottom two rows shows that $\PPP_{J,U}$ is up to sign\footnote{which is $+1$ in this case} the $3 \times 3$ determinant
\[
\PPP_{J,U} = \pm
\begin{footnotesize} 
\det
\begin{pmatrix}
    x_a (x_a^4 - x_2^4)(x_a^4 - x_3^4)(x_a^4 - x_4^4)(x_a^4 - x_5^4) & 
     x_a (x_a^4 - x_4^4)(x_a^4 - x_5^4) &
       x_a (x_a^4 - x_5^4)  \\
    x_b (x_b^4 - x_2^4)(x_b^4 - x_3^4)(x_b^4 - x_4^4)(x_b^4 - x_5^4) & 
     x_b (x_b^4 - x_4^4)(x_b^4 - x_5^4) &
       x_b (x_b^4 - x_5^4)  \\

       x_c (x_c^4 - x_2^4)(x_c^4 - x_3^4)(x_c^4 - x_4^4)(x_c^4 - x_5^4) & 
     x_c (x_c^4 - x_4^4)(x_c^4 - x_5^4) &
       x_c (x_c^4 - x_5^4)  \\
\end{pmatrix}
\end{footnotesize}.
\]
If $\{a < b < c \} \not\leq_\Gale \{1,3,4\}$, we must have
\[
a > 1, \quad b > 3, \quad \text{or} \quad c > 4
\]
and this $3 \times 3$ determinant vanishes for support reasons. On the other hand, if $J  =\{1,3,4\}$ we have
\begin{align*}
\PPP_{J,U} &= \pm
\begin{footnotesize} 
\det
\begin{pmatrix}
    x_1 (x_1^4 - x_2^4)(x_1^4 - x_3^4)(x_1^4 - x_4^4)(x_1^4 - x_5^4) & 
     x_1 (x_1^4 - x_4^4)(x_1^4 - x_5^4) &
       x_1 (x_1^4 - x_5^4)  \\
    0 & 
     x_3 (x_3^4 - x_4^4)(x_3^4 - x_5^4) &
       x_3 (x_3^4 - x_5^4)  \\
      0 & 
     0 &
       x_4 (x_4^4 - x_5^4)  \\
\end{pmatrix}
\end{footnotesize} \\
&= \pm x_1 (x_1^4 - x_2^4)(x_1^4 - x_3^4)(x_1^4 - x_4^4)(x_1^4 - x_5^4) \times x_3 (x_3^4 - x_4^4)(x_3^4 - x_5^4) \times 
x_4 (x_4^4 - x_5^4) \\
&= \pm f_{134,4}.
\end{align*}

In order to define our desired elements of $SI_{n,r}^\perp$, we need a couple more ideas. Let $J \subseteq [n]$ satisfy $\# J = k$. Observation~\ref{obs:matrix-mult} furnishes an $n \times n$ matrix $C \in GL_n(S)$ so that $F_{k,r}(\xx,\yy) = P_{k,r}(\yy) \cdot C$. We have the matrix equation
\begin{equation}
    \widetilde{F}_{k,r}(\xx,\yy) \cdot C^{-1} = 
    \left(
    \begin{array}{c}
            F_{k,r}(\xx,\yy) \\ \hline E_J
    \end{array}
    \right) \cdot C^{-1} =
    \left(
    \begin{array}{c}
        P_{k,r}(\yy) \\ \hline H_J
    \end{array}\right)
\end{equation}
where \begin{equation}H_J := E_J \cdot C^{-1}\end{equation} is an $(n-k) \times n$ matrix with entries in $S$ which depends on $J$. Given $K \subseteq [n]$ we write $K^* := \{ n-k+1 \,:\, k \in K \}$ for its `reversal'. We use the following operators to define 

\begin{definition}
    \label{def:D-operator}
    For $J \subseteq [n]$ with $\# J = k$, we define an operator $\DDD_{J,r}: \Omega \to \Omega$ by the formula
    \begin{equation}
        \DDD_{J,r}(f) := \sum_{\# U = n-k} (-1)^{\sum U} \Delta_U(H_J) \odot d_{\rho(U)}(f).
    \end{equation}
    Here $\sum U := \sum_{u \in U}u$, $\Delta_U(H_J)$ is the maximal minor of $H_J$ with columns indexed by $U$, and $\rho(U)$ is the set
    \[
    \rho(U) := \{ (n-u) \cdot r + 1 \,:\, u \in U \}
    \]
    obtained by `reversing with respect to $n$ and rescaling' the set $U$.
\end{definition}

Definition~\ref{def:D-operator} may appear complicated at first glance, but the formula defining $\DDD_{J,r}$ mimics the expansion of the block matrix determinant
\begin{equation} \det( \widetilde{F}_{k,r}(\xx,\xx_U) \cdot C^{-1}) = 
    \det \left(\begin{array}{c}
        P_{k,r}(\xx_U) \\ \hline H_J
    \end{array}\right)
\end{equation}  
in a way that will prove useful in establishing the triangularity result of Lemma~\ref{lem:D-leading} below.
We obtain elements of the inverse system $SI_{n,r}^\perp$ by applying the $\DDD_{J,r}$ operators to the $(\ZZ_r \wr \symm_n)$-Vandermonde polynomial $\delta_{n,r} \in S$. The $\DDD_{J,r}$ are designed so that these elements of $SI_{n,r}^\perp$ have advantagenous triangularity properties.

\begin{lemma}
    \label{lem:D-leading}
    Assume $r \geq 2$. Let $J \subseteq [n]$ and consider the superspace element $\DDD_{J,r}(\delta_{n,r}) \in \Omega$.
    \begin{enumerate}
        \item We have $\DDD_{J,r}(\delta_{n,r}) \in SI_{n,r}^\perp$.
        \item The unique Gale-maximal fermionic monomial which appears in $\DDD_{J,r}(\delta_{n,r})$. Up to sign, this monomial is 
        \[
        (f_{J,r} \odot \delta_{n,r}) \cdot \theta_J.
        \]
    \end{enumerate}
\end{lemma}

\begin{proof}
    (1) Steinberg's Theorem~\ref{thm:steinberg} implies that $\delta_{n,r} \in SI_{n,r}^\perp$. The $d$-operators appearing in the definition of $\DDD_{J,r}$ have indices of the form $(n-u) \cdot r + 1 \equiv 1 \mod r$, so by Theorem~\ref{thm:swanson-wallach} the $\DDD_{J,r}$ preserve $SI_{n,r}^\perp$.

    (2) Let $J' \subseteq [n]$ satisfy $\# J' = \# J$. We have the chain of up-to-sign equalities
    \begin{align}
        \text{coefficient of $\theta_{J'}$ in $\DDD_{J,r}(\delta_{n,r})$} &= \text{coefficient of $\theta_{J'}$ in }
        \sum_{\# U = n-k} (-1)^{\sum U} \Delta_U(H_J) \odot d_{\rho(U)}(\delta_{n,r}) \\
        &= \pm \left[ \det \left(
            \begin{array}{c}
                P_{k,r}(\xx_{J'}) \\ \hline H_J
            \end{array}
        \right) \right] \odot \delta_{n,r} \\
        &= \pm \left[ \det \left(
            \begin{array}{c}
                F_{k,r}(\xx,\xx_{J'}) \\ \hline E_J
            \end{array}
        \right) \right] \odot \delta_{n,r} \\
        &= \pm \PPP_{J,J'} \odot \delta_{n,r}
    \end{align}
    where the third equality follows because the matrix $C$ of Observation~\ref{obs:matrix-mult} satisfies $\det(C) = 1$. Lemma~\ref{lem:augmented-determinant} (1) shows that the coefficient of $\theta_{J'}$ in $\DDD_{J,r}(\delta_{n,r})$ is zero unless $J' \leq_\Gale J$. By Lemma~\ref{lem:augmented-determinant} (2), we have
    \begin{equation}
        \text{coefficient of $\theta_J$ in $\DDD_{J,r}(\delta_{n,r})$} = f_{J,r} \odot \delta_{n,r}.
    \end{equation}
    We will be done if we can establish the nonvanishing property \begin{equation}
    \label{eq:nonvanishing}
        f_{J,r} \odot \delta_{n,r} \neq 0 \quad \quad (J \subseteq [n], \, r \geq 2).
    \end{equation}
    Fixing $r \geq 2$, it is enough to prove the nonvanishing \eqref{eq:nonvanishing} when $J = [n]$. We have
    \begin{equation}\label{eq: f[n]r}
        f_{[n],r} = (x_1 \cdots x_n) \times \prod_{1\leq i < j \leq n} (x_i^r - x_j^r).
    \end{equation}
    If \eqref{eq:nonvanishing} failed, for $J = [n]$, we would have $f_{[n],r} \odot \delta_{n,r} = 0$, but then
    \begin{align}
        \delta_{n,r} \odot \delta_{n,r} &= \left[
            x_1^{r-1} \cdots x_n^{r-1} \times \prod_{1\leq i < j \leq n} (x_i^r - x_j^r)
        \right] \odot \delta_{n,r} \\
        &= (x_1^{r-2} \cdots x_n^{r-2}) \odot \left[ f_{[n],r} \odot \delta_{n,r}
        \right] \\
        &= (x_1^{r-2} \cdots x_n^{r-2}) \odot 0  \\
        &= 0.
    \end{align}
    But $\delta_{n,r} \odot \delta_{n,r}$ is a strictly positive real number and we have a contradiction. We conclude that $f_{[n],r} \odot \delta_{n,r} \neq 0$, the nonvanishing statement \eqref{eq:nonvanishing} is established, and the lemma is proven.
\end{proof}

Steinberg's Theorem~\ref{thm:steinberg} and Lemma~\ref{lem:D-leading} motivate the study of the colon ideals \[(I_{n,r}:f_{J,r}) \subseteq S.\] Even when $I \subseteq S$ is a `nice' ideal, colon ideals $(I:f)$ are typically badly behaved from a commutative algebra point of view. By contrast, the ideals $(I_{n,r}:f_{J,r})$ are complete intersections and relate to staircase combinatorics.

\begin{lemma}
    \label{lem:colon-ideal-lemma}
    Let $r \geq 2$. For any $J \subseteq [n]$, the colon ideal $(I_{n,r}:f_{J,r})$ is generated by a homogeneous regular sequence $p_{J,r,1},\dots,p_{J,r,n} \in S$ with
    \begin{equation}
        \deg(p_{J,r,i}) = \stair(J,r)_i + 1.
    \end{equation}
\end{lemma}

A result similar to Lemma~\ref{lem:colon-ideal-lemma} was proven in the $\symm_n$ case of $r=1$ in \cite{RW}. Lemma~\ref{lem:colon-ideal-lemma} itself was proven when $r=2$ in \cite{Bhattacharya}. 

\begin{proof}
    To start, we construct a regular sequence of homogeneous polynomials $p_{J,r,1}, \dots, p_{J,r,n} \in S$ with $\deg(p_{J,r,i}) = \stair(J,r)_i + 1$ so that $p_{J,r,i} \in (I_{n,r}:f_J)$ for all $i = 1, \dots, n$. Our construction is a generalization of the $r =2$ case found in \cite[Sec. 3]{Bhattacharya}.

    For $j \geq 1$, and any subset $K \subseteq [n]$, let $h_{j,r}(K) \in S$ be the polynomial
    \begin{equation}
        h_{j,r}(K) = \sum_{\substack{i_1 \leq \cdots \leq i_j \\ i_1, \dots, i_j \in K}} x_{i_1}^r  \cdots x_{i_j}^r.
    \end{equation}
    In other words, the polynomial $h_{j,r}(K)$ is the degree $j$ complete homogeneous symmetric polynomial in the variable set $\{x_i^r \,:\, i \in K \}$. It is well-known that $h_{j,r}(K) \in I_{n,r}$ whenever $(\# K) + j > n$.

    For $1 \leq i \leq n$, define an element $q_{J,r,i} \in \Omega$ by the rule
    \begin{equation}
        q_{J,r,i} := \begin{cases}
            h_{j_i,r}(J \cup \{i, i+1,\dots,n\}) \cdot \theta_J & i \notin J, \\
            d h_{j_i,r}(J \cup \{i, i+1,\dots,n\}) \cdot \theta_{J-i} & i \in J,
        \end{cases}
    \end{equation}
    where
    \begin{equation}
        j_i := n + 1 -  \# (J \cup \{i,i+1,\dots,n\}). 
    \end{equation}
    Since $h_{j,r}(K) \in I_{n,r}$ whenever $(\# K ) + j > n$ and the ideal $SI_{n,r}$ is closed under the Euler derivation $d: \Omega \to \Omega$, we deduce that $q_{J,r,i} \in SI_{n,r}$ for all $i = 1, \dots, n$. Furthermore, one has 
    \begin{equation}
    \label{eq:q-triangularity}
        q_{J,r,i} \in \bigoplus_{J \leq _{\Gale} K} S \cdot \theta_K
    \end{equation}
    where $\leq_{\Gale}$ is the Gale order. The projection of $q_{J,r,i}$ onto the summand $S \cdot \theta_J$ is given  by $\pm p_{J,r,i} \in S$ where
    \begin{equation}
        p_{J,r,i} := \begin{cases}
            h_{j_i,r}(J \cup \{i, i+1,\dots,n\})  & i \notin J, \\
            \partial_i h_{j_i,r}(J \cup \{i, i+1,\dots,n\})  & i \in J.
        \end{cases}
    \end{equation}
    For example, suppose $n = 5, r = 4,$ and $r = \{1,3,4\}$. We have
    \[
    q_{J,r,1} = d h_{1,4}(12345) \cdot \theta_{34}, \quad
    q_{J,r,2} = h_{1,4}(12345) \cdot \theta_{134}, \quad
    q_{J,r,3} = dh_{2,4}(1345) \cdot \theta_{14}, 
    \]
    \[
    q_{J,r,4} = dh_{2,4}(1345) \cdot \theta_{13}, \quad \text{and} \quad 
    q_{J,r,5} = h_{3,4}(134) \cdot \theta_{134}
    \]
    and 
    \[
    p_{J,r,1} = \partial_1 h_{1,4}(12345), \quad
    p_{J,r,2} = h_{1,4}(12345), \quad
    p_{J,r,3} = \partial_3 h_{2,4}(1345), 
    \]
    \[
    p_{J,r,4} = \partial_4 h_{2,4}(1345), \quad \text{and} \quad 
    p_{J,r,5} = h_{3,4}(134).
    \]
    Observe that $\stair(134,4) = (\underline{2}, 3,\underline{6}, \underline{6}, 7)$.
     Since $r \geq 2$, the polynomial $p_{J,i,r} \in S$ is homogeneous with positive degree
    \begin{equation}
    \label{eq:degree-of-p}
        \deg(p_{J,i,r}) = \stair(J,r)_i + 1.
    \end{equation}

    We claim that $p_{J,i,r} \in (I_{n,r}: f_{J,r})$ for all $i = 1, \dots, n$. Indeed, since $q_{J,i,r} \in SI_{n,r}$ and $\DDD_{J,r}(\delta_{n,r}) \in SI_{n,r}^\perp$ we have
    \begin{equation}
    \label{eq:q-triangle-one}
        q_{J,i,r} \odot \DDD_{J,r}(\delta_{n,r}) = 0.
    \end{equation}
    On the other hand, the triangularity assertions of Lemma~\ref{lem:D-leading} and \eqref{eq:q-triangularity} force
    \begin{equation}
    \label{eq:q-triangle-two}
        q_{J,i,r} \odot \DDD_{J,r}(\delta_{n,r}) = \pm p_{J,i,r} \odot \left[ f_{J,r} \odot \delta_{n,r} \right] = \pm (p_{J,i,r} \cdot f_{J,r}) \odot \delta_{n,r}.
    \end{equation}
    Combining Equations~\eqref{eq:q-triangle-one} and \eqref{eq:q-triangle-two} gives $(p_{J,i,r} \cdot f_{J,r}) \odot \delta_{n,r} = 0$. Steinberg's Theorem~\ref{thm:steinberg} implies $p_{J,i,r} \cdot f_{J,r} \in I_{n,r}$, or equivalently $p_{J,i,r} \in (I_{n,r} : f_{J,r})$.

    In the case $r = 2$ it is shown in \cite[Lem. 3.3, Lem. 3.4, Lem. 3.5]{Bhattacharya} that the variety in $\CC^n$ cut out by $p_{J,2,1} = \cdots = p_{J,2,n} = 0$ consists of the origin alone, so that Lemma~\ref{lem:locus-criterion} guarantees that $p_{J,2,1}, \dots, p_{J,2,n}$ is a regular sequence. This reasoning generalizes to $r \geq 2$ without modification. Specifically, one first proves the $r \geq 2$ analog of \cite[Lem. 3.3]{Bhattacharya} by showing
    \begin{equation}
    \label{eq:partial-h-identity}
        \partial_a h_{j,r}(K) = (x_a^r - x_b^r)  \partial_a h_{j-1,r} (K \cup b)  - r x_a^{r-1} \cdot h_{j-1,r}(K \cup b)
    \end{equation}
    for any $j \geq 1$, $r > 1$, and $K \subseteq [n]$ with $a \in S$ and $b \notin S$. Equation~\eqref{eq:partial-h-identity} is proven in exactly the same way as \cite[Lem. 3.3]{Bhattacharya}. As in \cite[Lem. 3.4]{Bhattacharya}, one uses  \eqref{eq:partial-h-identity} to establish the ideal memberships
    \begin{equation}
        \label{eq:ideal-membership}
       \partial_j  h_{(n+1-\#J),r}(J) \in (p_{J,r,1},\dots,p_{J,r,n}) \quad \text{for all $j \in J$.}
    \end{equation}
    Finally, as in the proof of \cite[Lem. 3.5]{Bhattacharya}, the ideal memberships \eqref{eq:ideal-membership} are used to show that $p_{J,r,1} = \cdots = p_{J,r,n} = 0$ has only the trivial solution $\{0\}$.

    At this stage we know that $p_{J,r,1}, \dots, p_{J,r,n}$ is a homogeneous regular sequence contained in the colon ideal $(I_{n,r}: f_{J,r})$. In particular, we have the containment of ideals \[(p_{J,r,1}, \dots, p_{J,r,n}) \subseteq (I_{n,r}:f_{J,r}).\] We want to apply Lemma~\ref{lem:regular-criterion} to upgrade this ideal containment to an equality. The hypotheses of Lemma~\ref{lem:regular-criterion} are checked as follows.
    \begin{itemize}
        \item Since $r \geq 2$, the nonvanishing property \eqref{eq:nonvanishing} and Theorem~\ref{thm:steinberg} imply $f_{J,r} \notin I_{n,r}$.
        \item The maximum nonvanishing degree of the graded quotient ring $S/I_{n,r}$ is $\sum_{i=1}^n ( r \cdot i - 1).$ From the definition of $f_{J,r}$ we have
        \begin{align}
            \sum_{i=1}^n ( r \cdot i - 1) &= \deg(f_{J,r}) + \stair(J,r)_1 + \cdots + \stair(J,r)_n + n \\
            &= \deg(f_{J,r}) + \deg(p_{J,1}) + \cdots + \deg(p_{J,n})
        \end{align}
        where the second line used \eqref{eq:degree-of-p}.
    \end{itemize}
    Lemma~\ref{lem:regular-criterion} applies to give the equality of ideals
    \begin{equation}
        (p_{J,r,1}, \dots, p_{J,r,n}) = (I_{n,r}:f_{J,r})
    \end{equation}
    and the proof is complete.
\end{proof}

\subsection{Operator theorem and monomial basis} We are ready to put the lemmas in the previous technical subsection to use.  We start by giving an operator-theoretic characterization of the inverse system $SI_{n,r}^\perp$.

\begin{theorem}
    \label{thm:operator}
    {\em (Operator Theorem)} The inverse system $SI_{n,r}^\perp$ is the smallest linear subspace of $\Omega$ which $\dots$
    \begin{itemize}
        \item contains $\delta_{n,r} \in S$,
        \item is closed under the partial derivatives $\partial_1, \dots, \partial_n$, and
        \item is closed under the higher Euler operators $d_1, d_{r+1}, \dots, d_{(n-1) \cdot r+1}$.
    \end{itemize}
\end{theorem}

When $r = 1$, Theorem~\ref{thm:operator} was proven in \cite{RW}. When $r =2$, Theorem~\ref{thm:operator} was proven in \cite{Bhattacharya}. The proof of Theorem~\ref{thm:operator} presented here is more streamlined.

\begin{proof}
    When $r = 1$ this was proven in \cite{RW}, so we assume $r \geq 2$ throughout.
    Let $SH_{n,r} \subseteq \Omega$ be the linear subspace of $\Omega$ described in the statement. Theorems~\ref{thm:steinberg} and \ref{thm:swanson-wallach} give the containment $SH_{n,r} \subseteq SI_{n,r}^\perp$. A glance at Definition~\ref{def:D-operator} implies that the $\DDD_{J,r}$ operators only involve the higher Euler operators $d_1, d_{r+1}, \dots, d_{(n-1) \cdot r+1}$. Lemma~\ref{lem:D-leading} implies
    \begin{align}
        \dim_\CC SH_{n,r} &\geq \sum_{J \subseteq [n]} \dim_\CC (S/(I_{n,r}:f_{J,r}) \\
        &= \sum_{J \subseteq [n]} (\stair(J,r)_1 + 1) \cdots (\stair(J,r)_n + 1)
    \end{align}
    where the equality follows from Lemma~\ref{lem:colon-ideal-lemma}. On the other hand, Lemma~\ref{lem:spanning} implies
    \begin{equation}
        \dim_\CC SI_{n,r}^\perp = \dim_\CC SR_{n,r} \leq \# \AAA_{n,r} = \sum_{J \subseteq [n]} (\stair(J,r)_1 + 1) \cdots (\stair(J,r)_n + 1).
    \end{equation}
    Since $SH_{n,r} \subseteq SI_{n,r}^\perp$, these dimension inequalities force $SH_{n,r} = SI_{n,r}^\perp$ as desired.
\end{proof}

Next, we prove the Sagan--Swanson Conjecture~\ref{conj:SS} to give the first explicit vector space basis of $SR_{n,r}$.

\begin{theorem}
    \label{thm:ss-basis}
    The set $\AAA_{n,r}$ descends to a vector space basis of $SR_{n,r}$.
\end{theorem}

\begin{proof}
    Lemma~\ref{lem:spanning} says that $\AAA_{n,r}$ descends to a spanning set of $SR_{n,r}$. The proof of Theorem~\ref{thm:operator} shows that 
    \begin{equation}
        \dim_\CC SR_{n,r} = \sum_{J \subseteq [n]} (\stair(J,r)_1  +1) \cdots (\stair(J,r)_n + 1) = \# \AAA_{n,r}
    \end{equation}
    so this spanning set is also a basis.
\end{proof}

Counting elements of $\AAA_{n,r}$ by bidegree gives the Hilbert series of $SR_{n,r}$.

\begin{corollary}
    \label{cor:hilbert}
    The bigraded Hilbert series of $SR_{n,r}$ is given by
    \begin{equation}
        \Hilb(SR_{n,r};q,z) = \sum_{J \subseteq [n]} z^{\# J } \cdot [\stair(J,r)_1 + 1]_q \cdots [\stair(J,r)_n + 1 ]_q.
    \end{equation}
\end{corollary}

Corollary~\ref{cor:hilbert} proves a conjecture of Sagan and Swanson \cite[Conj. 5.3]{SS}. As explained in \cite{SS}, the Hilbert series of Corollary~\ref{cor:hilbert} can be written in another way. Define the {\em $(\ZZ_r \wr \symm_n)$-$q$-Stirling number} $\Stir^{(r)}_q(n,k)$ by the recursion
\begin{equation}
    \Stir^{(r)}_q(n,k) = \Stir^{(r)}_q(n-1,k-1) + [(k+1) \cdot r - 1]_q \cdot \Stir^{(r)}_q(n-1,k)
\end{equation}
and the initial condition
\begin{equation}
    \Stir^{(r)}_q(0,k) = \begin{cases}
        1 & k = 0, \\
        0 & k \neq 0.
    \end{cases}
\end{equation}
When $r = q = 1$, these are the usual Stirling numbers of the second kind counting set partitions of $[n]$ into $k$ blocks.
 Combining Corollary~\ref{cor:hilbert} and \cite[Prop. 5.7]{SS}, we have 
\begin{equation}
    \label{eq:hilbert-reformulation}
    \Hilb(SR_{n,r};q,z) = \sum_{k \geq 0} [kr]_q \cdot [(k-1) \cdot r]_q \cdots [r]_q \cdot \Stir^{(r)}_q(n,k) \cdot z^{n-k}.
\end{equation}

\subsection{Basis for colon ideal quotient}
In order to describe the module structure of $SR_{n,r}$, we will need a vector space basis of the colon ideal quotients $S/(I_{n,r}:f_{J,r})$. Specifically, we show that the bosonic part $\AAA_{n,r}(J,r)$ of the Sagan--Swanson basis $\AAA_{n,r}$ indexed by $J$ is a basis of $S/(I_{n,r}:f_{J,r})$. When $r = 1$ the following result was proven by Angarone--Commins--Karn--Murai--Rhoades \cite[Thm. 7.1]{ACKMR}.

\begin{proposition}
    \label{prop:colon-ideal-quotient-basis}
    Let $r \geq 1$ and $J \subseteq [n]$. The set 
    \[
    \AAA_{n,r}(J,r) = \{x_1^{a_1} \cdots x_n^{a_n} \,:\, a_i \leq \stair(J,r)_i \}
    \]
    descends to a vector space basis of $S/(I_{n,r}: f_{J,r})$.
\end{proposition}

\begin{proof}
    Let $\varphi_r: S \to S$ be the $\CC$-algebra homomorphism defined on variables by $x_i \mapsto x_i^r$. Drawing inspiration from the proof of \cite[Thm. 7.1]{ACKMR}, we define $\widetilde{f}_{J,r} \in S$ to be the polynomial
    \begin{equation}
        \widetilde{f}_{J,r} := \prod_{j \in J} \prod_{j < i \leq n} (x_j^r - x_i^r)
    \end{equation}
    obtained from $f_{J,r}$ by removing the factors of $x_j$ for $j \in J$. It is easily seen that
    \begin{equation}
        f_{J,r} = \prod_{j \in J} x_j \times \widetilde{f}_{J,r} \quad \text{and} \quad
        \varphi_r(\widetilde{f}_{J,1}) = \widetilde{f}_{J,r}.
    \end{equation}

    Let $\widetilde{\AAA}_{n,r}(J,r)$ be the set of monomials $x_1^{a_1} \cdots x_n^{a_n}$ whose exponent sequences satisfy
    \begin{equation}
        \begin{cases}
            a_i \leq r \cdot \stair(J,1)_i + r - 1 & \text{if $i \notin J$,} \\
            a_i \leq r \cdot \stair(J,1)_i + 2r - 1 & \text{if $i \in J$.}
        \end{cases}
    \end{equation}
    For example, if $n = 5, r= 4,$ and $J = \{1,3,4\}$ we have
    \[ \stair(J,1) = (\underline{-1}, 0, \underline{0}, \underline{0}, 1) \quad \text{and} \quad 
    \stair(J,4) = (\underline{2}, 3, \underline{6}, \underline{6},7)\]
    so that
    \[
    \widetilde{\AAA}_{n,1}(J,1) = \{ x_1^{a_1} \cdots x_5^{a_5} \,:\, (a_1,\dots,a_5) \leq (0,0,1,1,1) \}
    \]
    and
    \[
    \widetilde{\AAA}_{n,4}(J,4) = \{x_1^{a_1} \cdots x_5^{a_5} \,:\, (a_1,\dots,a_5) \leq (3,3,7,7,7) \}.
    \]
    Observe that the sequence $(3,3,7,7,7)$ characterizing $\widetilde{\AAA}_{n,4}(J,4)$ is obtained from the corresponding sequence $(0,0,1,1,1)$ for $\widetilde{A}_{n,1}(J,1)$ by multiplying each term by $r= 4$ and adding $r-1 = 3$. Observe also that, even if $1 \in J$ so that $\stair(J,1) = -1$, the set $\widetilde{A}_{n,1}(J,1)$ is nonempty (it contains the monomial 1). In the proof of \cite[Thm. 7.1]{ACKMR} it is shown that $\widetilde{\AAA}_{n,1}(J,1)$ descends to a $\CC$-basis of $S/(I_{n,1} : \widetilde{f}_{J,1})$.\footnote{The conventions in \cite{ACKMR} are such that their $\stair(J)_i$ is our $\stair(J,1)_i - 1$.} We enhance this result to $r \geq 1$.

    {\bf Claim:} {\em For $r \geq 1$, the set $\widetilde{\AAA}_{n,r}(J,r)$ descends to a $\CC$-basis of $S/(I_{n,r} : \widetilde{f}_{J,r})$.}

    To prove the Claim, we first show that $\widetilde{\AAA}_{n,r}(J,r)$ descends to a spanning set of $S/(I_{n,r} : \widetilde{f}_{J,r})$. To this end, let $m \in S$ be a monomial. There is a unique monomial factorization $m = m' \cdot m''$ where every exponent in $m'$ is divisible by $r$ and every exponent in $m''$ is $< r$. Let $\bar{m}' \in S$ be the monomial satisfying $\varphi_r(\bar{m}') = m'$. Since spanning holds for $r = 1$ by \cite[Thm. 7.1]{ACKMR}, there exist constants $\gamma_a \in \CC$ so that 
    \begin{equation}
    \label{eq:colon-span-one}
        \bar{m}' = \sum_{a \in \widetilde{\AAA}_{n,1}(J,1)} \gamma_a \cdot a + g \quad \text{where} \quad \widetilde{f}_{J,1} \cdot g \in I_{n,1}.
    \end{equation}
    Applying $\varphi_r$ to both sides of Equation~\eqref{eq:colon-span-one} yields 
    \begin{equation}
    \label{eq:colon-span-two}
        m' = \sum_{a \in \widetilde{\AAA}_{n,1}(J,1)} \gamma_a \cdot \varphi_r(a)  + \varphi_r(g) \quad \text{where} \quad \widetilde{f}_{J,r} \cdot \varphi_r(g) \in I_{n,r}.
    \end{equation}
    Since every exponent of $m''$ is $< r$, we have 
    \begin{equation}
    \label{eq:colon-span-three}
        \{ m'' \cdot \varphi_r(a) \,:\, a \in \widetilde{\AAA}_{n,1}(J,1) \} \subseteq \widetilde{\AAA}_{n,r}(J,r).
    \end{equation}
    We may therefore multiply both sides of Equation~\eqref{eq:colon-span-two} by $m''$ to get an expression of the form
    \begin{equation}
        \label{eq:colon-span-four}
        m = \sum_{a \in \widetilde{\AAA}_{n,r}(J,r)} \gamma'_a \cdot a + m'' \cdot \varphi_r(g) \quad \text{where} \quad 
        \widetilde{f}_{J,r} \cdot m'' \cdot \varphi_r(g) \in I_{n,r}
    \end{equation}
    for some constants $\gamma'_a \in \CC$. Equation~\eqref{eq:colon-span-four} implies that $m$ lies in the span of $\widetilde{\AAA}_{n,r}(J,r)$ modulo $(I_{n,r}:\widetilde{f}_{J,r})$ and the spanning part of the Claim is proven.

    The linear independence part of the Claim is established as follows. It follows from \cite[Lem. 3.3, Thm. 6.2]{ACKMR} that the ideal $(I_{n,1}:\widetilde{f}_{J,1})$ is a complete intersection generated by homogeneous polynomials $g_1, \dots, g_n \in S$ with (positive) degrees
    \begin{equation}
        \deg(g_i)  = \begin{cases}
            \stair(J,1)_i + 1 & \text{if $i \notin J$}, \\
            \stair(J,1)_i + 2 & \text{if $i \in J$.}
        \end{cases}
    \end{equation}
    Since the ring homomorphism $x_i \mapsto x_i^r$ sends regular sequences to regular sequences (for example, by Lemma~\ref{lem:locus-criterion}), we know that $\varphi_r(g_1), \dots, \varphi_r(g_n)$ is a regular sequence with 
    \begin{equation}
        \deg(\varphi_r(g_i))  = \begin{cases}
            r \cdot \stair(J,1)_i + r & \text{if $i \notin J$}, \\
            r \cdot \stair(J,1)_i + 2r & \text{if $i \in J$.}
        \end{cases}
    \end{equation}
    For $1 \leq i \leq n$, we have $\widetilde{f}_{J,1} \cdot g_i \in I_{n,1}$, and applying $\varphi_r$ yields $\widetilde{f}_{J,r} \cdot \varphi_r(g_i) \in I_{n,r}$ so that 
    \begin{equation}
    \label{eq:phi-colon-containment}
        ( \varphi_r(g_1), \dots, \varphi_r(g_n)) \subseteq (I_{n,r} : \widetilde{f}_{n,r}).
    \end{equation}
    An application of Lemma~\ref{lem:regular-criterion} implies that the containment \eqref{eq:phi-colon-containment} is an equality of ideals so that 
    \begin{equation}
        \dim_\CC S/(I_{n,r} : \widetilde{f}_{n,r}) = \prod_{i=1}^n \deg(\varphi_r(g_i)) = \# \widetilde{\AAA}_{n,r}(J,r)
    \end{equation}
    and the spanning set $\widetilde{\AAA}_{n,r}(J,r)$ of $S/(I_{n,r} : \widetilde{f}_{n,r})$ must be a basis. This completes the proof of the Claim.

    The Claim proves the proposition as follows. Since $f_{J,r} = \prod_{j \in J} x_j \times \widetilde{f}_{J,r}$, we have an  $S$-module homomorphism
    \begin{equation}
        \psi: S/(I_{n,r} : f_{J,r}) \xrightarrow{ \, \, \times \prod_{j \in J} x_j \, \, } S/(I_{n,r} : \widetilde{f}_{J,r})
    \end{equation}
    induced by multiplication by $\prod_{j \in J} x_j$. The image of $\AAA_{n,r}(J,r)$ under $\psi$ is a subset of $\widetilde{\AAA}_{n,r}(J,r)$, so the Claim forces $\AAA_{n,r}(J,r)$ to be linearly independent in $S/(I_{n,r} : f_{J,r})$. When $r \geq 2$, Lemma~\ref{lem:colon-ideal-lemma} implies that 
    \begin{equation}
        \dim_\CC S/(I_{n,r} : f_{J,r}) = \prod_{i=1}^n (\stair(J,r)_i + 1) = \# \AAA_{n,r}(J,r)
    \end{equation}
    so that the linearly independent set $\AAA_{n,r}(J,r)$ must be a basis of $S/(I_{n,r} : f_{J,r})$ and we are done. When $r= 1$, the proposition is \cite[Thm. 7.1]{ACKMR}.
\end{proof}

\section{Module Structure}
\label{sec:Module}

\subsection{$r$-ified ordered set partitions}
We present our combinatorial model for the ring $SR_{n,r}$.
Given positive integers $n,r$, a \textit{$r$-ified ordered set partition}\footnote{The term `$r$-colored ordered set partition' has seen various meanings in the literature, so we use a new word reminiscent of `rareified'.} is a pair $\sigma=(Z,\mathcal B)$ where \begin{itemize}
    \item $Z$ (called the \textit{zero block}) is a subset of $[n]$ such that every element of $Z$ is decorated with one of the colors $\{0,\ldots, r-2\}$;
    \item $\mathcal B=(B_1\mid \cdots\mid B_k)$ is an ordered set partition of $[n]-Z$, where every elements in each block $B_i$ is decorated by one of the colors $\{0,\ldots, r-1\}$.
\end{itemize}
Note that $Z$, unlike blocks of $\mathcal B$, is allowed to be empty. For notational ease, we will write the ordered set partition $(Z,\mathcal B)$ as $(Z\mid B_1\mid \cdots\mid B_k)$. Let $\mathcal{\widetilde{{OP}}}_{n,r}$ denote the set of $r$-ified ordered set partitions. For example, when $n=6$ and $r=4$, $$(1^25^0\mid 3^34^2\mid2^06^1),  \quad (\varnothing\mid 1^34^25^1\mid 2^03^0\mid 6^1), \quad(1^02^23^14^15^16^2)$$are all valid elements of $\mathcal{\widetilde{{OP}}}_{n,r}$ (the last example satisfying $Z=[n]$), whereas $(1^35^0|3^34^2|2^06^1)$ is not (since the elements of the zero block can only have colors in the set $\{0,1,2\}$). Note that this is in bijection with the set of ordered super set partitions defined in \cite{SS}.

Consider the vector space $\CC[\mathcal{\widetilde{{OP}}}_{n,r}]$ freely generated by $\mathcal{\widetilde{{OP}}}_{n,r}$. We give this vector space the structure of a $\ZZ_r\wr \mathfrak S_n$-representation by defining its action on the basis elements. To do so, it suffices to describe how $\mathfrak S_n$ and the diagonal subgroup $\ZZ_r^n$ acts on this basis. 

We will let a permutation $w \in\mathfrak S_n$ act on $\sigma\in \mathcal{\widetilde{{OP}}}_{n,r}$ by permuting the numbers, while keeping the colors in place. For example in the $n=6$, $r=4$ case, if \[\sigma=(1^25^0\mid 3^34^2\mid2^06^1),\] then the permutation $w=[4,6,3,1,2,5] \in \symm_6$ acts via \[w\cdot\sigma=(4^22^0\mid 3^31^2\mid6^05^1)=(2^04^2\mid 1^23^3\mid5^16^0).\]

Now consider $g=(\zeta^{c_1},\ldots,\zeta^{c_n})\in \mathbb Z_r^n$, where $\zeta=\exp(2\pi i/r)$ is a generator of $\ZZ_r$. The element $g \in \ZZ_r^n$ acts on $\sigma$ as follows.
\begin{itemize}
\item
If $i$ is in a non-zero block and has color $k$, it gets the color $k+c_i\pmod{r}$ in $g\cdot\sigma$. 
\item
However, if $i$ is in the zero block and has color $k$, the action of $g$ does not change this color, but scales $\sigma$ by a factor of $\zeta^{kc_i}$ within the vector space $\CC[\widetilde{\OP}_{n,r}]$.
\end{itemize}
In our previous example with $n = 6$ and $r = 4$, if $g=(\zeta^{2},\zeta^1,\zeta^0,\zeta^3,\zeta^2,\zeta^3,\zeta^1)$ then $$g\cdot (1^25^0\mid 3^34^2\mid2^06^1)=\zeta^{2\cdot 2}\cdot\zeta^{0\cdot 3}\cdot(1^25^0\mid 3^3 4^1\mid 2^16^2).$$

\begin{lemma}
    \label{lem:well-defined-action}
    The above actions of $\symm_n$ and $\ZZ_r^n$ on $\CC[\widetilde{\OP}_{n,r}]$ extend to give $\CC[\widetilde{\OP}_{n,r}]$ the structure of an ungraded $(\ZZ_r \wr \symm_n)$-module.
\end{lemma}

\begin{proof}
    Let $\zeta = \exp(2 \pi i /r)$ generate $\ZZ_r$.
    The wreath product $\ZZ_r \wr \symm_n$ is generated by $\symm_n$ and $\ZZ_r^n$ subject to the internal relations in these groups together with the relations
    \begin{equation}
    \label{eq:semidirect-relation}
        w \cdot (\zeta^{c_1}, \dots, \zeta^{c_n}) \cdot w^{-1} = (\zeta^{c_{w^{-1}(1)}}, \dots, \zeta^{c_{w^{-1}(n)}})
    \end{equation}
    for all $w \in \symm_n$ and $(\zeta^{c_1}, \dots, \zeta^{c_n}) \in \ZZ_r^n$. One checks that the actions of $\symm_n$ and $\ZZ_r^n$ on $\CC[\widetilde{\OP}_{n,r}]$ described above respect the relation \eqref{eq:semidirect-relation}.
\end{proof}

Although $\CC[\widetilde{\OP}_{n,r}]$ is not a $(\ZZ_r \wr \symm_n)$-permutation representation, it is very nearly so: the representing matrix of any $g \in \ZZ_r \wr \symm_n$ with respect to the natural basis $\widetilde{\OP}_{n,r}$ is a monomial matrix (has a unique nonzero entry in every row and column). Such representations are called {\em monomial representations}.
We will show in Theorem~\ref{thm: module iso} that $\CC[\widetilde{\OP}_{n,r}]$, twisted by the action of the one-dimensional determinant representation, is isomorphic to $SR_{n,r}$ as an ungraded $(\ZZ_r\wr \mathfrak S_n)$-module: 
\begin{equation}
    SR_{n,r}\cong_{\ZZ_r\wr \mathfrak S_n}\CC[\mathcal{\widetilde{{OP}}}_{n,r}]\otimes \det.\label{combmodel}
\end{equation}
The remainder of this section is devoted to proving \eqref{combmodel}. We adopt a strategy similar to that used in \cite{MRW} for proving the Fields Conjectures.

Recall that we have the linear characters $\sign, \chi: \ZZ_r \wr \symm_n \to \CC$ given by
\[
\sign(g) := \text{sign of the underlying permutation matrix of $g$}\]
and
\[
\chi(g) := \text{product of the nonzero entries of $g$.}
\]
The first step towards proving the isomorphism \eqref{combmodel} is to generalize $\sign$ to arbitrary $r$-partitions $\llambda \vdash n$.

Consider an $r$-partition $\bm\lambda=(\lambda^{(1)},\ldots,\lambda^{(r)})\vdash_r n$, where  $\lambda^{(i)}=\left(\lambda^{(i)}_1,\lambda^{(i)}_2,\ldots,\lambda^{(i)}_{k_i}\right)$. Write $\alpha = (\alpha_1,\dots,\alpha_r) := (|\lambda^{(1)}|, \dots,|\lambda^{(r)}|)$. Define $\ZZ_r \wr \symm_\llambda \subseteq \ZZ_r \wr \symm_\alpha$ to be the following block diagonal subgroup:
\begin{equation}
\ZZ_r\wr \symm_{\bm\lambda}:=\left(\ZZ_r\wr\symm_{\lambda^{(1)}_1}\times \cdots\times \ZZ_r\wr\symm_{\lambda^{(1)}_{k_1}}\right)\times\cdots\times \left(\ZZ_r\wr\symm_{\lambda^{(r)}_1}\times\cdots\times\ZZ_r\wr\symm_{\lambda^{(r)}_{k_r}}\right)\subseteq \ZZ_r\wr\symm_n.
\end{equation}
We define a linear character 
\begin{equation}
    \sign_\llambda: \ZZ_r \wr \symm_\llambda \longrightarrow \CC
\end{equation}
as follows. For any $r$-tuple $g=(g^{(1)}_{1},\cdots,g^{(r)}_{k_r}) \in \ZZ_r \wr \symm_\llambda$, we set 
\begin{equation}
\label{eq:colored-sign-defintion}
\sign_{\bm\lambda}(g):= \\ \prod_{i=1}^{k_1}\sign(g^{(1)}_i)\cdot\chi^0(g^{(1)}_{i})\prod_{i=1}^{k_2}\sign(g^{(2)}_i)\cdot\chi^1(g^{(2)}_{i})\cdots \prod_{i=1}^{k_r}\sign(g^{(r)}_i)\cdot\chi^{r-1}(g^{(r)}_{i}).
\end{equation}
Equation~\eqref{eq:colored-sign-defintion} uses the notation $\chi^{a-1}(g_i^{(a)}) = ( \chi(g^{(a)})^{a-1}$.
This gives rise the the following element of $\CC\left[\ZZ_r\wr\symm_{n}\right]$ \begin{equation}\varepsilon_{\bm\lambda}:=\sum_{g\in \ZZ_r\wr\symm_{\bm\lambda}}\sign_{\bm\lambda}(g)\cdot g.\end{equation} 
Roughly speaking $\varepsilon_\llambda \in \CC[\ZZ_r \wr \symm_n]$ should be thought of as a `colored antisymmetrizer' corresponding to $\llambda$.

If $V$ is a $(\ZZ_r \wr \symm_n)$-module and $\llambda \vdash_r n$, the image $\varepsilon_\llambda \cdot V$ of $V$ under the action of $\varepsilon_\llambda$ is a $\CC$-vector space. The dimension of this vector space is related to the Frobenius image of $V$ as follows.

\begin{lemma}\label{perplemma}
    Let $\bm\lambda=(\lambda^{(1)},\ldots,\lambda^{(r)})$ be an $r$-partition of $n$, and let $\llambda^{\text{rev}}$ denote the $r$-partition $(\lambda^{(r)},\lambda^{(r-1)},\ldots, \lambda^{(1)})$ obtained by reversing the order of the partitions in $\llambda$. Then for any $\ZZ_r\wr\symm_n$-module $V$, $$\dim_\CC\varepsilon_{\bm\lambda}\cdot V=\left\langle \Frob(V),e_{\llambda^{\text{rev}}}\right\rangle.$$
\end{lemma}
\begin{proof}
    Since $\antisymm\cdot V$ is precisely the $\sign_\llambda$-isotypic component of the restriction $\Res^{\ZZ_r \wr \symm_n}_{\ZZ_r \wr \symm_\llambda} (V)$, its dimension is given by 
    \begin{multline}\label{eq: dim hom}\dim_\CC\antisymm\cdot V=\dim_\CC\Hom_{\ZZ_r\wr\symm_\llambda}\left(\sign_{\llambda},\Res^{\ZZ_r \wr \symm_n}_{\ZZ_r \wr \symm_\llambda} (V)\right) \\ 
    =\dim_\CC\Hom_{\ZZ_r\wr \symm_n}\left(\Ind_{\ZZ_r \wr \symm_\llambda}^{\ZZ_r \wr \symm_n} (\sign_{\llambda}),V\right).
    \end{multline}where we have used Frobenius reciprocity. Since $$\Frob(\sign_\llambda)=e_{\lambda^{(1)}}\left(\xx^{(r)}\right)\times e_{\lambda^{(2)}}\left(\xx^{(1)}\right)\times \cdots \times e_{\lambda^{(r)}}\left(\xx^{(r-1)}\right)=e_{(\lambda^{(2)},\ldots,\lambda^{(r)},\lambda^{(1)})},$$ the quantity in \eqref{eq: dim hom} may be calculated as $\left\langle\Frob(V),e_{\llambda^{\text{rev}}}\right\rangle$ where $\llambda^{\text{rev}}$ is the dual of $(\lambda^{(2)},\ldots,\lambda^{(r)},\lambda^{(1)})$ in the sense of \cite{CR}.
\end{proof}

Our strategy for proving the module isomorphism \eqref{combmodel} is as follows. Since the set $\{e_{\bm\lambda}\}_{\bm\lambda\vdash_r n}$ is a basis of $\Lambda^{(r)}$, in order to prove that the two sides of \eqref{combmodel} have identical Frobenius series, it suffices to show
\[
\dim_\CC \varepsilon_\lambda \cdot SR_{n,r} = 
\dim_\CC \varepsilon_\llambda \cdot (\CC[\widetilde{\OP}_{n,r}] \otimes \det) 
\]
for all $r$-partitions $\bm\lambda \vdash_r n$. We prove that the vector spaces $\varepsilon_\lambda \cdot SR_{n,r}$ and $\varepsilon_\llambda \cdot (\CC[\widetilde{\OP}_{n,r}] \otimes \det)$ have the same dimension in several stages. The first step is establishing the inequality
\[
\dim_\CC \varepsilon_\lambda \cdot SR_{n,r} \leq 
\dim_\CC \varepsilon_\llambda \cdot (\CC[\widetilde{\OP}_{n,r}] \otimes \det)
\]
using the basis $\AAA_{n,r}$ of $SR_{n,r}$ furnished by Theorem~\ref{thm:ss-basis}.

\subsection{Antisymmetrized monomials and upper bound}In the previous section, we obtained a monomial basis $\mathcal A_{n,r}$ of $SR_{n,r}$. If we apply $\varepsilon_{\bm\lambda}$ to this basis, some monomials vanish, and some monomials yield the same image as other monomials up-to a scalar multiple. If we discard these obvious linear dependencies, we obtain a smaller set of antisymmetrized monomials that is guaranteed to span $\varepsilon_{\bm\lambda}\cdot SR_{n,r}$. In this subsection we enumerate the monomials in this smaller set to give an upper bound on $\dim_\CC \varepsilon_{\bm\lambda}\cdot SR_{n,r}$.
Before we start counting in earnest, we introduce some  terminology that will help us describe this collection of antisymmetrized monomials.

\begin{definition}
   Given an $r$-partition $\bm\lambda$ of $n$, a {\em signed $r$-partition of $n$}  is a pair $(\bm\lambda,\bm\gamma)$, where $\bm\gamma=(\gamma^{(1)},\ldots, \gamma^{(r)})$ is an $r$-tuple of weak compositions such that for each $1\le i\le r$, the composition $\gamma^{(i)}$ has at most as many parts as $\lambda^{(i)}$ and satisfies $\gamma^{(i)}\le \lambda^{(i)}$ componentwise.
\end{definition}

Given $\bm\lambda\vdash_rn$, we will often write $\bm\gamma\le \bm\lambda$ to mean that $(\bm\lambda,\bm\gamma)$ is a signed $r$-partition of $[n]$. As an example, suppose $r=3$, $n=10$, and consider the the 3-partition $\bm\lambda=(2,1\parallel 3,2\parallel2)$, where we have used double vertical bars to separate the 3 partitions $\lambda^{(1)}, \lambda^{(2)},$ and $\lambda^{(3)}$. One possible signed $r$-partition would be the pair $(\llambda, \bm\gamma)$ where 
$\bm\gamma = (1,1\parallel 2,0\parallel1)$. Here we have $\gamma^{(1)} = (1,1), \gamma^{(2)} = (2,0),$ and $\gamma^{(3)} = (2)$.

Let $(\llambda, \bm\gamma)$ be a signed $r$-partition. For notational ease, we will use $L_i$ to denote the total size of the first $i$ partitions in $\bm\lambda$: \begin{equation}L_i:=|\lambda^{(1)}|+\cdots+|\lambda^{(i)}|.\end{equation} Similarly, we write $G_i$ for the corresponding partial sum for $\bm\gamma$: 
\begin{equation} G_i:=|\gamma^{(1)}|+\cdots+|\gamma^{(i)}|.
\end{equation}
By convention, we set $L_0=G_0=0.$

Any $r$-partition splits the set $[n] = \{1,\ldots, n\}$ into consecutive intervals: for example, if $n = 10$ and  $\bm\lambda=(2,1\parallel 3,2\parallel2)$ we have the splitting $$(1,2\mid 3\parallel4,5,6\mid7,8\parallel9,10).$$Here we have used a single line to separate blocks corresponding to parts of the same $\lambda^{(i)}$, and doubled lines to separate blocks corresponding to different $\lambda^{(i)}$'s. Using our $L_i$ notation, the block of numbers corresponding to $\lambda^{(i)}_j$ is explictly given by $$\left\{L_{i-1}+\lambda^{(i)}_1+\cdots +\lambda^{(i)}_{j-1}+1,\ldots, L_{i-1}+\lambda^{(i)}_1+\cdots \lambda^{(i)}_j+\lambda^{(i)}_{j}\right\}.$$ Call this the \textit{$\lambda^{(i)}_j$-batch of $[n]$}. In the previous example, the $\lambda^{(2)}_1$-batch of $[10]$ would be $\{4,5,6\}$.

\begin{definition}
    Given a signed $r$-partition $(\bm\lambda,\bm\gamma)$ of $n$, define $J(\bm\lambda,\bm\gamma)$ to be the following subset of $[n]$:

    $$J(\bm\lambda,\bm\gamma)=\bigsqcup_{\substack{1\le i\le r\\1\le j\le k_i}} \left\{L_{i-1}+\lambda^{(i)}_1+\cdots+\lambda^{(i)}_j-\gamma^{(i)}_j+1,\cdots,L_{i-1}+\lambda^{(i)}_1+\cdots+\lambda^{(i)}_j\right\}.$$
\end{definition}

In other words, $J(\bm\lambda,\bm\gamma)$ is obtained by picking the largest $\gamma^{(i)}_j$ elements from the $\lambda^{(i)}_j$-batch of $[n]$. Since $\gamma^{(i)}_j \leq \lambda^{(i)}_j$, this is a well-defined operation. 
In our running example of $\llambda = (2,1 \mid \mid 3,2 \mid \mid 2)$ and $\bm \gamma = (1,1 \mid \mid 2, 0 \mid \mid 1)$, one has \[J(\bm\lambda,\bm\mu) = \{2,3,5,6,10\}.\] Recall that every subset $J\subseteq[n]$ is associated to a $(J,r)$-staircase $\stair(J,r)$. Applying this to our example $J(\bm\lambda,\bm\gamma)$, we obtain the staircase $$\stair(J(\bm\lambda,\bm\gamma),r)=\left(2,\underline{4}\mid\underline{4}\parallel5,\underline{7},\underline{7}\mid8,11\parallel14,\underline{16}\right)$$
where again single lines correspond to parts of the same $\lambda^{(i)}$ and double lines separate blocks corresponding to different $\lambda^{(i)}$'s. 

We will apply the group algebra element $\varepsilon_\llambda$ to certain superspace monomials with fermionic part $\theta_{J(\bm\lambda,\bm\gamma)}$. To describe these monomials, it remains to specify their bosonic parts for each signed $r$-partition $(\llambda, \bm\gamma)$, which we do in the following definition.

\begin{definition}
    Let $(\bm\lambda,\bm\gamma)$ be a signed $r$-partition of $n$. Define $\mathcal A_{n,r}(\bm\lambda,\bm\gamma)$ to be the collection of monomials $x_1^{a_1}\cdots x_n^{a_n}$ in $S = \CC[x_1,\dots,x_n]$, subject to the following conditions:
    \begin{enumerate}
        \item we have $a_k\le \stair(J(\bm\lambda,\bm\gamma),r)_k$ for every $k\in [n]$;
        \item for each $1\le i\le r$ and for each $1\le j\le k_i$, 
        \begin{enumerate}
            \item $a_{L_{i-1}+\lambda^{(i)}_1+\cdots+\lambda^{(i)}_{j-1}+1},\ldots,a_{L_{i-1}+\lambda^{(i)}_1+\cdots+\lambda^{(i)}_{j-1}+\lambda^{(i)}_j-\gamma^{(i)}_j}$ is a strictly increasing sequence of non-negative integers congruent to $(i-1)\pmod r$;
            \item $a_{L_{i-1}+\lambda^{(i)}_1+\cdots+\lambda^{(i)}_{j}-\gamma^{(i)}_j+1},\ldots,a_{L_{i-1}+\lambda^{(i)}_1+\cdots+\lambda^{(i)}_j}$ is a weakly increasing sequence of non-negative integers congruent to $(i-2)\pmod r$.
        \end{enumerate}
    \end{enumerate}
\end{definition}

In our running example of $\llambda = (2,1 \mid \mid 3,2 \mid \mid 2)$ and $\bm \gamma = (1,1 \mid \mid 2, 0 \mid \mid 1)$, the $\mathcal A_{n,r}(\bm\lambda,\bm\gamma)$ would consists of  monomials of the form $x_1^{a_1}x_2^{a_2}\cdots x_{10}^{a_{10}}$ where the inequality 
\begin{equation}\label{cond1}
    (a_1,\underline{a_2}\mid \underline{a_3}\parallel a_4,\underline{a_5},\underline{a_6}\mid a_7,a_8\parallel a_9,\underline{a_{10}})\le \left(2,\underline{4}\mid\underline{4}\parallel5,\underline{7},\underline{7}\mid8,11\parallel14,\underline{16}\right)
\end{equation}
holds componentwise, we have the following modular equalities: \begin{equation}\label{cond2}
\begin{cases}
a_1\equiv a_5\equiv a_6\equiv  0\pmod 3, \\ a_4\equiv a_7\equiv a_{8}\equiv a_{10}\equiv 1\pmod{3}, \\ a_2\equiv a_3\equiv a_9\equiv2\pmod{3}
\end{cases}\end{equation} and we have the inequalities: \begin{equation}\label{cond3}
\begin{cases}a_5\le a_6,\\a_7<a_8.\end{cases} \end{equation}
Here \eqref{cond2} comes from the fact that in any block corresponding to $\lambda^{(i)}$, the non-underlined integers must be $(i-1)\pmod 3$ and the underlined integers must be $(i-2)\pmod 3$. The condition \eqref{cond3} represents the fact that in each block the non-underlined integers form a strictly increasing sequence and the underlined integers form a weakly increasing sequence.

\begin{lemma}
\label{lem:parabolic-spanning-set}
    Given any multipartition $\bm\lambda\vdash_r n$, the set of superspace monomials $$\bigsqcup_{\bm\gamma\le\bm\lambda}\varepsilon_{\bm\lambda}\cdot\left(\mathcal A_{n,r}(\bm\lambda,\bm\gamma)\cdot\theta_{J(\bm\lambda,\bm\gamma)}\right)$$descends to a spanning set of $\varepsilon_{\bm\lambda}\cdot SR_{n,r}$. Here the disjoint union is taken over all $\bm\gamma$ so that $(\bm\lambda,\bm\gamma)$ forms a signed $r$-partition.
\end{lemma}

It will develop (Corollary~\ref{cor:parabolic-basis}) that the spanning set of $\varepsilon_\llambda \cdot SR_{n,r}$ in Lemma~\ref{lem:parabolic-spanning-set} is in fact a basis of this vector space. 

\begin{proof}
    We start by establishing showing that the union in the lemma is actually disjoint. Write $\llambda = (\lambda^{(1)}, \dots, \lambda^{(r)}) \vdash_r n$.  For any signed $r$-partition $(\llambda,\ggamma)$ with $\ggamma \leq \llambda$ and any  element \[f_{\llambda,\ggamma} \in \varepsilon_\llambda \cdot \left(  \AAA_{n,r}(\llambda,\ggamma) \cdot \theta_{J(\llambda,\ggamma}\right),\]
    every superspace monomial appearing in $f_{\llambda,\ggamma}$ will have precisely $\gamma^{(i)}_j$ fermoinic monomials appearing among the $\lambda^{(i)}_j$-batch of indices in $[n]$. The disjointness of the union over all $\ggamma \leq \llambda$ follows.\footnote{It also follows that the sum $\sum_{\ggamma \leq \llambda} \mathrm{span}_\CC (\varepsilon_\llambda \cdot (\AAA_{n,r}(\llambda,\ggamma) \cdot \theta_{J(\llambda,\ggamma)})$ of vector spaces is direct.}

    Theorem \ref{thm:ss-basis} guarantees that $\antisymm\cdot \AAA_{n,r}$ already descends to a spanning set of $\antisymm\cdot SR_{n,r}$. We now observe that many elements in the set $\antisymm\cdot \AAA_{n,r}$ are either zero, or are scalar multiples of other elements in this set, so let us remove these ``obvious" linear dependencies. Consider any monomial $m=x_1^{a_1}\cdots x_n^{a_n}\cdot\theta_J$.

    \begin{itemize}
        \item We can permute the exponents $a_{L_{i-1}+\lambda^{(i)}_1+\cdots+\lambda^{(i)}_{j-1}+1},\ldots,a_{L_{i-1}+\lambda^{(i)}_1+\cdots+\lambda^{(i)}_{j-1}+\lambda^{(i)}_j-\gamma^{(i)}_j}$ or the exponents $a_{L_{i-1}+\lambda^{(i)}_1+\cdots+\lambda^{(i)}_{j}-\gamma^{(i)}_j+1},\ldots,a_{L_{i-1}+\lambda^{(i)}_1+\cdots+\lambda^{(i)}_j}$ by acting on $m$ by an element of $\ZZ_r\wr\symm_\llambda$, which fixes $\antisymm\cdot m$ up to sign.

        \item If any two of the exponents $a_{L_{i-1}+\lambda^{(i)}_1+\cdots+\lambda^{(i)}_{j-1}+1},\ldots,a_{L_{i-1}+\lambda^{(i)}_1+\cdots+\lambda^{(i)}_{j-1}+\lambda^{(i)}_j-\gamma^{(i)}_j}$ are equal, $\antisymm\cdot m=0$.

        \item If any of the exponents $a_{L_{i-1}+\lambda^{(i)}_1+\cdots+\lambda^{(i)}_{j-1}+1},\ldots,a_{L_{i-1}+\lambda^{(i)}_1+\cdots+\lambda^{(i)}_{j-1}+\lambda^{(i)}_j-\gamma^{(i)}_j}$ is not congruent to $(i-1)\pmod r$, or if any of the exponents $a_{L_{i-1}+\lambda^{(i)}_1+\cdots+\lambda^{(i)}_{j}-\gamma^{(i)}_j+1},\ldots,a_{L_{i-1}+\lambda^{(i)}_1+\cdots+\lambda^{(i)}_j}$ is not congruent to $(i-2)\pmod{r}$, $\antisymm\cdot m=0$.
    \end{itemize}
    Removing these dependencies, we end up with precisely the set $\bigsqcup_{\ggamma\le \llambda}\AAA_{n,r}(\llambda,\ggamma)\cdot\theta_{J(\llambda,\ggamma)}$, which implies the above lemma.
\end{proof}

Counting the elements of the set in Lemma~\ref{lem:parabolic-spanning-set} yields an upper bound for $\dim_\CC \varepsilon_{\bm\lambda}\cdot SR_{n,r}$. We will now show that this upper bound is precisely the dimension of $\varepsilon_{\bm\lambda}\cdot \left(\CC[\mathcal{\widetilde{{OP}}}_{n,r}]\otimes \det \right)$.

\begin{lemma}\label{lem: monomial count}
    For any $\bm\lambda\vdash_r n$, we have the equality $$\sum_{\bm\gamma\le \bm\lambda}\#\mathcal A_{n,r}(\bm\lambda,\bm\gamma)=\dim_\CC\varepsilon_{\bm\lambda}\cdot \left(\CC[\mathcal{\widetilde{{OP}}}_{n,r}]\otimes \det \right).$$
\end{lemma}
\begin{proof}
    We give a combinatorial basis of the vector space $\varepsilon_{\bm\lambda}\cdot \left(\CC[\mathcal{\widetilde{{OP}}}_{n,r}]\otimes \det \right)$ on the right-hand side. We will say an $r$-ified ordered set partition $\sigma=(Z\mid B_1\mid\cdots\mid B_k)$ of $[n]$ is \textit{$\bm\lambda$-surviving} if it satisfies for the following conditions:
    \begin{itemize}
        \item for each $1\le i\le r$ and each $1\le j\le k_i$, the numbers in the $\lambda^{(i)}_j$-batch of $[n]$ appear left-to-right in $\sigma$ in increasing order;
        \item every element in a non-zero block has color $0$;
        \item if any element appearing in the zero block belong to the $\lambda^{(i)}_j$-batch of $[n]$, then it has color $r-i$. 
    \end{itemize}
    In particular the last condition means numbers from the $\lambda^{(1)}_j$-batch (for any $j$) can never appear in the zero block.
    If $v$ is a vector spanning the $\det$ representation, it is easy to see that the set $$\{\varepsilon_{\bm\lambda}\cdot (\sigma\otimes v):\sigma\text{ is a }\bm\lambda\text{-surviving $r$-ified ordered set partition}\}$$ forms a basis for $\varepsilon_{\bm\lambda}\cdot \left(\CC[\mathcal{\widetilde{{OP}}}_{n,r}]\otimes \det \right)$. In order to count this set, we describe an algorithm to build all $\bm\lambda$-surviving $r$-ified ordered set partitions of $[n]$.

    This proceeds as follows: start with a sequence of exactly $\lambda^{(1)}_1-\gamma^{(1)}_1$ $\bullet$'s following the (initially empty) zero block: \[(\;\varnothing\mid \overbrace{\bullet\mid\bullet\mid\cdots\mid\bullet}^{\lambda^{(1)}_1 - \gamma^{(1)}_1}\;).\] Add a total of $\gamma{(1)}_1$ $\circ$'s in the parts containing black dots, possibly with multiplicity. This can be done in $\binom{\lambda^{(1)}_1-1}{\gamma^{(1)}_1}$ ways. Finally replace the $\lambda^{(1)}_1$ $\bullet$'s and $\circ$ with the $\lambda^{(1)}_1$-batch of $[n]$ (namely, $\{1,\ldots, \lambda^{(1)}_1\}$) in increasing order left to right, and color them all with $0$. This yields a valid $r$-ified order set partition of the  first batch of $[n]$ with an empty zero block in which every entry is colored 0.

    The rest of the process works similarly. Suppose we have finished putting all the integers prior to the $\lambda^{(i)}_j$-batch of $[n]$ into our $r$-ified ordered set partition. This should have created an $r$-ified ordered set partition with exactly $$S=L_{i-1}-G_{i-1}+\lambda^{(i)}_1+\cdots+\lambda^{(i)}_{j-1}-\gamma^{(i)}_1-\cdots-\gamma^{(i)}_{j-1}$$non-zero blocks. We insert $\lambda^{(i)}_j-\gamma^{(i)}_j$ new singleton non-zero blocks with a $\bullet$ in each: this can be done in $$\binom{S+\lambda^{(i)}_j-\gamma^{(i)}_j}{\lambda^{(i)}_j-\gamma^{(i)}_j}$$ways. Next, we add $\gamma^{(i)}_j$ $\circ$'s (possibly with multiplicity) to the existing blocks: if $i=1$, we only add these to the non-zero blocks, but if $i>1$, we may add these to any existing block including the zero block. This can be done in $$\binom{S+\lambda^{(i)}_j-1}{\gamma^{(i)}_j}\text{ or }\binom{S+\lambda^{(i)}_j}{\gamma^{(i)}_j}$$ways depending on whether $i=1$ or $i>1$. Finally, the $\bullet$'s and $\circ$'s can be replaced with the numbers in the $\lambda^{(i)}_j$-batch of $[n]$ left to right in increasing order. The colors of these new numbers are determined by the $\bm\lambda$-surviving condition: every number gets color $0$ if it is in a non-zero block, and color $r-i$ if it is in a zero block (this possibility does not arise if $i=1$).

    Executing this process until we process the $\lambda^{(r)}_{k_r}$-batch of $[n]$ yields a $\bm\lambda$-surviving $r$-ified ordered set partition of $[n]$. Combining the enumeration of the number of options at each stage, we see that for a given $\bm\gamma$ the number of ways to perform this algorithm is 
   
    \begin{multline}\label{eq: total count}
        \prod_{j=1}^{k_1}\binom{\lambda^{{(1)}}_1+\cdots+\lambda^{(1)}_j-\gamma^{(1)}_1-\cdots-\gamma^{(1)}_j}{\lambda^{(1)}_j-\gamma^{(1)}_{j}}\binom{\lambda^{{(1)}}_1+\cdots+\lambda^{(1)}_j-\gamma^{(1)}_1-\cdots-\gamma^{(1)}_{j-1}-1}{\gamma^{(1)}_j}\\\times\prod_{i=2}^r\left(\prod_{j=1}^{k_i}\binom{L_{i-1}-G_{i-1}+\lambda^{{(i)}}_1+\cdots+\lambda^{(i)}_j-\gamma^{(i)}_1-\cdots-\gamma^{(i)}_j}{\lambda^{(i)}_j-\gamma^{(i)}_{j}}\right.\\
        \left.\times\binom{L_{i-1}-G_{i-1}+\lambda^{{(i)}}_1+\cdots+\lambda^{(i)}_j-\gamma^{(i)}_1-\cdots-\gamma^{(i)}_{j-1}}{\gamma^{(i)}_j}\right).
    \end{multline}

    As we vary $\bm\gamma$, we obtain every possible $\bm\lambda$-surviving $\sigma$; so the total number of possible $\sigma$'s is obtained by summing the above expression over all $\bm\gamma\le\bm\lambda$. However, one may verify that the above expression is precisely $\#\mathcal A_{n,r}(\bm\lambda,\bm\gamma)$, which proves our claim.
\end{proof}
By combining the above two lemmas, we can now conclude the inequality
\begin{equation}\label{eq: upper bound}
    \dim_\CC\antisymm\cdot SR_{n,r}\le \dim _\CC\varepsilon_{\bm\lambda}\cdot \left(\CC[\mathcal{\widetilde{{OP}}}_{n,r}]\otimes \det \right).
\end{equation}
In order to prove the reverse inequality, we will switch tactics and use inverse systems. Before doing so, we give an example of the construction process used in the proof of Lemma~\ref{lem: monomial count}.

\begin{example}
    To illustrate the above argument, let us look at an example of an $\llambda$-surviving $r$-ified ordered set partition being built step-by-step.

    Let's take the example of $r=3$, $n=15$,  $\llambda=(5,3\parallel3,2\parallel2)$ and $\ggamma=(3,2\parallel 1,2\parallel 1)$ which creates the following batches: $$(\underbrace{1,2,3,4,5}_{\lambda^{(1)}_1}\mid \underbrace{6,7,8}_{\lambda^{(1)}_2}\parallel\underbrace{9,10,11}_{\lambda^{(2)}_1}\mid\underbrace{12,13}_{\lambda^{(2)}_2}\parallel\underbrace{14,15}_{\lambda^{(3)}_1}).$$
The batches are inserted as follows.
\begin{align*}
    &\varnothing\\
    &(\;\varnothing\mid \bullet\mid\bullet\;)\\
    &(\;\varnothing\mid \bullet\circ\circ\mid\bullet\circ\;)\\
    &(\;\varnothing\mid 1^02^03^0\mid4^05^0\;)\\
     &(\;\varnothing\mid 1^02^03^0\mid\bullet\mid4^05^0\;)\\
     &(\;\varnothing\mid 1^02^03^0\circ\mid\bullet\circ\mid4^05^0\;)\\
     &(\;\varnothing\mid 1^02^03^06^0\mid7^08^0\mid4^05^0\;)\\
     &(\;\varnothing\mid \bullet\mid\bullet\mid 1^02^03^06^0\mid7^08^0\mid4^05^0\;)\\
     &(\;\circ\mid \bullet\mid\bullet\mid 1^02^03^06^0\mid7^08^0\mid4^05^0\;)\\
     &(\;9^1\mid 10^0\mid11^0\mid 1^02^03^06^0\mid7^08^0\mid4^05^0\;)\\
     &(\;9^1\circ\mid 10^0\circ\mid11^0\mid 1^02^03^06^0\mid7^08^0\mid4^05^0\;)\\
     &(\;9^112^1\mid 10^013^0\mid11^0\mid 1^02^03^06^0\mid7^08^0\mid4^05^0\;)\\
     &(\;9^112^1\mid 10^013^0\mid11^0\mid \bullet\mid1^02^03^06^0\mid7^08^0\mid4^05^0\;)\\
     &(\;9^112^1\circ\mid 10^013^0\mid11^0\mid \bullet\mid1^02^03^06^0\mid7^08^0\mid4^05^0\;)\\
      &(\;9^112^114^0\mid 10^013^0\mid11^0\mid 15^0\mid1^02^03^06^0\mid7^08^0\mid4^05^0\;)\\
\end{align*}
\end{example}
\subsection{Parabolic $\DDD$-operators and lower bound} In the previous subsection, we saw that the dimension of $\varepsilon_{\bm\lambda}\cdot SR_{n,r}$ cannot exceed the target quantity $\dim_\CC \varepsilon_\llambda \cdot (\CC[\widetilde{\OP}_{n,r}] \otimes \det)$. To show the reverse bound, one must produce sufficiently many linearly independent elements of $\varepsilon_{\bm\lambda}\cdot SR_{n,r}$. 
Due to the difficulty of working with elements of $SR_{n,r}$ (which are cosets as opposed to genuine superspace elements), we turn our attention to the harmonic space $\SH$. Since $SR_{n,r}$ and $\SH$ are isomorphic as bigraded $(\ZZ_r \wr \symm_n)$-modules, one has
\begin{equation}
    \dim_\CC \antisymm \cdot SR_{n,r} = \dim_\CC \antisymm \cdot \SH
\end{equation} and it will suffice to construct sufficiently many linearly independent elements of $\antisymm\cdot \SH$.

Our method for bounding $\antisymm \cdot \SH$ from below is a more elaborate version of the argument in Subsection \ref{subsec: inverse system} which takes the $r$-partition $\llambda$ into account. Starting with the element $\delta_{n,r}\in \SH$, we hope to construct more elements of $\SH$ by applying differential operators. However, since $\antisymm\cdot\delta_{n,r}$ is not necessarily non-zero, we need to modify our approach as in the following lemma. Let $\mathbf X_{\bm\lambda}$ denote the following monomial in $\CC[\mathbf x_n]$:
$$\mathbf X_{\bm\lambda}:=\prod_{i=1}^rx_{L_{i-1}+1}^{r-i}x_{L_{i-1}+2}^{r-i}\cdots x_{L_i}^{r-i}.$$As an example, if $n=10, r=3$ and $\bm\lambda=(2,1\parallel 3,2\parallel2)$ then \[\mathbf X_{\bm\lambda}=x_1^2x_2^2x_3^2\times x_4^1x_5^1x_6^1x_7^1x_8^1\times 1.\]
We write $\delta^\llambda_{n,r} \in \CC[\xx_n]$ for the polynomial
\begin{equation}
    \delta^\llambda_{n,r} := \varepsilon_\llambda \cdot (\XX_\llambda \odot \delta_{n,r}).
\end{equation}

\begin{lemma}\label{lem: SH element}
    Let $\bm\lambda\vdash_r n$ be an $r$-partition, and let $S^{\ZZ_r \wr \symm_n} = \CC[x_1,\dots,x_n]^{\ZZ_r\wr\symm_{\bm\lambda}}$ be the fixed space of $S = \CC[x_1,\dots,x_n]$ under the action of the subgroup $\ZZ_r\wr\symm_{\bm\lambda} \subseteq \ZZ_r \wr \symm_n$. Then 
    \begin{enumerate}
        \item the element $\delta^\llambda_{n,r}$ is non-zero and belongs to $\antisymm\cdot \SH$;
        \item $\antisymm\cdot\SH$ is closed under the operator $d_j$ for $j\equiv 1\pmod{r}$;
        \item $\antisymm\cdot\SH$ is closed under the operator $f\odot (-)$ whenever $f\in S^{\ZZ_r\wr\symm_{\bm\lambda}}$.
    \end{enumerate}
\end{lemma}
\begin{proof}
    To prove (1), consider the ``uncolored" parabolic subgroup $\symm_{\bm\lambda}$ defined as $$ \symm_{\bm\lambda}:=\left(\symm_{\lambda^{(1)}_1}\times \cdots\times \symm_{\lambda^{(1)}_{k_1}}\right)\times\cdots\times \left(\symm_{\lambda^{(r)}_1}\times\cdots\times\symm_{\lambda^{(r)}_{k_r}}\right)\subseteq \ZZ_r\wr\symm_n.$$This has a corresponding antisymmetrizer $$\antisymm^0:=\sum_{g\in \symm_{\bm\lambda}}\sign(g)\cdot g.$$ One can verify that this commutes with the operator $f\odot(-)$ whenever $f$ is invariant under $\symm_{\bm\lambda}$-action. Further, if $x_1^{a_1}\cdots x_n^{a_n}$ is a monomial in which $a_k\equiv (i-1)\pmod r$ whenever $k$ belongs to the $\lambda^{(i)}_j$-batch of $[n]$, $\antisymm\cdot x_1^{a_1}\cdots x_n^{a_n}=c_{\bm\lambda}\cdot\antisymm^0\cdot x_1^{a_1}\cdots x_n^{a_n}$, where $c_{\bm\lambda}$ is a non-zero constant depending only on $\bm\lambda$. We will call such a monomial \textit{$\bm\lambda$-stable}.

    Consider the following chain of equalities: $$\antisymm\cdot \left(\mathbf X_{\bm\lambda}\odot \delta_{n,r}\right)\doteq \antisymm^0\cdot \left(\mathbf X_{\bm\lambda}\odot \delta_{n,r}\right)=\mathbf X_{\bm\lambda}\odot\left(\antisymm^0\cdot\delta_{n,r}\right)\doteq \mathbf X_{\bm\lambda}\odot\delta_{n,r} $$where $\doteq$ denotes equality up-to non-zero scaling. One can verify that every monomial in $\mathbf X_{\bm\lambda}\odot \delta_{n,r}$ is $\bm\lambda$-stable, and therefore the first equality holds. Next, $\mathbf X_{\bm\lambda}$ is fixed under the action $\symm_{\bm\lambda}$, yielding the second equality. Finally, since $w\cdot \delta_{n,r}=\sign(w)\cdot\delta_{n,r}$ for any $w\in \symm_n$, $\antisymm^0\cdot \delta_{n,r}$ is simply $\delta_{n,r}$ up-to a non-zero scalar, and the third equality follows as well. 

    It remains to show $\mathbf X_{\bm\lambda}\odot\delta_{n,r}$ is non-zero, or equivalently (by Theorem~\ref{thm:steinberg}) $\mathbf X_{\bm\lambda}\ne 0$ as an element of the quotient $S/I_{n,r}$. 
    The vector space $S/I_{n,r}$ is the classical (polynomial) coinvariant ring associated to the complex reflection group $\ZZ_r \wr \symm_n$; it is known to have basis
    \[
    \{ x_1^{b_1} \cdots x_n^{b_n} \,:\, b_i < r \cdot (n-i)  \}.
    \]
    (This basis is obtained from Proposition~\ref{prop:colon-ideal-quotient-basis} at $J = \varnothing$ so that $f_{J,r} = 1$ after transforming the variables by $x_i \leftrightarrow x_{n-i}$.)
    Since $\mathbf X_{\bm\lambda}$ lies in this set of monomials, we obtain $\mathbf X_{\bm\lambda}\ne 0$ in the quotient ring $S/I_{n,r}$ and get $\mathbf X_{\llambda} \odot \delta_{n,r} \neq 0$.

    Parts (2) and (3) are immediate from Theorems \ref{thm:steinberg} and \ref{thm:swanson-wallach}, and the observation that $\antisymm$ commutes with the operators $d_{rj+1}$ and $f\odot(-)$ for $f\in S^{\ZZ_r\wr\symm_{\bm\lambda}}$.
    \end{proof}

    In Section \ref{sec:Basis} we applied the $\mathfrak D_{J,r}$ operators on $\delta_{n,r}$ in order to generate elements of $\SH$. To extend this idea in order to get elements of $\antisymm\cdot\SH$, we use a modified form of these $\DDD$-operators that will be indexed so-called translation sequences, defined below. 

    \begin{definition}
        Given an $r$-partition $\bm\lambda\vdash_r n$, a $\bm\lambda$-translation sequence is the sequence of sets $\TT=(T^{(1)}_1,\ldots, T^{(1)}_{k_1},\ldots, T^{(r)}_1,\ldots, T^{(r)}_{k_r})$ where $T^{(i)}_j$ is a subset of the $\lambda^{(i)}_j$-batch of $[n]$. We define $\bm\gamma(\TT)$ to be the sequence of cardinalities of $T^{(i)}_j$: $\bm\gamma(\TT)=\left(\#T^{(1)}_1,\ldots, \#T^{(1)}_{k_1},\cdots, \#T^{(r)}_1,\ldots, \#T^{(r)}_{k_r}\right).$

    \end{definition}

    In our previous example with $\bm\lambda=(2,1\parallel3,2\parallel2)$ that induces the batches $$(1,2\mid 3\parallel4,5,6\mid7,8\parallel9,10),$$ one possible $\bm\lambda$-translation sequence would be $$\TT=\left(\{1\},\{3\},\{4,6\},\varnothing,\{10\}\right),$$ and we would have $\bm\gamma(\TT)=(1,1\parallel2,0\parallel 1)$. By definition, $(\bm\lambda,\bm\gamma(\TT))$ is always a signed $r$-partition. 

    Next, we associate polynomial weights to $\bm\lambda$-translation sequences that, roughly speaking, measure how far away $\TT$ is from the Gale-maximal $\bm\lambda$-translation sequence with the same $\bm\gamma(\TT)$.

    \begin{definition}
    \label{def:translation-weight}
        Start with $\bm\lambda\vdash_rn$ and a $\bm\lambda$-translation sequence $\TT=(T^{(1)}_1,\ldots,  T^{(r)}_{k_r})$, and suppose $\bm\gamma(\TT)=(\gamma^{(1)}_1,\ldots,\gamma^{(r)}_{k_r})$. We associate a partition $\nu^{(i)}_j(\TT)$ with $\leq \gamma^{(i)}_j$ parts to each set $T^{(i)}_j$ as follows. If $T^{(i)}_j=\{t_1<\ldots<t_{\gamma^{(i)}_j}\}$, then 
        \begin{multline}\nu^{(i)}_j(\TT):= \\\left(L_{i-1}+\lambda^{(i)}_1+\cdots+\lambda^{(i)}_j-\gamma^{(i)}_j+1-t_1\ge \cdots\ge L_{i-1}+\lambda^{(i)}_1+\cdots+\lambda^{(i)}_j-t_{\gamma^{(i)}_j}\right).
        \end{multline}Given this, we attach a polynomial weight $\sss^r(\TT) \in S$ to the $\llambda$-translation sequence $\TT$ given by 
        \begin{equation}
            \sss^r(\TT):=\prod_{i=1}^r\prod_{j=1}^{k_i}s_{\nu^{(i)}_j}\left(x^r_{L_{i-1}+\lambda^{(i)}_1+\cdots+\lambda^{(i)}_j-\gamma^{(i)}_j+1}, \dots, x^r_{L_{i-1}+\lambda^{(i)}_1+\cdots+\lambda^{(i)}_j}\right).
        \end{equation}
    \end{definition}
    In other words, for each $\lambda^{(i)}_j$-batch, we obtain the partition $\nu^{(i)}_j(\TT)$ by componentwise subtracting the elements of $T^{(i)}_j$ from the sequence of the $\gamma^{(i)}_j$ biggest numbers in this batch (which is the Gale-maximal possible $T^{(i)}_j$ with the same cardinality $\# T^{(i)}_j = \gamma^{(i)}_j$). Then we form the Schur function corresponding to this partition on the variable set indexed by the $\gamma^{(i)}_j$ biggest numbers in this batch, raised to the $r$-th power. Finally, multiplying these Schur functions for each $\lambda^{(i)}_j$ gives $\sss^r(\TT)$. 

    An example might help clarify Definition~\ref{def:translation-weight}. Suppose $n = 20$, $r= 3$, and \[\llambda = ( 4,3 \mid \mid 4,4 \mid \mid 5) \vdash_3 20.\] The set $[n] = [20]$ is decomposed by $\llambda$ into the batches
    \[
    (1,2,3,4 \mid 5,6,7 \mid\mid8,9,10,11 \mid 12,13,14,15 \mid \mid 16,17,18,19,20 ).
    \]
    A possible $\llambda$-translation sequence is given by
    \[
    \TT = \left(T^{(1)}_1, T^{(1)}_2, T^{(2)}_1, T^{(2)}_2, T^{(3)}_1\right) =  ( \{1,3\}, \{5\}, \varnothing, \{13,15\}, \{16,17,20\}).
    \]
    Counting the sizes of these sets gives
    \[
    \ggamma(\TT) = (2,1 \mid \mid 0, 2 \mid \mid 3).
    \]
    The $\llambda$-translation sequence $\TT$ gives rise to the partitions $\nu^{(i)}_j(\TT)$ given by
    \[
    \nu^{(1)}_1(\TT) = (2,1), \quad \nu^{(1)}_2(\TT) = (3), \quad \nu^{(2)}_1(\TT) = \varnothing, \quad \nu^{(2)}_2(\TT) = (1,0), \quad  \nu^{(3)}_1(\TT) = (2,2,0),
    \]
    where we have padded partitions with zeros to represent $\nu^{(i)}_j(\TT)$ as a sequence of length $\gamma(\TT)^{(i)}_j = \# T^{(i)}_j$. Finally, the polynomial weight $\sss^r(\TT)$ attached to $\TT$ is given by
    \[
    \sss^r(\TT) = s_{2,1}(x_3^3,x_4^3) \times s_3(x_7^3) \times 1 \times s_{1,0}(x_{14}^3,x_{15}^3) \times s_{2,2,0}(x_{18}^3,x_{19}^3,x_{20}^3)
    \]
    where the exponents $x_i^3$ occur because $r=3$.

   Recall the $k$-row power matrix $P_{k,r}(\mathbf y)$ defined in Section~\ref{sec:Basis} with entries $P_{k,r}(\yy)_{i,j}=y_i^{r(n-j)+1}$; we will use these matrices again in this section, but the factor matrices will need to be altered to take the parabolic structure given by $\llambda$ into account.
   Define the \textit{$\llambda$-factor matrix} $F_{k,r}(\xx,\yy,\bm\lambda)$ as the $k\times n$ matrix whose entries are given by $$F_{k,r}(\xx,\yy,\bm\lambda)_{i,j}:=y_i^{r(L_{p-1}+\lambda^{(p)}_1+\cdots\lambda^{(p)}_{q}-j)+1}\prod_{m=L_{p-1}+\lambda^{(p)}_1+\cdots\lambda^{(p)}_{q}+1}^n(y_i^r-x_m^r)$$where $(p,q)$ is the lex-minimal pair satisfying $r(L_{p-1}+\lambda^{(p)}_1+\cdots\lambda^{(p)}_{q})\ge rj-1$.
   
   As an example, suppose $\bm\lambda=(2,1\parallel 3,2\parallel 2)$, and assume $r = k = 3$. Then the power matrix is $$P_{3,3}(\yy)=\left(\begin{array}{cc|c||ccc|cc||cc}
      y_1^{28}  &y_1^{25}  &y_1^{22}&y_1^{19}&y_1^{16}&y_1^{13}&y_1^{10}&y_1^7&y_1^4&y_1^1  \\
      y_2^{28} &y_2^{25}&y_2^{22}&y_2^{19}&y_2^{16}&y_2^{13}&y_2^{10}&y_2^7&y_2^4&y_2^1\\
      y_3^{28} &y_3^{25}&y_3^{22}&y_3^{19}&y_3^{16}&y_3^{13}&y_3^{10}&y_3^7&y_3^4&y_3^1
   \end{array}\right),$$where we have used bars and doubled bars as usual to show the $\lambda^{(i)}_j$-batches. The full explicit form of $F_{3,3}(\xx,\yy,\bm\lambda)$ is cumbersome to write down, but for the sake of example, its $(2,5)$-entry is $$F_{3,3}(\xx,\yy,\llambda)_{2,5} = y_2^4(y_2^3-x_7^3)(y_2^3-x_8^3)(y_2^3-x_9^3)(y_2^3-x_{10}^3).$$

Every entry of the factor matrix can be expressed as a polynomial in $y_i$, where the coefficients are elementary symmetric functions in $x_i^r$, where the $x$-variables indices come from subsequent $\lambda^{(i)}_j$-batches from the entries in $[n]$. Observation \ref{obs:matrix-mult} generalizes as follows.

\begin{observation}\label{obs: lambda matrix mult}
    There exists a lower triangular $n \times n$ matrix $C(\llambda)$ with 1s on the diagonal and entries in $S^{\ZZ_r\wr \symm_{\bm\lambda}}$ so that 
    \begin{equation}
        F_{k,r}(\xx,\yy,\bm\lambda) = P_{k,r}(\yy) \cdot C(\llambda).
    \end{equation}
\end{observation}

Let $\TT$ be a $\bm\lambda$-translation sequence. Let $T\subseteq [n]$ denote the disjoint union of the sets in $\TT$, i.e. $T = \bigsqcup_{i,j} T^{(i)}_j$. Construct the \textit{augmented factor matrix} as $$\widetilde{F_{r}}(\xx,\yy,\bm\lambda,\TT)=\left(\begin{array}{c}
    F_{\#T,r}(\xx,\yy,\bm\lambda)\\
    \hline
    E_T
\end{array}\right)$$ where $E_T$ is the $(n-\#T)\times n$ row-echelon $0,1$-matrix whose pivot columns are those indexed by $[n]-T$.

For example, suppose $\llambda = (2,1 \mid \mid 3,2 \mid \mid 2) \vdash_3 10$ so that we have the batches
\[
( 1,2 \mid 3 \mid \mid 4,5,6 \mid 7,8 \mid \mid 9,10).
\]
A possible $\llambda$-translation sequence $\TT$ is given by 
\[
\TT = (T^{(1)}_1 \mid T^{(1)}_2 \mid \mid T^{(2)}_1 \mid T^{(2)}_2 \mid \mid T^{(3)}_1) = (2 \mid \varnothing \mid 4,6 \mid 7,8 \mid 9)
\]
and for this $\TT$ we have $T = \bigsqcup_{i,j} T^{(i)}_j = \{2,4,6,7,8,9\}$ so that $[n] - T = [10] - T = \{1,3,5,10\}$ and
\[
E_T = \left(\begin{array}{c c | c || c c c | c c | | c c }
    1 & 0 & 0 & 0 & 0 & 0 & 0 & 0 & 0  & 0 \\
    0 & 0 & 1 & 0 & 0 & 0 & 0 & 0 & 0  & 0 \\
    0 & 0 & 0 & 0 & 1 & 0 & 0 & 0 & 0  & 0 \\
    0 & 0 & 0 & 0 & 0 & 0 & 0 & 0 & 0  & 1
\end{array}\right)
\]
where we have put vertical lines in the matrix $E_T$ in the positions determined by $\llambda$.

For any $J\subseteq [n]$ with the same cardinality as $T$, one may form the matrix $\widetilde{F_r}(\xx,\xx_J,\bm\lambda,\TT)$ by replacing the $y_i$'s by variables from the set $\{x_j:j\in J\}$; as before, the order of substitution will be of no concern. Then a triangularity statement similar to Lemma \ref{lem:augmented-determinant} holds here as well.

\begin{lemma}\label{lem:PPP-triangularity}
    Using the above notations for $\bm\lambda$, $\TT$, $T$ and $\bm\gamma(\TT)$, define $\PPP_{T,J}:=\det\widetilde{F_r}(\xx,\xx_J,\bm\lambda,\TT)$. Then
    \begin{enumerate}
        \item we have $\PPP_{T,J}=0$ whenever $J\not\le_\Gale J(\bm\lambda,\bm\gamma)$;
        \item we have $$\PPP_{T,J(\bm\lambda,\bm\gamma)}\doteq f_{J(\llambda,\ggamma),r}\cdot\sss^r(\TT),$$where $\doteq$ denotes equality up-to a non-zero scalar.
    \end{enumerate}
\end{lemma}
\begin{proof}
    It is possible to use an argument similar to the proof of Lemma~\ref{lem:augmented-determinant}. However, there is a quicker proof using the analogous result for $r=1$ already proved in \cite[Lem. 4.9]{MRW}. Note that the proof of the cited lemma never uses the fact that the parts of the $\lambda$ are decreasing, so we can apply it to the sequences $\lambda=(\lambda^{(1)}_1,\ldots, \lambda^{(r)}_{k_r})$ and $\gamma=(\gamma^{(1)}_1,\ldots, \gamma^{(r)}_{k_r})$. Our Lemma now follows by setting $x_i\mapsto x_i^r$ and dividing both sides by $x_1^{r_1}\cdots x_n^{r-1}$.
\end{proof}
Let $C(\bm\lambda)$ denote the lower-triangular matrix with 1's on the diagonal seen to exist in Observation~\ref{obs: lambda matrix mult}. Thanks to Observation~\ref{obs: lambda matrix mult}, the matrix $C(\llambda)$ is invertible over the ring $S = \CC[\xx_n]^{\ZZ_r \wr \symm_n}$. The matrix $H(\llambda)$ defined by 
\begin{equation}
H(\bm\lambda):=E\cdot C(\bm\lambda) 
\end{equation}
has entries in $\CC[\xx_n]^{\ZZ_r\wr\symm_{\bm\lambda}}$. One has the block matrix equality \begin{equation}
    \widetilde{F}_{r}(\xx,\yy,\bm\lambda,\TT) \cdot C(\llambda)^{-1} = 
    \left(
    \begin{array}{c}
            F_{\#T,r}(\xx,\yy,\bm\lambda) \\ \hline E_J
    \end{array}
    \right) \cdot C(\llambda)^{-1} =
    \left(
    \begin{array}{c}
        P_{k,r}(\yy) \\ \hline H(\llambda).
    \end{array}\right)
\end{equation} 
We are now ready to define the parabolic $\DDD$-operators. Recall the `reversing and rescaling' function $\rho(\cdot)$, and the notation for matrix minors $\Delta_U$ from Definition \ref{def:D-operator}

\begin{definition}
    \label{def:parabolic D-operator}
    Using notation from above, let $\DDD_{\llambda,r}^\TT: \Omega \to \Omega$ be given by
    \begin{equation}
        \DDD^\TT_{\llambda,r}(f) := \sum_{\# U = n-\#T} (-1)^{\sum U} \Delta_U(H(\llambda)) \odot d_{\rho(U)}(f).
    \end{equation}
\end{definition}

It follows from Lemma~\ref{lem: SH element} that the superspace element \[\DDD_{\llambda,r}^\TT(\delta^\llambda_{n,r}) = \DDD_{\llambda,r}^\TT( \varepsilon_\llambda \cdot (\XX_\llambda \odot \delta_{n,r}))\] lies in $\varepsilon_\llambda \cdot \SH$. In analogy with Lemma~\ref{lem:D-leading}, the following result shows that these elements have convenient triangularity properties.

\begin{lemma}\label{lem:element gen}
    If $1\not\in T^{(1)}_1$, then $\DDD_{\llambda,r}^{\TT}(\delta^\llambda_{n,r})$ is a nonzero element of $\antisymm\cdot\SH$ with unique Gale-maximal fermionic monomial $\theta_{J(\llambda,\ggamma)}$. Furthermore, the coefficient of $\theta_{J(\llambda,\ggamma)}$, up to a nonzero scaling factor, equals $$\left(\XX_\llambda\cdot\sss^r(\TT)\cdot f_{J(\llambda,\ggamma),r}\right)\odot\delta_{n,r}.$$
\end{lemma}

If $\lambda^{(1)}$ happens to be the empty partition, $T^{(1)}_1$ will not exist. In this case, the condition $1\not\in T^{(1)}_1$ is assumed to vacuously hold true.
\begin{proof}
    The same argument used to prove Lemma~\ref{lem:D-leading} shows that the coefficient of $\theta_J$ in $\DDD_{\llambda,r}^\TT(\delta'_{n,r})$ equals $\PPP_{T,J}\odot\delta^\llambda_{n,r}$, which, according to Lemma~\ref{lem:PPP-triangularity}, vanishes whenever $J\not\le_\Gale J(\llambda,\ggamma)$. In case $J=J(\llambda,\ggamma)$, by Lemma~\ref{lem:PPP-triangularity} this coefficient is equal to $\left(f_{J(\llambda,\ggamma),r}\cdot\sss^r(\TT)\right)\odot \delta^\llambda_{n,r}$ up to non-zero scaling, which equals \begin{equation}
    \label{eq:desired-nonzero}\left(f_{J(\llambda,\ggamma),r}\cdot\sss^r(\TT)\right)\odot(\XX_\llambda\odot\delta_{n,r})=\left(\XX_\llambda\cdot\sss^r(\TT)\cdot f_{J(\llambda,\ggamma),r}\right)\odot\delta_{n,r}.\end{equation} 
    It remains to show \eqref{eq:desired-nonzero} is nonzero.

    By Steinberg's Theorem~\ref{thm:steinberg}, the polynomial \eqref{eq:desired-nonzero} is nonzero if and only if  
    \begin{equation}\XX_\llambda\cdot \sss^r(\TT) \in S = \CC[x_1,\dots,x_n]\end{equation} descends to a nonzero element in the quotient $S/\left(I_{n,r}:f_{J(\llambda,\ggamma),r}\right)$. However, we derived a monomial basis $\AAA_{n,r}(J(\llambda,\ggamma),r)$ for this quotient in Proposition~\ref{prop:colon-ideal-quotient-basis} given by 
    \begin{equation}
        \AAA_{n,r}(J(\llambda,\ggamma),r) = \{ x_1^{a_1} \cdots x_n^{a_n} \,:\, a_i \le \stair(J(\llambda,\ggamma),r)_i \}.
    \end{equation}
    The monomial expansion of $\XX_\llambda \cdot \sss^r(\TT)$ has the following crucial property.

    {\bf Claim:}
    {\em All monomials that appear in the expansion of $\XX_\llambda\cdot \sss^r(\TT)$ as a polynomial in $S = \CC[x_1,\dots,x_n]$ lie in the set $\AAA_{n,r}(J(\llambda,\ggamma),r)$.}

    To see why the Claim is true, consider an index $\alpha$ in the $\lambda^{(i)}_j$-batch of $[n]$. If $\alpha$ is in the among the $\lambda^{(i)}_j-\gamma^{(i)}_j$ smallest elements in this batch, the polynomial $\sss^r(\TT)$ does not involve the variable $x_{\alpha}$. It follows that 
    \begin{equation}
        \text{degree of $x_\alpha$ in $\XX_\llambda \cdot \sss^r(\TT)$} \leq 
        \begin{cases}
            r-1 & i =1, \\
            r-2 & \text{otherwise}.
        \end{cases}
    \end{equation}
    In either case this degree is $\leq \stair(J(\llambda,\ggamma),r)_\alpha$.

    If instead $\alpha$ is among the $\gamma^{(i)}_j$ biggest entries of this batch, the largest exponent $x_{\alpha}$ can have in $\sss^r(\TT)$ is $r$ times the size of the largest part of $\nu^{(i)}_j$, which is at most $\lambda^{(i)}_j-\gamma^{(i)}_j$ (or $\lambda^{(i)}_j-\gamma^{(i)}_j-1$ if $i=1$, due to the $1\not\in T^{(1)}_1$ assumption). The exponent of $x_{\alpha}$ in $\XX_\llambda$ is $r-i$ by definition. Therefore 
    \begin{equation}
        \text{degree of $x_\alpha$ in $\XX_\llambda \cdot \sss^r(\TT)$} \leq \begin{cases}
            r-1+r(\lambda^{(i)}_j-\gamma^{(i)}_j-1) & i=1, \\
            r-2+r(\lambda^{(i)}_j-\gamma^{(i)}_j) & \text{otherwise.}
        \end{cases}
    \end{equation}
    Either of these numbers is $\leq \stair(J(\llambda,\ggamma),r)_\alpha$.

    Thus the exponent sequences all monomials of $\XX_\llambda\cdot \sss^r(\TT)$ fit under the staircase $\stair(J(\llambda,\ggamma),r)$ and the Claim follows.
    Proposition~\ref{prop:colon-ideal-quotient-basis} implies that $\XX_\llambda \cdot \sss^r(\TT)$ is nonzero in the quotient ring $S/(I_{n,r}:f_{J(\llambda,\ggamma),r})$ if and only if $\XX_\llambda \cdot \sss^r(\TT)$ is nonzero as a polynomial in $S$. Since we certainly have $\XX_\llambda \cdot \sss^r(\TT) \neq 0$ in $S$, the proof is complete.
\end{proof}

The previous lemma gives us a way to generate elements of $\antisymm \cdot \SH$ with leading fermionic term $\left(\left(\XX_\llambda\cdot\sss^r(\TT)\cdot f_{J(\llambda,\ggamma),r}\right)\odot\delta_{n,r}\right)\cdot\theta_{J(\llambda,\ggamma)}$. Our goal is to create sufficiently many linearly independent elements of $\antisymm \cdot \SH$ by combining terms of this form for varying $\llambda$-translation sequences $\TT$.

Consider all the different polynomials $\XX_\llambda\cdot\sss^r(\TT)$ one could conjure by varying $\TT$. All such polynomials are products of the form $$\prod_{i=1}^rx_{L_{i-1}+1}^{r-i}x_{L_{i-1}+2}^{r-i}\cdots x_{L_i}^{r-i}\prod_{j=1}^{k_r}s_{\nu^{(i)}_j}\left(x^r_{L_{i-1}+\lambda^{(i)}_1+\cdots+\lambda^{(i)}_j-\gamma^{(i)}_j+1}, \cdots, x^r_{L_{i-1}+\lambda^{(i)}_1+\cdots+\lambda^{(i)}_j}\right).$$The partitions $\nu^{(i)}_j$ fit inside the boxes $$\nu^{(1)}_j\subseteq\left(\lambda^{(1)}_j-\gamma^{(1)}_j-1\right)^{\gamma^{(1)}_j},\text{ and }\nu^{(i)}_j\subseteq\left(\lambda^{(i)}_j-\gamma^{(i)}_j\right)^{\gamma^{(i)}_j}\text{ for }i>1,$$and all such partitions can appear depending on $\TT$. These considerations lead us to the following definition.

\begin{definition}
    For a signed $r$-partition $(\llambda,\ggamma)$ with $\llambda\vdash_rn$, let $Z_{\llambda,\ggamma} \subseteq S$ be the set of polynomials of the form $$\prod_{i=1}^rx_{L_{i-1}+1}^{r-i}\cdots x_{L_i}^{r-i}\prod_{j=1}^{k_r}s_{\nu^{(i)}_j}\left(x^r_{L_{i-1}+\lambda^{(i)}_1+\cdots+\lambda^{(i)}_j-\gamma^{(i)}_j+1}, \cdots, x^r_{L_{i-1}+\lambda^{(i)}_1+\cdots+\lambda^{(i)}_j}\right)$$ where the partitions $\nu^{(i)}_j$ satisfy 
    \[\nu^{(1)}_j\subseteq\left(\lambda^{(1)}_j-\gamma^{(1)}_j-1\right)^{\gamma^{(1)}_j} \quad \text{and}  \quad  \nu^{(i)}_j\subseteq\left(\lambda^{(i)}_j-\gamma^{(i)}_j\right)^{\gamma^{(i)}_j} \text{ for $i>1$}.\] 
    Let $R_{\llambda,\ggamma}$ be the $S^{\ZZ_r\wr\symm_\llambda}$-submodule of $S/\left(I_{n,r}:f_{J(\llambda,\ggamma),r}\right)$ generated by (the image of) $Z_{\llambda,\ggamma}$.
\end{definition}

The following lemma tells us that the various modules $R_{\llambda,\ggamma}$   may be used to bound $\dim_\CC \antisymm\cdot\SH$ from below.

\begin{lemma}\label{lem:consolidate basis}
    We have $$\dim_\CC\antisymm\cdot\SH\ge \sum_{\ggamma\le \llambda}\dim_\CC R_{\llambda,\ggamma}.$$
\end{lemma}
\begin{proof}
    For a fixed $\ggamma \leq \llambda$, let $N = \dim_\CC R_{\llambda,\ggamma}$ and suppose $g_1,\dots,g_N \in S$ are $S^{\ZZ_r \wr \symm_n}$-linear combinations of the elements of $Z_{\llambda,\ggamma}$ which descend to a basis of $R_{\llambda,\ggamma} \subseteq S/(I_{n,r}:f_{J(\llambda,\ggamma),r}).$
    Since the operation $f\odot(-)$ for $f\in S^{\ZZ_r\wr\symm_\llambda}$ is preserves $\antisymm \cdot \SH$, by Lemma~\ref{lem:element gen} we can obtain $N$ elements in $\antisymm \cdot SI_{n,r}^\perp$ which only involve fermionic monomials  which are $\leq_\Gale \theta_{J(\llambda,\ggamma})$ such that the coefficients of $\theta_{J(\llambda,\ggamma})$ in these elements of $\antisymm \cdot SI_{n,r}^\perp$ are given by 
    \begin{equation}
    \label{eq:N-elements}
            (g_1\cdot f_{J(\llambda,\ggamma),r})\odot\delta_{n,r} \,, \quad \ldots \quad ,(g_N\cdot f_{J(\llambda,\ggamma),r}) \odot \delta_{n,r}.
    \end{equation}
Since $g=0$ in $S/\left(I_{n,r}:f_{J(\llambda,\ggamma),r}\right)$ if and only if $(g\cdot f_{J(\llambda,\ggamma),r})\odot\delta_{n,r}=0$ in $S$, the polynomials \eqref{eq:N-elements} are linearly independent in $S$. We therefore obtain $N$ linearly independent elements of $\antisymm \cdot \SH$, each with unique Gale-maximal fermionic monomial $\theta_{J(\llambda,\ggamma)}$. Performing this procedure for different choices of $\ggamma$ leads to elements of $\antisymm \cdot SI_{n,r}^\perp$ with different fermionic leading terms $\theta_{J(\llambda,\ggamma)}$, so varying $\ggamma$ over all possible $\ggamma \leq \llambda$ leads to $\sum_{\ggamma \leq \llambda} \dim_\CC R_{\llambda,\ggamma}$ linearly independent elements of $\antisymm\cdot\SH$, proving the lemma.
\end{proof}

The problem of bounding $\varepsilon_\llambda \cdot SI_{n,r}^\perp$ from below  reduces to showing $\dim_\CC R_{\llambda,\ggamma}$ is sufficiently large. The analysis of the polynomials in $Z_{\llambda,\ggamma}$ generating $R_{\llambda,\ggamma}$ simplifies dramatically if we focus on the variable set indexed by each $\lambda^{(i)}_j$. The following definitions helps us do precisely that.

\begin{definition}
\label{def:L-set}
    For integers $m,k,t\ge 0$ and $i\in [r]$, define a set of polynomials \[\LLL_i(m,k,t)\subseteq\CC[x_1,\dots,x_m]\] as follows:
    $$\LLL_i(m,k,t):=\{e_{\mu}(x_1^r,\ldots,x_m^r)\cdot x_1^{r-i}\cdots x_m^{r-i} \cdot s_{\nu}(x^r_{m-k+1},\ldots, x^r_m):\nu\subseteq (m-k)^k,\;\mu\subseteq (m)^{t_{\nu,i}}\},$$
    where $$t_{\nu.i}=\begin{cases}
         t-1 & i=1\text{ and }\nu_1=m-k,\\ t &\text{ otherwise.}
    \end{cases}\;\;$$
\end{definition}
Roughly speaking, $\LLL_i(m,k,t)$ represents the kind of polynomials that must be multiplied together for each variable set indexed by the $\lambda^{(i)}_j$-th batches of $[n]$ in order to get basis elements of $R_{\llambda,\ggamma}$. By multiplying together all possible choices for polynomials for each such batch, we will obtain a basis for $R_{\llambda,\ggamma}$. Before doing so, we will prove some useful properties of the $\LLL_i(m,k,t)$.

\begin{lemma}\;\label{lem: L lin indep}
\begin{enumerate}
    \item For any monomial $x_1^{b_1}\cdots x_m^{b_m}$ appearing in an element of $\LLL_i(m,k,t)$ the exponent vector $(b_1,\ldots, b_m)$ is componentwise bounded from above by  $$ (r(t+1)-1,r(t+2)-1,\ldots, r(t+m-k)-1,\underbrace{r(t+m-k+1)-2,\ldots,r(t+m-k+1)-2}_{k\text{ copies}}).$$
    \item The set $\LLL_i(m,k,t)$ is linearly independent.
    \item We have $$\#\LLL_i(m,k,t)=\begin{cases}
        \dbinom{m-k+t}{m-k}\dbinom{m+t-1}{k} & i=1\\
        \dbinom{m-k+t}{m-k}\dbinom{m+t}{k} & i>1
    \end{cases}$$
\end{enumerate}  
\end{lemma}
\begin{proof}
    Part (1) is a simple counting problem: the highest possible exponent of $x_i$ from $e_{\mu}(x_1^r,\ldots,x_m^r)$ is $r\cdot \mu'_1$, and the highest possible exponent of $x_i$ appearing in $s_{\nu}(x^r_{m-k+1},\ldots, x^r_m)$ is $r\cdot \nu_1$. Since the elements of the sets $\LLL_i(m,k,t)$ are chosen such that $r\cdot \mu'_1+(r-i)\le r(t+1)-1$ and $r\cdot\mu'_1+(r-i)+r\cdot\nu_1\le r(t+m-k+1)-2\cdot$, the claim follows.

    For (2), note that the elements $s_{\nu}(x^r_{m-k+1},\ldots, x^r_m)$ for varying $\nu \subseteq (m-k)^k$ are clearly linearly independent in $S$. Since each monomial in the expansion of $s_{\nu}(x^r_{m-k+1},\ldots, x^r_m)$ lies in the basis $\AAA_{n,r}(\varnothing,r)$ of $S/I_{m,r}$, the elements $s_{\nu}(x^r_{m-k+1},\ldots, x^r_m)$ as $\nu$ varies over partitions $\nu \subseteq (m-k)^k$ therefore descend to a linearly independent subset of $S/I_{m,r}$. The ideal $I_{m,r}$ is a complete intersection and generated by the regular sequence $e_1(x_1^r,\ldots, x_m^m),\ldots,e_m(x_1^r,\ldots, x_m^m).$ 
    The set of polynomials 
    \begin{equation}
    \label{eq:linearly-independent-set-1}
    e_1^{d_1}(x_1^r,\ldots, x_m^m)\cdots e_m^{d_m}(x_1^r,\ldots, x_m^m)\cdot s_{\nu}(x^r_{m-k+1},\ldots, x^r_m) \quad \quad \nu \subseteq (m-k)^k, \, \, d_1,\dots,d_m \geq 0
    \end{equation}
    is therefore linearly independent in $S$. Mulyiplying the polynomials in \eqref{eq:linearly-independent-set-1} by the monomial $x_1^{r-i}\cdots x_m^{r-i}$ gives rise to a new linearly indepdenent subset of $S$ given by
    \begin{multline}
    \label{eq:linearly-independent-set-2}
        (x_1^{r-i}\cdots x_m^{r-i}) \cdot e_1^{d_1}(x_1^r,\ldots, x_m^m)\cdots e_m^{d_m}(x_1^r,\ldots, x_m^m)\cdot s_{\nu}(x^r_{m-k+1},\ldots, x^r_m)  \\
        \nu \subseteq (m-k)^{k}, \, \, d_1, \dots, d_m \geq 0.
    \end{multline}
    By definition, the set $\LLL_i(m,k,t)$ is a subset of the polynomials in \eqref{eq:linearly-independent-set-2}, and (2) follows.

    For (3), when $i>1$ there are $\binom{m}{k}$ choices for $\nu$ and $\binom{m+t}{t}$ choices for $\mu$ in Definition~\ref{def:L-set}, so the total number of choices is $$\binom{m}{k} \binom{m+t}{t}=\binom{m-k+t}{m-k}\binom{m+t}{k}$$ which agrees with (3) in the case $i>1$. For $i=1$ we must subtract the cases where $\nu_1=m-k$ and $\mu'_1=t$. The number of ways this could happen is $\binom{m-1}{k-1}\binom{m+t-1}{t}=\binom{m-k+t}{m-k}\binom{m+t-1}{k-1}$, so the remaining number of choices for $i=1$ is $$\binom{m-k+t}{m-k}\left\{\binom{m+t}{k}-\binom{m+t-1}{k-1}\right\}=\binom{m-k+t}{m-k}\binom{m+t-1}{k}$$
    which agrees with (3) in the case $i=1$.
\end{proof}

We are ready to show how to combine elements from $\LLL_i(m,k,t)$ to generate elements from $R_{\llambda,\ggamma}$. Note that the sequence $\stair(J(\llambda,\ggamma),r)$ can be divided into batches corresponding to each part in $\llambda$, where the batch corresponding to $\lambda^{(i)}_j$ looks like $$r(t^{(i)}_j+1)-1,r(t^{(i)}_j+2)-1,\ldots, r(t^{(i)}_j+\lambda^{(i)}_j-\gamma^{(i)}_j)-1,\underbrace{r(t^{(i)}_j+\lambda^{(i)}_j-\gamma^{(i)}_j+1)-2,\ldots}_{\gamma^{(i)}_j\text{ copies}}$$where 
\begin{equation}\label{tij formula}
    t^{(i)}_j=L_{i-1}-G_{i-1}+\lambda^{(i)}_1+\cdots +\lambda^{(i)}_{j-1}-\gamma^{(i)}_1+\cdots +\gamma^{(i)}_{j-1}.
\end{equation}

Let $\LLL_i\left(X^{(i)}_j,\gamma^{(i)}_j,t^{(i)}_j\right)$ denote the result of substituting the $\lambda^{(i)}_j$ variables $$x_{L_{i-1}+\lambda^{(i)}_1+\cdots +\lambda^{(i)}_{j-1}+1},\ldots, x_{L_{i-1}+\lambda^{(i)}_1+\cdots +\lambda^{(i)}_{j}}$$ into $\LLL_i\left(\lambda^{(i)}_j,\gamma^{(i)}_j,t^{(i)}_j\right)$. Consider the set 
\begin{equation}\label{EEE defn}
    \EEE(\llambda,\ggamma):=\left\{\prod_{i=1}^r\prod_{j=1}^{k_i}f^{(i)}_j:f^{(i)}_j\in \LLL_i\left(X^{(i)}_j,\gamma^{(i)}_j,t^{(i)}_j\right)\right\}.
\end{equation}

\begin{lemma}
\label{lem:E-linearly-independent}
    The set $\EEE(\llambda,\ggamma)$ descends to a linearly independent set in $R_{\llambda,\ggamma}$.
\end{lemma}
\begin{proof}
    By construction, elements in $\EEE(\llambda,\ggamma)$ lie in $S^{\ZZ_r\wr \symm_\llambda}\cdot Z_{\llambda,\ggamma}$, so we only need to show this set is linearly independent in $S/\left(I_{n,r}:f_{J(\llambda,\ggamma),r}\right)$. However, by Lemma \ref{lem: L lin indep}, $\EEE(\llambda,\ggamma)$ is linearly independent in $S$, and each monomial appearing in an element of this set lies under the staircase $\stair(J(\llambda,\ggamma),r)$ and thus is a basis element for $S/\left(I_{n,r}:f_{J(\llambda,\ggamma),r}\right)$ as in Proposition~\ref{prop:colon-ideal-quotient-basis}. These two facts combined prove the lemma.
\end{proof}

We have thus laid out a procedure to generate a large collection of linearly independent elements in $\antisymm\cdot\SH$: start with $\EEE(\llambda,\ggamma)$ for each $\ggamma$, use Lemma~\ref{lem:E-linearly-independent} to create linearly independent elements of $R_{\llambda,\ggamma}$, and finally use Lemma~\ref{lem:consolidate basis} to consolidate them over varying $\ggamma$ to obtain a set of linearly independent elements in $\antisymm\cdot\SH$. It remains to verify this has enough elements to provide a satisfactory lower bound. The following lemma confirms this.

\begin{lemma}\label{lem: dim equal}
    We have \begin{equation}\label{eq: dim equal}
    \dim\antisymm\cdot SR_{n,r}= \dim \varepsilon_{\bm\lambda}\cdot \left(\CC[\mathcal{\widetilde{{OP}}}_{n,r}]\otimes \det \right).
\end{equation}
\end{lemma}
\begin{proof}
    Using our count for $\LLL_i(m,k,t)$ from Lemma~\ref{lem: L lin indep} together with \eqref{tij formula} and \eqref{EEE defn}, we see that 

\begin{multline*}
        \#\EEE(\llambda,\ggamma)=\prod_{j=1}^{k_1}\binom{\lambda^{{(1)}}_1+\cdots+\lambda^{(1)}_j-\gamma^{(1)}_1-\cdots-\gamma^{(1)}_j}{\lambda^{(1)}_j-\gamma^{(1)}_{j}}\binom{\lambda^{{(1)}}_1+\cdots+\lambda^{(1)}_j-\gamma^{(1)}_1-\cdots-\gamma^{(1)}_{j-1}-1}{\gamma^{(1)}_j}\\\times\prod_{i=2}^r\left(\prod_{j=1}^{k_i}\binom{L_{i-1}-G_{i-1}+\lambda^{{(i)}}_1+\cdots+\lambda^{(i)}_j-\gamma^{(i)}_1-\cdots-\gamma^{(i)}_j}{\lambda^{(i)}_j-\gamma^{(i)}_{j}}\right.\\
        \left.\times\binom{L_{i-1}-G_{i-1}+\lambda^{{(i)}}_1+\cdots+\lambda^{(i)}_j-\gamma^{(i)}_1-\cdots-\gamma^{(i)}_{j-1}}{\gamma^{(i)}_j}\right).
    \end{multline*}
Varying $\ggamma$ and using the reasoning in the paragraph immediately preceding this lemma, we obtain $$\sum_{\ggamma\le\llambda}\#\EEE(\llambda,\ggamma)$$linearly independent elements in $\antisymm\cdot\SH$. However, this coincides with our
earlier count for $\dim \varepsilon_{\bm\lambda}\cdot \left(\CC[\mathcal{\widetilde{{OP}}}_{n,r}]\otimes \det \right)$ obtained in the proof of Lemma~\ref{lem: monomial count}. Thus we have the inequality\begin{equation}\label{eq: lower bound}
    \dim\antisymm\cdot \SH\ge\dim \varepsilon_{\bm\lambda}\cdot \left(\CC[\mathcal{\widetilde{{OP}}}_{n,r}]\otimes \det \right).\end{equation}
The lemma now follows readily from \eqref{eq: upper bound}, \eqref{eq: lower bound} and the fact that $SR_{n,r}$ and $\SH$ are isomorphic as $\ZZ_r\wr\symm_n$-modules.
\end{proof}

Combining Lemma~\ref{lem:parabolic-spanning-set}, Lemma~\ref{lem: monomial count}, and Lemma~\ref{lem: dim equal} gives a vector space basis of $\varepsilon_\llambda \cdot SR_{n,r}$.

\begin{corollary}
    \label{cor:parabolic-basis}
    For any $r$-partition $\llambda \vdash n$, the set of superspace elements
    \[
    \bigsqcup_{\ggamma\leq\llambda} \varepsilon_\llambda \cdot (\AAA(\llambda,\ggamma) \cdot \theta_{J(\llambda,\ggamma)})
    \]
    descends to a basis of $\varepsilon_\llambda \cdot SR_{n,r}$.
\end{corollary}



\subsection{Module isomorphisms}
We are ready to establish our combinatorial model \eqref{combmodel} for the $\ZZ_r \wr \symm_n$-module $SR_{n,r}$. With the machinery we have developed so far, we need only invoke the appropriate lemmas.

\begin{theorem}\label{thm: module iso}
    For any $n,r \geq 1$ we have the isomorphism of ungraded $\ZZ_r \wr \symm_n$-modules
    \begin{equation*}
    SR_{n,r}\cong_{\ZZ_r\wr \mathfrak S_n}\CC[\mathcal{\widetilde{{OP}}}_{n,r}]\otimes \det.
    \end{equation*}
\end{theorem}
\begin{proof}
    Since the set $\{e_{\llambda}\}_{\llambda\vdash_r n}$ spans $\Lambda^{(r)}$, in view of Lemma~\ref{perplemma}, it suffices to show that 
    \begin{equation*}
        \dim_\CC  \antisymm\cdot SR_{n,r}= \dim_\CC \varepsilon_{\bm\lambda}\cdot \left(\CC[\mathcal{\widetilde{{OP}}}_{n,r}]\otimes \det \right)
    \end{equation*} for any  $\llambda\vdash_r n$, which was established in Lemma~\ref{lem: dim equal}.
\end{proof}

Theorem~\ref{thm: module iso} refines to give the fermionic-graded module structure of $SR_{n,r}$. Let $\mathcal{\widetilde{{OP}}}^{(k)}_{n,r}$ denote the set of $r$-ified ordered set partitions of $[n]$ that have exactly $k$ non-zero blocks.   We write $(SR_{n,r})_{*,n-k} := \bigoplus_{i \geq 0} (SR_{n,r})_{i,n-k}$ for the fermionic-degree $n-k$ part of $SR_{n,r}$.
\begin{theorem}
\label{thm:fermionic-piece-iso}
    We have an isomorphism of $\ZZ_r\wr\symm_n$-modules $$(SR_{n,r})_{*,n-k}\cong_{\ZZ_r \wr \symm_n} \CC[\mathcal{\widetilde{{OP}}}^{(k)}_{n,r}]\otimes \det.$$
\end{theorem}
\begin{proof}
    As in the proof of Theorem~\ref{thm: module iso}, it suffices to establish the dimensional equality $$\dim_\CC \antisymm\cdot (SR_{n,r})_{*,n-k}=\dim_\CC\antisymm\cdot\left(\CC[\mathcal{\widetilde{{OP}}}^{(k)}_{n,r}]\otimes \det\right)$$holds for any $r$-partition $\llambda\vdash_r n$.

Borrowing terminology from the proof Lemma~\ref{lem: monomial count}, we see that $\antisymm\cdot\left(\CC[\mathcal{\widetilde{{OP}}}^{(k)}_{n,r}]\otimes \det\right)$ has a basis in bijection with $\llambda$-surviving $r$-ified ordered set partitions with exactly $k$ non-zero parts. The procedure in the proof of Lemma~\ref{lem: monomial count} constructs such a set partition, adding exactly $\lambda^{(i)}_j-\gamma^{(i)}_j$ non-zero parts at each step, so the resulting set partition has exactly $$(\lambda^{(1)}_1+\cdots+\lambda^{(r)}_{k_r})-(\gamma^{(1)}_1+\cdots+\gamma^{(r)}_{k_r})=n-|J(\llambda,\ggamma)|$$parts. So the number of all such $r$-ified ordered set partitions with exactly $k$ non-zero partitions is obtained by summing the expression in \eqref{eq: total count} over all $\ggamma$ satisfying $|J(\llambda,\ggamma)|=n-k$. By the same proof, this also counts the number of elements in $$\bigsqcup_{\substack{\ggamma\le \llambda\\|J(\llambda,\ggamma)|=n-k}}\AAA_{n,r}(\llambda,\ggamma)\cdot\theta_{J(\llambda,\ggamma)}.$$But applying $\antisymm$ to this set generates a spanning set of $\antisymm\cdot (SR_{n,r})_{*,n-k}$, which means 
\begin{equation}\label{eq: graded bound}
    \dim_\CC \antisymm\cdot (SR_{n,r})_{*,n-k}\le \dim_\CC \antisymm\cdot\left(\CC[\mathcal{\widetilde{{OP}}}^{(k)}_{n,r}]\otimes \det\right).
\end{equation}
Summing \eqref{eq: graded bound} over all $\llambda\vdash_r n$ yields the two sides of \eqref{eq: dim equal}, and so equality must hold in \eqref{eq: graded bound} for each $\llambda \vdash_r n$, as desired.
\end{proof}

For $r=1$, the colors of the zero block in an $r$-ified ordered set partition are restricted to the empty set, which forces the zero block to be empty. All the colors in the non-zero blocks must be $0$. Thus for the symmetric group, there is a natural bijection between $\widetilde{\mathcal{OP}_{n,1}}$ and the $\mathcal{OP}_n$, the family of  ordered set partitions of $[n]$, which commutes with the action of $\symm_n$. The $r =1$ case of Theorem~\ref{thm: module iso} therefore recovers one of the main results of \cite{MRW}:
\begin{equation}
    SR_{n,1} \cong_{\symm_n} \CC[\OP_n] \otimes \det.
\end{equation}

For $r=2$, the wreath product $\ZZ_2\wr\symm_n$ is the hyperoctahedral group $B_n$ of signed permutations of $[n]$. 
The {\em  $B_n$-reflection arrangement} is the hyperplane arrangement $\mathcal{A}(B_n)$ in $\RR^n$ with hyperplanes
\[
\mathcal{A}(B_n) = \{ x_i \pm x_j = 0 \,:\, 1 \leq i < j \leq n \} \cup \{x_i = 0 \,:\, 1 \leq i \leq n \}.
\]
The hyperplanes of $\mathcal{A}(B_n)$ partition $\RR^n$ into {\em faces} of various dimensions. We write $\OP^B_n$ for the set of these faces; this set is the {\em type B Coxeter complex}. The collection $\OP^B_n$ carries a permutation action of $B_n$ given by reflections in the hyperplanes of $\mathcal{A}(B_n)$. 

Given an element $\sigma = (Z,\BBB) \in \widetilde{\OP}_{n,2}$ with $B = (B_1 \mid \cdots \mid B_k)$, the elements of $Z$ must be colored 0 and the elements of $B_1, \dots, B_k$ are colored 0 or 1. Given $\sigma = (Z,\BBB) \in \widetilde{\OP}_{n,2}$ with $\BBB = (B_1 \mid \cdots \mid B_k)$, we obtain a face $F(\sigma) \in \OP^B_n$ characterized by the following coordinate (in)equalities.
\begin{itemize}
    \item If $i \in Z$ lies in the zero block of $\sigma$ then $x_i = 0$ on $F(\sigma)$.
    \item If $1 \leq i \neq i' \leq n$ and $i^c \sim j^d$ are blockmates in $\BBB$ with colors $c, d \in \{0,1\}$ then $x_{i^c} = x_{i^d} > 0$ on $F(\sigma)$. Here we interpret $x_{i^0} := x_i$ and $x_{i^1} := -x_i$.
    \item If $i^c$ and $j^d$ are colored indices appearing in different blocks of $\BBB$ such that $i^c$'s block is before $j^d$'s block, then $x_{i^c} < x_{j^d}$  on $F(\sigma)$.
\end{itemize}
For example, if $n = 8$ and $\sigma = (Z, \BBB)$ where $Z = \{2^0,5^0,6^0\}$ and $\BBB = ( 3 ^0 7^1 \mid 4^1 \mid 1^0 8^1 9^0)$ then $F(\sigma)$ is given by 
\[
0 = x_2 = x_5 = x_6 < x_3 = -x_7 < -x_4 < x_1 = -x_8 = x_9.
\]
It is not difficult to see that $\sigma \mapsto F(\sigma)$ is a $B_n$-equivariant bijection from $\widetilde{\OP}_{n,2}$ to $\OP^B_n$. The following result was conjectured by the second author \cite[(7.2)]{Bhattacharya}.

\begin{corollary}
\label{cor:B-identification}
    We have an isomorphism of ungraded $B_n$-modules $$SR_{n,2}\cong_{B_n}\CC[\mathcal{OP}^B_n]\otimes \det.$$
\end{corollary}

Let $G$ be a real reflection group and let $\OP_G$ be the set of faces of the $G$-reflection arrangement. Then $\CC[\OP_G]$ is naturally a permutation $G$-module. As explained in the introduction,
Murai, Rhoades, and Wilson conjectured \cite[Conj. 7.1]{MRW} that there is a $G$-equivariant surjection 
\begin{equation}
    \varphi: SR_G \twoheadrightarrow \CC[\OP_G] \otimes \det.
\end{equation}
In \cite{MRW} it is proven that such an epimorphism $\varphi$ exists in type A, where it is in fact an isomorphism.
Corollary~\ref{cor:B-identification} proves \cite[Conj. 7.1]{MRW} in type BC, again giving an isomorphism in this case. It is observed in \cite{MRW} that an epimorphism exists in type F$_4$, but fails to be an isomorphism due to a single additional irreducible character in two consecutive exterior degrees of $SR_G$.



\section{Conclusion}
\label{sec:Conclusion}

In this paper we extended the results on the symmetric group superspace coinvariant ring of \cite{RW, ACKMR, MRW} to the wreath product groups $\ZZ_r \wr \symm_n = G(r,1,n)$. It is natural to consider the larger family of complex reflection groups $G(r,p,n)$ where $p \mid r$. These are the groups of monomial matrices
\begin{equation}
    G(r,p,n) := \left\{
        g \in G(r,1,n) \,:\, \begin{array}{c}\text{the product of the nonzero entries of $g$} \\ \text{is a $(r/p)^{th}$ root-of-unity}\end{array}
    \right\}.
\end{equation}  
For example, the group $G(2,2,n)$ is the Weyl group of type D. It was observed by Swanson and Wallach \cite{SW} that the inverse system $SI_G^\perp$ for $G$ of type D is not generated by its $\det$-isotypic component under the $\odot$-action. The Operator Theorem~\ref{thm:operator} would therefore need to take a substantially different form for $G(r,p,n)$ and many of our methods would require significant changes.

Chan and Rhoades considered  \cite{CR} an alternative combinatorial model of ordered set partitions associated to $\ZZ_r \wr \symm_n$. Define a {\em $\ZZ_r \wr \symm_n$-face} to be an element $\sigma = (Z,\BBB) \in \widetilde{\OP}_{n,r}$ such that every element of the zero block $Z$ gets the color 0. Write $\FFF_{n,r} \subseteq \widetilde{\OP}_{n,r}$ for the family of $\ZZ_r \wr \symm_n$-faces. The monomial action of $\ZZ_r \wr \symm_n$ on $\CC[\widetilde{\OP}_{n,r}]$ preserves the subspace $\CC[\FFF_{n,r}]$. 

Generalizing work of Haglund--Rhoades--Shimozono \cite{HRS} in the case of $\symm_n$, Chan and Rhoades \cite{CR} defined a singly-graded quotient of the polynomial ring $\CC[x_1,\dots,x_n]$ whose algebraic properties are governed by the combinatorics of $\ZZ_r \wr \symm_n$-faces. We give a superspace quotient whose combinatorics is conjecturally related to $\FFF_{n,r}$.

\begin{conjecture}
\label{conj:face-conjecture}
    Assume $r \geq 2$ and let $SI'_{n,r} \subseteq \Omega$ be the superspace ideal generated by $\dots$
    \begin{itemize}
        \item the elements $x_1^{ir} + \cdots + x_n^{ir}$ for $i = 1,2,\dots,n$, and
        \item the element $x_1^{(i-1)r + 1} \theta_1 + \cdots + x_n^{(i-1)r + 1} \theta_n$ for $i = 1,2,\dots,n$.
    \end{itemize}
    Then $SI'_{n,r}$ is stable under the action of $\ZZ_r \wr \symm_n$ so that $\Omega/SI'_{n,r}$ is a bigraded $\ZZ_r \wr \symm_n$-module. There exists a module isomorphism
    \[
    \Omega/SI'_{n,r} \cong_{\ZZ_r \wr \symm_n} \CC[\FFF_{n,r}] \otimes \det.
    \]
\end{conjecture}

 Theorem~\ref{thm: module iso} proves Conjecture~\ref{conj:face-conjecture} when $r=2$. For $r > 2$, although the ideal $SI'_{n,r}$ has a similar-looking generating set to the ideal $SI_{n,r}$, the methods of \cite{MRW} and this paper do not seem to easily apply to Conjecture~\ref{conj:face-conjecture}. 

The ideal $SI_{n,r}'$ appearing in Conjecture~\ref{conj:face-conjecture} has a more intrinsic definition. Let $G \subseteq GL(V)$ be a complex reflection group and let $\Omega' := \CC[V] \otimes \wedge(V)$ be the modified version of the differential forms ring $\Omega = \CC[V] \otimes \wedge(V^*)$ where the exterior algebra is generated by $V$ rather than $V^*$. The natural action of $G$ on $V$ induces an action on $\Omega'$, and one can consider the ideal $SI'_G \subseteq \Omega'$ generated by $G$-invariants with vanishing constant term.

We claim that when $G = G(r,1,n)$ the ideal $SI'_G \subseteq \Omega'$ is given by $SI'_G = SI'_{n,r}$. The key is the following analogue of Theorem~\ref{thm:solomon} modified to suit this new $G(r,1,n)$-action.

\begin{proposition}
\label{prop:modified-generation}
    Let $n \geq 1$ and $r \geq 2$. For any $1 \leq i \leq n$, let $p_{ir} \in \Omega'$ denote the usual power sum $p_{ir} := x_1^{ir}+\cdots+x_n^{ir}$. Define $g_{i} \in \Omega'$ by $$g_i:=x_1^{(i-1)r + 1} \theta_1 + \cdots + x_n^{(i-1)r + 1} \theta_n.$$Then any $\ZZ_r \wr \symm_n$-invariant element $\Phi\in (\Omega')^{\ZZ_r \wr \symm_n}$ of homogeneous fermionic degree $k$  can be written in the form $$\Phi=\sum_{1\le i_1<\ldots<i_k\le n}f_{i_1,\ldots, i_k}\cdot g_{i_1}g_{i_2}\cdots g_{i_k}$$where $f_{i_1,\ldots, i_k}\in S^{\ZZ_r \wr \symm_n}$ is a $\ZZ_r\wr\symm_n$-invariant polynomial.
\end{proposition}

The proof of Proposition~\ref{prop:modified-generation} is a variant of Solomon's classical argument proving Theorem~\ref{thm:solomon} \cite{Solomon}.

\begin{proof}
    It is clear that $p_{ir}$'s and $g_i$'s are all invariant under the action of $\ZZ_r\wr\symm_n$ on $\Omega'$. We begin by noting two important properties of $g_i$'s:
    \begin{itemize}
        \item The $g_i$'s anticommute: we have $g_ig_j=-g_jg_i$ for any $1 \leq i, j \leq n$. In particular, we have $g_i^2=0$.
        \item We have $$g_1g_2\cdots g_n=\det\left(x_j^{r(i-1)+1}\right)_{1\le i,j\le n}\theta_1\theta_2\cdots\theta_n=\pm f_{[n],r}\cdot\theta_1\theta_2\cdots\theta_n$$where $f_{[n],r}$ is as in \eqref{eq: f[n]r}.
    \end{itemize}
    Let $K = \CC(V) \cong \CC(x_1,\dots,x_n)$ denote the field of fractions of $S = \CC[V] \cong \CC(x_1,\dots,x_n)$. We claim that the set $$\{g_{i_1}\cdots g_{i_k}:1\le i_1<\cdots<i_k\le n\}$$ is linearly independent over $K$. Indeed, suppose
    \begin{equation}
    \label{eq:fake-linear-dependence}
    \sum_{i_1<\ldots<i_k}f_{i_1,\ldots, i_k}\cdot g_{i_1}g_{i_2}\cdots g_{i_k}=0
    \end{equation}
    where $f_{i_1,\ldots, i_k}\in K$. Fix an index set $\{i_1<\ldots<i_k\}$ and multiply both sides of \eqref{eq:fake-linear-dependence} by $g_{j_1}\cdots g_{j_{n-k}}$ where $\{j_1<\ldots<j_{n-k}\}=[n]-\{i_1<\ldots<i_k\}$. Since the $g_i$'s anticommute, this yields 
    \begin{equation}f_{i_1,\ldots, {i_k}}\cdot g_1\cdots g_n=f_{i_1,\ldots, {i_k}}\cdot f_{[n],r}\cdot\theta_1\cdots\theta_n=0,
    \end{equation}
    implying $f_{i_1,\ldots, i_k}=0$.

    The $\binom{n}{k}$ elements of the form $\theta_{i_1}\cdots\theta_{i_k}$ span the $K$-vector space $K \otimes_\CC \wedge^k V$.
    It follows the $\binom{n}{k}$ $K$-linearly independent elements of the form $g_{i_1}\cdots g_{i_k}$ must $K$-span the $k$-th fermionic graded piece of $K \otimes_\CC \wedge^k V$ as well. Thus we have 
    \begin{equation}
    \label{eq:phi-into-K}
    \Phi=\sum_{1\le i_1<\ldots<i_k\le n}\varphi_{i_1,\ldots, i_k}\cdot g_{i_1}g_{i_2}\cdots g_{i_k}\end{equation} 
    for some $\varphi_{i_1,\ldots, i_k}\in K$. Applying the averaging group algebra element \[\eta_{n,r}:=\frac{1}{\left|\ZZ_r\wr\symm_n\right|}\sum_{g\in \ZZ_r\wr\symm_n}g \in \CC[\ZZ_r \wr \symm_n]\] on both sides of \eqref{eq:phi-into-K} yields \begin{equation}
    \label{eq:phi-into-K-symmetrized}
    \Phi=\sum_{1\le i_1<\ldots<i_k\le n}f_{i_1,\ldots, i_k}\cdot g_{i_1}g_{i_2}\cdots g_{i_k}\end{equation}
    where $f_{i_1,\ldots, i_k}:=\eta_{n,r}\cdot \varphi_{i_1,\ldots, i_k} \in K^{\ZZ_r \wr \symm_n}$ is a $\ZZ_r\wr\symm_n$-invariant element of $K$. It remains to show $f_{i_1,\ldots, i_k}$ is in fact a polynomial in $S$.
    To this end, multiply both sides of \eqref{eq:phi-into-K-symmetrized} by $g_{j_1,\ldots, j_{n-k}}$, where $\{j_1<\ldots<j_{n-k}\}=[n]-\{i_1<\ldots<i_k\}$. This yields an equation of the form \begin{equation}\phi\cdot \theta_1\ldots\theta_n=f_{i_1,\ldots, i_k}\cdot g_1\cdots g_n=f_{i_1,\ldots, i_k}\cdot f_{[n],r}\cdot\theta_1\cdots\theta_n\end{equation}
    where $\phi\in S$. This implies $\phi/f_{[n],r}=(x_1\cdots x_n)^{r-2}\phi/\delta_{n,r}$ is $\ZZ_r\wr\symm_n$-invariant, and thus the polynomial $\psi:=(x_1\cdots x_n)^{r-2}\phi$ satisfies $g\cdot \psi=\det(g)\cdot\psi$ for all $\psi\in \ZZ_r\wr\symm_n$. (Since we are assuming $r \geq 2$, the element $\psi$ is in fact a polynomial.) By \cite[\S 3, Lemma]{Solomon}, $\psi$ must be divisible by $\delta_{n,r}$. Therefore $\phi/f_{[n],r}=f_{i_1,\ldots, {i_k}}\in S$ as claimed. 
\end{proof}

It may be interesting to understand the combinatorics of $\Omega'/SI'_G$ in for more general reflection groups $G$. Conjecture~\ref{conj:face-conjecture} and the work of Sagan and Swanson \cite{SS} suggest that the bigraded $\Omega'/SI'_G$ could be tied to the coexponents of $G$.

\section*{Acknowledgements}

The authors are grateful to John Lentfer, Bruce Sagan, and Josh Swanson for helpful conversations. B. Rhoades was partially supported by NSF Grant DMS-2246846.

\end{document}